\documentclass[a4paper, english, 10pt, hidelinks]{article}

\usepackage[utf8]{inputenc}			
\usepackage[T1]{fontenc} 			
\usepackage[english]{babel}
\usepackage{lmodern}
\usepackage[numbers]{natbib}    	

\usepackage{graphicx}				
\usepackage{subcaption}				
\usepackage{hyperref}				

\usepackage{geometry}				
\usepackage{listings}				
\usepackage{mathcomp} 				
\usepackage{tabularx}				
\usepackage{amsmath} 				
\usepackage{amssymb}				
\usepackage{bbold}					
\usepackage{amsthm}					
\usepackage{varwidth}				

\usepackage{tikz}							
\usetikzlibrary{calc,intersections}			
\usepackage{pgfplots}
\pgfplotsset{compat=1.11}
\usepgfplotslibrary{fillbetween}
\usetikzlibrary{patterns}
\usetikzlibrary{arrows.meta}

\usepackage{bbm}							
\usepackage{multirow}					
\usepackage{hyperref}

\usepackage{ulem}

\newtheorem{satz}{Theorem}
\newtheorem{definition}[satz]{Definition}

\newtheorem{lemma}[satz]{Lemma}
\newtheorem{remark}[satz]{Remark}
\newtheorem{corollary}[satz]{Corollary}

\DeclareMathOperator\erf{erf}

\title{Probabilistic Constrained Optimization on Flow Networks}
\author{Michael Schuster\footnotemark[1], Elisa Strauch\footnotemark[2],
Martin Gugat\footnotemark[1], Jens Lang\footnotemark[2]}
\begin{document}

\maketitle






\footnotetext[1]{Friedrich-Alexander University Erlangen-Nürnberg (FAU), Department of Mathematics, Cauerstr. 11, 91058 Erlangen, Germany, \textit{martin.gugat@fau.de}, \textit{michi.schuster@fau.de} \\ Correspondence should be addressed to Michael Schuster, \textit{michi.schuster@fau.de} \\[-5pt]}
\footnotetext[2]{Technical University of Darmstadt, Department of Mathematics, Dolivostr. 15, 64293 Darmstadt, Germany, \textit{lang@mathematik.tu-darmstadt.de}, \textit{strauch@mathematik.tu-darmstadt.de}}



\begin{center}
\begin{minipage}{15cm}
\paragraph{Abstract.} \small Uncertainty often plays an important role in dynamic flow problems. In this paper, we consider both, a stationary and a dynamic flow model with uncertain boundary data on networks. We introduce two different ways how to compute the probability for random boundary data to be feasible, discussing their advantages and disadvantages. In this context, feasible means, that the flow corresponding to the random boundary data meets some box constraints at the network junctions. The first method is the spheric radial decomposition and the second method is a kernel density estimation. \\
In both settings, we consider certain optimization problems and we compute derivatives of the probabilistic constraint using the kernel density estimator. Moreover, we derive necessary optimality conditions for an approximated problem for the stationary and the dynamic case. \\
Throughout the paper, we use numerical examples to illustrate our results by comparing them with a classical Monte Carlo approach to compute the desired probability.
\end{minipage}
\end{center}

\vspace{.5cm}

\noindent \small\textbf{Key words:} Stochastic Optimization, Probabilistic Constraints, Uncertain Boundary Data, Spheric Radial Decomposition, Kernel Density Estimator, Flow Networks, Gas Networks, Contamination of Water

\vspace{.5cm}


\large

\section{Introduction and motivation} \label{sec:introductionMotivation}

In this paper, we present a method which describes how to deal with uncertain loads in the context of flow networks. The modeling and simulation of flow networks like gas flow, water flow and the diffusion of harmful substances inside flow networks become more and more important. So in this paper, we analyze the gas flow through a pipeline network in a stationary setting and the contamination of water in a dynamic setting. The aim of this paper is to solve probabilistic constrained optimization problems and to derive necessary optimality conditions for them in the context of flow networks. \\

Gas transport resp. general flow problems have been a goal of many studies. In general, such a flow problem is modeled as a system of hyperbolic balance laws based on e.g. the isothermal Euler equations (for gas flow, see e.g. \cite{BandaHertyKlar2,BandaHertyKlar,  WeierstrassInstitut, GugatHanteHirschLeugering, GugatSchultzWintergerst, GugatSchuster}) or the shallow water equations (for water flow, see e.g. \cite{BastinCoronNovel, Coron, GugatLeugering,  GugatLeugeringGeorg, LeugeringGeorg}). The model in the stationary setting in this paper is based on the stationary isothermal Euler equations for modeling the gas flow through a pipeline network. In \cite{Gas-Modell:DomschkeHillerLangTischendorf2017, EvaluatingGasNetworkCapacities} one can find a great overview about the topic of gas transport, existing models, and network elements. The existence of a unique stationary state is shown in \cite{GugatHanteHirschLeugering}, the stationary states for real gas are analyzed in \cite{GugatSchultzWintergerst, GugatWintergerst}. The existence of solutions for the dynamic case have been analyzed in \cite{GugatUlbrich, GugatUlbrich2}. Optimal control problems in gas networks have been studied e.g. in \cite{BermudezGonzalezFernandez,ColomboGuerraHertySchleper,GugatHerty}. For the problem of contamination of water by harmful substances, we use a linear scalar hyperbolic balance law. This has also been analyzed in \cite{ FuegenschuhGoettlichHerty, GugatWater}. \\

An important aspect of this paper is that we consider random boundary data. In the context of gas transport, that means that the loads (i.e., the gas demand) are random. In the context of water contamination, that means that the contaminant injection is random. This leads to optimization problems with probabilistic constraints (see e.g. \cite{Prekopa, ShapiroDentechevaRuszczynski}). We also assume box constraints for the solution of the balance law at the network nodes and we define the set of feasible loads $M$ as all loads, for which the solution of the balance law meets these box constraints. Our aim is to compute the probability for a random load vector to be feasible, i.e., we want to compute the probability for a random vector to be in a certain set $M$. So we identify the load vector with some random vector $\xi$ on an appropriate probability space $(\Omega, \mathcal{A}, \mathbb{P})$ and we want to compute the probability
\begin{equation*}
	\mathbb{P}( \omega \in \Omega\ \vert\ \xi(\omega) \in M ),
\end{equation*}
which we write as
\begin{equation*}
	\mathbb{P}( \xi \in M ).
\end{equation*}
A direct approach to compute this probability is to integrate the probability density function of the balance law solution over the box constraints. However this density function may not be known. We use a \textit{kernel density estimator} (see e.g. \cite{Nadaraya,Parzen,  ScottTerrell}) to obtain an approximation of this function. In \cite{Duller, HaerdleWerwatzMuellerSperlich} one can get a great overview about the area of nonparametric statistics.
The authors in \cite{WeierstrassInstitut, GugatSchuster} use the spheric radial decomposition (see e.g. \cite{VanAckooijAleksovskaMunozZuniga, VanAckooijHenrion,  FarshbafShakerHenrionHoemberg, WeierstrassInstitut, GonzalezGradonHeitschHenrion}) for a similar gas flow problem to compute the desired probability.

\begin{satz} (\textbf{spheric radial decomposition, see \cite{WeierstrassInstitut}, Theorem 2)} \label{theorem:srd}
Let $ \xi \sim \mathcal{N}(0, R) $ be the $n$-dimensional standard Gaussian distribution with zero mean and positive definite correlation matrix $ R $. Then, for any Borel measurable subset $ M \subseteq \mathbb{R}^n $ it holds that
\begin{align} \label{eq:srd}
	\mathbb{P}(\xi \in M) = \int\limits_{\mathbb{S}^{n-1}} \mu_\chi \{ \hat{r} \geq 0 \vert \hat{r} L s \in M \} d \mu_\eta(s),
\end{align}
where $ \mathbb{S}^{n-1} $ is the $(n-1)$-dimensional sphere in $ \mathbb{R}^n $, $ \mu_\eta $ is the uniform distribution on $ \mathbb{S}^{n-1} $, $\mu_\chi $ denotes the $ \chi $-distribution with $ n $ degrees of freedom and $ L $ is such that $ R = L L^\top $ (e.g., Cholesky decomposition).
\end{satz}
This result can be applied to general Gaussian distributions easily: For $ \xi \sim \mathcal{N}(\mu, \Sigma) $, set $\xi^* = D^{-1} (\xi - \mu) \sim \mathcal{N}(0,R)$ with $D = \text{diag}\left( \sqrt{ \Sigma_{ii} } \right)$ and $R = D^{-1} \Sigma D^{-1}$. Then it follows $\mathbb{P}(\xi \in M) = \mathbb{P}(\xi^* \in D^{-1} (M - \mu))$. An algorithmic formulation of the spheric radial decomposition is given in \cite{WeierstrassInstitut, GugatSchuster}. \\

From now on we use SRD instead of spheric radial decomposition and KDE instead of kernel density estimator. The difference in both methods is shown in \hyperref[fig:KDEvsSRD]{\textit{Figure \ref*{fig:KDEvsSRD}}}.

\begin{figure}[htbp]
	\centering
	\begin{subfigure}[c]{5cm}
		\centering
		\includegraphics[width=5cm]{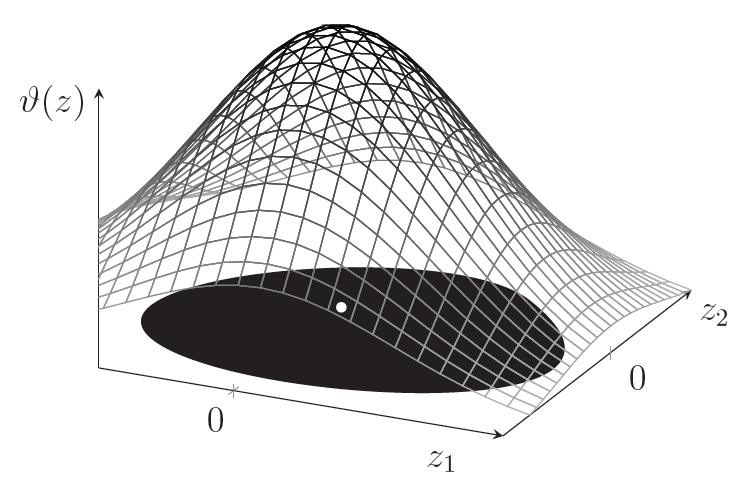}
		\caption{Direct Approach: Integrate the density function over a certain set}
	\end{subfigure} \hspace{1cm}
	\begin{subfigure}[c]{5cm}
		\centering
		\includegraphics[width=5cm]{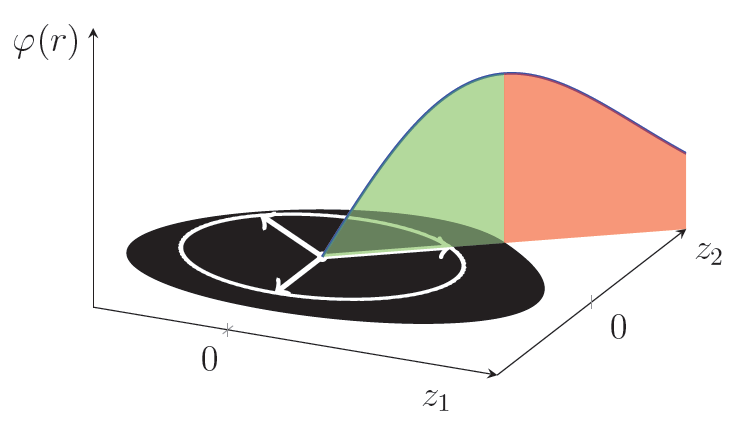}
		\caption{SRD: Integrate rays evaluated at the $\chi$-distribution over the unit sphere}
	\end{subfigure}
	\caption{Computing the probability for a random vector to be in a certain set: Direct approach vs. SRD}
	\label{fig:KDEvsSRD}
\end{figure}

An advantage of the SRD is that we exploit almost all information we can get from the model. Thus the only numerical error occurs while approximating the spherical integral. The big disadvantage of the SRD is that we need to know an analytical solution of our model which we cannot always guarantee. Therefore we introduce a KDE, which estimates the probability density function of a random variable by using a sampling set of the variable.
\begin{definition} (\textbf{kernel density estimator, see \cite{Gramacki}}) \label{def:KDE}
Let $y$ be a $n$-dimensional real-valued random variable with an absolutely continuous distribution and probability density function $\varrho$ with respect to the Lebesgue-measure. Moreover, let $\mathcal{Y} = \{ y^{S,1}, \cdots, y^{S,N} \}$ be an independent and identically distributed sample of $y$. Then, the kernel density estimator $\varrho_N:\mathbb{R}^n \rightarrow \mathbb{R}_{\geq 0}$ is defined as 
\begin{equation*}
	\varrho_{N}(z) = \frac{1}{N \det(H)^{\frac{1}{2}}} \sum_{i=1}^N K \left( H^{-\frac{1}{2}} \left( z - y^{S,i}) \right) \right),
\end{equation*}
with a symmetric positive definite bandwidth matrix $H  \in \mathbb{R}^{n \times n}$ and a \mbox{$n$-variate} density $K:\mathbb{R}^n \rightarrow \mathbb{R}_{\geq 0}$ called kernel.
\end{definition}
We apply the kernel density estimation to the balance law solution.
If an analytical solution of the model is not known, we can compute the solution numerically and use a sampling set of approximated solutions.
Then the desired probability can be computed by integrating the kernel density estimator over the box constraints. That means, we get an approximation error but we can analytically work with numerical solutions of our model. A KDE approach was used in \cite{CaillauCerfSassiTrelatZidani} for solving chance constrained optimal control problems with ODE constraints. But it was neither used in optimal control problems with PDE constraints and random boundary data nor in the context of continuous optimization with hyperbolic balance laws on networks. Throughout this paper, we illustrate the idea of the KDE in both, the stationary and the dynamic case, such that we can state necessary optimality conditions for optimization problems with probabilistic constraints. This paper is structured as follows: \\

In the next section, we consider stationary gas networks, similar to \cite{WeierstrassInstitut, GugatSchuster}. We first consider a simple model on a graph with only one edge to explain the ideas of the SRD and the KDE. We compare both results in a numerical computation with a classical Monte Carlo method (All numerical tests have been done with MATLAB$\textsuperscript{\textregistered}$, version $2015$a). Next we use both methods, the SRD and the KDE, to solve a model on a general tree-structured graph and again we compare both methods with a classical Monte Carlo method. Finally, we state necessary optimality conditions for probabilistic constrained optimization problems related to our stationary model. Last in this section we solve a probabilistic constrained optimization problem on a realistic gas network setting. \\

In \hyperref[sec:dynamic]{\textit{Section \ref*{sec:dynamic}}}, we consider a dynamic water network, in which the contaminant injection occurs at the boundaries. We consider a general linear hyperbolic balance law, which models the diffusion of harmful substances on a linear graph in order to discuss probabilistic constraints in the time dependent case and, whether the SRD can be expanded to this case. We also use the KDE for this model. Finally, we state necessary optimality conditions for probabilistic constrained optimization problems related to our dynamic model and solve a probabilistic constrained optimization problem for a realistic water contamination network setting.

\section{Gas networks in a stationary setting} \label{sec:stationary}

In this section, we consider stationary states in gas networks. As mentioned before, the model here is similar to the model in \cite{WeierstrassInstitut, GugatSchuster}. The main difference is that we fix an inlet pressure for our model, which the authors in \cite{WeierstrassInstitut} and \cite{GugatSchuster} did not. The network is described by a connected, directed, tree-structured graph $G = (\mathcal{V}, \mathcal{E})$ (i.e., the graph does not contain cycles) with the vertex set $\mathcal{V} = \{v_0, \cdots, v_n \}$ and the set of edges $\mathcal{E} = \{e_1, \cdots,e_n \} \subseteq \mathcal{V} \times \mathcal{V}$.
We assume that the graph has only one inflow node $v_0$ and that the other nodes are outflow nodes.
Let node $v_0$ be the root of the graph orientated away from the root.
The graph is numbered from the root $v_0$ using breadth-first search or depth-first search. Every edge $e \in \mathcal{E}$ represents a pipe with positive length $L^e$. For $x \in [0,L^e]$ we consider the stationary semi-linear isothermal Euler equations for horizontal pipes and ideal gases
\begin{equation} \label{eq:isothermalEulerStationary} \left\{ \quad \begin{aligned}
	q^e_x(x) &= 0, \\
	(c^e)^2 p^e_x(x) &= - \frac{\lambda^e}{2D^e} (R_S T)^2 \frac{q^e(x) \vert q^e(x) \vert}{p^e(x)}.
\end{aligned} \right. \end{equation}
With $p^e = p\big|_{e}$ we represent the restriction of the pressure defined over the network to a single edge $e$ and $p^e_x$ resp. $q^e_x$ is the derivative of $p^e$ resp. $q^e$ w.r.t. $x$.
Here, $q^e$ is the flow along edge $e$ and $c^e, \lambda^e, D^e \in \mathbb{R}_{>0}$ denote the sound speed, the friction coefficient and the pipe diameter. The parameters $R_S$ and $T$ denote the specific gas constant of natural gas and the (constant) temperature.
Note that $q^e$ is constant on every edge. With $q^e \geq 0$ we denote that gas flows along the orientation of edge $e$ and with $q^e \leq 0$ we denote that gas flows against the orientation of edge $e$.\\
%
We consider conservation of mass for the flow at the nodes (cf. \textit{Kirchhoff's first law}). Let $\mathcal{E}_+(v)$ resp. $\mathcal{E}_-(v)$ be the set of all outgoing resp. ingoing edges at node $v \in \mathcal{V}$. Let $b^v \in \mathbb{R}$ be the load at node $v \in \mathcal{V}$. With $b^v \geq 0$ we denote that gas leaves the network at node $v$ (exit node) and with $b^v \leq 0$ that gas enters the network at node $v$ (entry node).
The equation for mass conservation for every node $v \in \mathcal{V}$ is given by
\begin{equation} \label{eq:kirchhoff}
	\sum_{e \in \mathcal{E}_-(v)} q^e (L^e) = b^v + \sum_{e \in \mathcal{E}_+(v)} q^e (0).
\end{equation}
%
Let $p_i$ denote the pressure at the node $v_i$ for $i = 0, \cdots, n$.
We assume continuity in pressure at every node, i.e., for all $v_i \in \mathcal{V}$  it holds
\begin{equation} \label{eq:momentumConservation}
\begin{aligned}
	p^{e_1}(L^{e_1}) &= p_i   \qquad \forall e_1 \in \mathcal{E}_-(v_i),\\
	 p^{e_2}(0) &= p_i \qquad \forall e_2 \in \mathcal{E}_+(v_i) .
\end{aligned}
\end{equation}
Therefore the pressure $p_i$ is defined by the pressure $p^e(L^e)$ resp. $p^e(0)$
with ingoing resp. outgoing edge $e$ at node $v_i$.
We consider (positive) box constraints for the pressures at all outflow nodes $v_1, \cdots, v_n$, s.t.
\begin{equation}
	p_i \in \left[ p_i^{\min}, p_i^{\max} \right] \quad \forall i \in \{ 1,\cdots,n\}.
\end{equation}
In addition, we impose pressure $p_0$ at node $v_0$. So for the full graph, we consider the following model:
%
%

\begin{equation} \label{eq:stationaryModel} \left\{
\begin{tabular}{r c l l l}
	$q_x^e(x)$ 							& $=$	& $0$																				& & $\forall e \in \mathcal{E}$, \\
	$(c^e)^2 p_x^e(x)$ 					& $=$	& $-\frac{\lambda^e}{2 D^e} (R_s T)^2 \frac{q^e(x) \vert q^e(x) \vert}{p^e(x)}$ 	& & $\forall e \in \mathcal{E}$, \\[0.5cm]
	$p^e(0)$ 							& $=$	& $p_0$																				& & $\forall e \in \mathcal{E}_+(v_0)$, \\[0.5cm]
$\sum\limits_{e \in \mathcal{E}_-(v)} q^e(L^e)$	&  $=$  & $b^{v} + \sum\limits_{e \in \mathcal{E}_+(v)} q^e(0)$ 		
										& & $\forall v \in \mathcal{V}$ \\
$p^{e_1}(L^{e_1})$						&  $=$  & $p_i$ 		
										& & $\forall v_i \in \mathcal{V}, e_1 \in \mathcal{E}_-(v_i)$ \\
$p^{e_2}(0)$						&  $=$  & $p_i$ 		
										& & $\forall v_i \in \mathcal{V}, e_2 \in \mathcal{E}_+(v_i)$ \\[0.5cm]
$p_i$									&  $\in$  & $[p_i^{\min}, p_i^{\max}]$ 		
										& & $\forall i \in \{1, \cdots, n \}$.
\end{tabular} \right.
\end{equation}
\\
Our aim is now to find loads $b = (b^v)_{v \in \mathcal{V} \backslash \{ v_0\}} \in \mathbb{R}^n_{\geq 0}$  (corresponding to the outflow nodes $v_1, \cdots, v_n$), s.t. the model (\ref{eq:stationaryModel}) has a solution.
We do not consider the load $b^{v_0}$ at the inflow node $v_0$, because equation  \eqref{eq:kirchhoff} provides the relation $b^{v_0} = - \sum_{v \in \mathcal{V} \backslash \{ v_0\}} b^v$.
We define
\begin{equation*}
	M := \left\{\ b \in \mathbb{R}^n_{\geq 0}\ \vert\ \text{There exists a solution of (\ref{eq:stationaryModel}) } \right\}
\end{equation*}
as the set of feasible loads. To go one step further we assume that $b$ is random. This is motivated by reality. Because of the liberalization of the gas market\footnote[1]{\url{http://www.gesetze-im-internet.de/enwg_2005/index.html}}, the gas network is independent of the consumers and the gas companies. That means a gas network company must guarantee that the required gas can be transported through the network, but the company does not know the exact amount of gas a priori, so it can be seen as random. We assume
\begin{equation*}
	b \sim \mathcal{N}\left( \mu, \Sigma \right)
\end{equation*}
with mean value $\mu \in \mathbb{R}_+^n$ and positive definite covariance matrix $\Sigma \in \mathbb{R}^{n \times n}$ on an appropriate probability space $(\Omega, \mathcal{A}, \mathbb{P})$. Motivated by the application, we assume $\mu$ and $\Sigma$ are chosen s.t. the probability that $b$ takes positive values is almost $1$. This assumption would not be needed if the gas demand would be modeled e.g. by a truncated Gaussian distribution but we focus on the non truncated case in this work. In \cite{WeierstrassInstitut} and in \cite{EvaluatingGasNetworkCapacities}, \textit{Chapter 13}, the authors explain why a multivariate Gaussian distribution is a good choice for the random load vector. So  we want to know the probability that for a Gaussian distributed load vector $b$, the model (\ref{eq:stationaryModel}) has a solution, i.e.,
\begin{equation*}
	\mathbb{P}(b \in M).
\end{equation*}
In the next subsections we will use two different approaches for computing this probability. First, to explain the ideas of our approaches we consider a graph with only one edge and later we use the introduced tree-structured graph with one input.

\subsection{Uncertain load on a single edge} \label{sec:stationarySingleEdge}


As mentioned before, we will give two different ways to compute the probability for a random load value to be feasible. Here, we consider a single edge as graph. For the boundary conditions $p(0) = p_0$ and $q(L) = b$ model (\ref{eq:isothermalEulerStationary}) has an  analytical solution. Therefore our problem simplifies to verifying the following inequalities:
\begin{equation} \label{eq:isothermalEulerStationarySolution} \left\{ \quad \begin{aligned}
	&q(x) = b, \\
	&p(x) = \sqrt{p_0^2 - \frac{\lambda}{c^2 D} (R_S T)^2 \ q(x) \vert q(x) \vert x}, \\
	&p_1 = p(L) \in \left[ p^{\min}, p^{\max} \right].
\end{aligned} \right. \end{equation}
In fact, that means a random value $b \in \mathbb{R}_{\geq 0}$ is feasible, if the pressure at the end satisfies the box constraints, i.e.
\begin{equation*}
	b \in M \quad \Leftrightarrow \quad b \geq 0\ \text{ and }\ p_1 = \sqrt{ p_0^2 - \phi\ b\ \vert b \vert\ } \in \left[ p^{\min}, p^{\max} \right]
\end{equation*}
with $\phi = \frac{\lambda}{c^2 D} (R_S T)^2 L$. We can rewrite the box constraints for the pressure as inequalities. That means $p_1 \in \left[ p^{\min}, p^{\max} \right]$ iff
\begin{equation} \label{eq:oneEdgeInequalities} \begin{aligned}
	(p^{\min})^2\ &\leq\ p_0^2 - \phi\ b\ \vert b \vert\ ,\\
	(p^{\max})^2\ &\geq\ p_0^2 - \phi\ b\ \vert b \vert\ .
\end{aligned} \end{equation}
Now we can use the SRD (introduced in \hyperref[sec:introductionMotivation]{\textit{Section \ref*{sec:introductionMotivation}}}) in an algorithmic way. We consider $b \sim \mathcal{N}(\mu, \sigma^2)$ with mean value $\mu \in \mathbb{R_+}$ and standard deviation $\sigma \in \mathbb{R}_+$. Because the unit sphere in this example is given by $\{ -1, 1 \}$, we need not to sample $N$ points, we can use the unit sphere itself as sampling. For $s \in \mathbb{S}^0$, we set
\begin{equation*}
	b_s(\hat{r}) = \hat{r} \sigma s + \mu.
\end{equation*}
Note that in the multivariate Gaussian distribution, the covariance $\Sigma$ is given, which contains the variances of the random variables on its diagonal. In the one dimensional case, the standard deviation is given, which is a Cholesky decomposition of the variance. \\
To guarantee, that $b_s(\hat{r})$ is positive, we define the regular range
\begin{equation*}
	R_{s,\text{reg}} := \{ \hat{r} \geq 0\ \vert\ b_s(\hat{r}) \geq 0 \}.
\end{equation*}
Thus we have
\begin{equation*}
	M_s = \{ \hat{r} \in R_{s,\text{reg}}\ \vert\ b_s(\hat{r}) \in M \}\ =\ \{ \hat{r} \in R_{s,\text{reg}}\ \vert\ b_s(\hat{r}) \text{ satisfies } (\ref{eq:oneEdgeInequalities}) \}.
\end{equation*}
We insert $b_s(\hat{r})$ in the inequalities in (\ref{eq:oneEdgeInequalities}). Since $b_s(\hat{r}) \geq 0$ on $R_{s,\text{reg}}$, we can write $b_s^2$ instead of $b_s \vert b_s \vert$. Thus we have quadratic inequalities in $r$:
\begin{equation*} \begin{aligned}
	&\hat{r}^2 \left( \sigma^2 s^2 \phi  \right) + \hat{r} \left( 2 \sigma s \mu \phi \right) + \left( \mu^2 \phi  + (p^{\min})^2 - p_0^2 \right) \leq 0, \\
	&\hat{r}^2 \left( \sigma^2 s^2 \phi \right) + \hat{r} \left( 2 \sigma s \mu \phi \right) + \left( \mu^2 \phi + (p^{\max})^2 - p_0^2 \right) \geq 0.
\end{aligned} \end{equation*}
So we can write the set $M_s$ as a union of $\kappa \in \mathbb{N}$ disjoint intervals, i.e.
\begin{equation*}
	M_s = \bigcup_{j=1}^{\kappa} I_{s,j}
\end{equation*}
with intervals $I_{s,j} = [ \underline{a}_{s,j}, \overline{a}_{s,j} ]$ and interval bounds $\underline{a}_{s,j}, \overline{a}_{s,j} \in \mathbb{R}$, $\underline{a}_{s,j} \leq \overline{a}_{s,j}$ ($j = 1, \cdots, \kappa$). Since the unit sphere contains only two values and
\begin{equation} \begin{aligned} \label{eq:computeProbabilityOneEdge}
	\mathbb{P}(b \in M)\ &=\ \frac{1}{2} \sum_{s \in \{-1,1\}} \mu_{\chi} (M_s)\ =\ \frac{1}{2} \sum_{s \in \{-1,1\}} \sum_{j=1}^{\kappa} \mu_{\chi}(I_{s,j}) \\
	&=\ \frac{1}{2} \sum_{s \in \{-1,1\}} \sum_{j=1}^{\kappa} \mathcal{F}_{\chi} \left( \overline{a}_{s,j} \right) - \mathcal{F}_{\chi} \left( \underline{a}_{s,j} \right),
\end{aligned} \end{equation}
where $\mathcal{F}_{\chi}(\cdot)$ is the cumulative distribution of the $\chi$-distribution, we can compute the probability for a random vector to be feasible. As mentioned before, the SRD gives us an efficient way to compute this probability, but we need to know the analytical solution of our system (\ref{eq:stationaryModel}). \\
Another way to compute the probability for a random load vector to be feasible is the \textit{kernel density estimator}.
It is more general and does not require the analytical solution of our model.
We consider the stochastic equation corresponding to (\ref{eq:isothermalEulerStationary}) with random boundary condition $b$ and coupling conditions (\ref{eq:kirchhoff}), (\ref{eq:momentumConservation}) which has also a solution $\mathbb{P}$-almost surely. Hence the pressure at node $v_1$ is a random variable which we denote with $p_1$.
The probability that a random load is feasible is equal to the probability that the pressure $p_1$ is in the interval $\left[ p^{\min}, p^{\max} \right]$.
We assume that the variance of $p_1$ is positive and that its distribution of the pressure $p_1$ is absolutely continuous with probability density function $\varrho_p$. Now, the probability $\mathbb{P}(p_1 \in \left[ p^{\min}, p^{\max} \right])$ can be computed by integrating the probability density function $\varrho_p$ over the pressure bound, so we get
\begin{equation} 
\mathbb{P}(b \in M ) = \mathbb{P}(p_1 \in [ p^{\min}, p^{\max} ]) =   \int_{p^{\min}}^{p^{\max}}  \varrho_p(z) ~ dz .
\end{equation}
If the exact probability density function $\varrho_p$ is not known, we can approximate the function by a kernel density estimator. Then we integrate this estimator over the interval $[ p^{\min}, p^{\max} ]$ to get an approximation of the probability $ {\mathbb{P}(b \in M) }$.\\
As before, we consider
\begin{equation*}
	b \sim \mathcal{N}(\mu, \sigma^2)
\end{equation*}
with mean value $\mu \in \mathbb{R_+}$ and standard deviation $\sigma \in \mathbb{R}_{+}$. \\
Let $\mathcal{B} = \{b^{\mathcal{S},1}, \cdots, b^{\mathcal{S},N}\} \subseteq \mathbb{R}_{\geq 0} $ be independent and identically distributed samples of the random variable $b$. Let $\mathcal{P}_{\mathcal{B}} = \{p_1(b^{\mathcal{S},1}), \cdots, p_1(b^{\mathcal{S},N}) \} \subseteq \mathbb{R}$ be the pressures at the end of the edge for the different loads $b^{\mathcal{S},i} \in \mathcal{B}$ ($i = 1, \cdots, N$), which are also independent and identically distributed. We use the KDE in one dimension with bandwidth $h \in \mathbb{R_+}$ and with a Gaussian kernel (see e.g. \cite{Duller, ScottTerrell}) to get an approximation of the density function $\rho_p$:
\begin{equation} \label{eq:KDE1d}
	\varrho_{p,N}(z) = \frac{1}{N h} \sum_{i = 1}^N \frac{1}{\sqrt{2 \pi}} \exp\left( - \frac{1}{2} \left( \frac{z - p_{1}(b^{\mathcal{S},i})}{h} \right)^2 \right).
\end{equation}
\begin{remark} \label{remark:bandwidth}
The choice of the bandwidth $h$ is a separate topic. We refer to \cite{Duller}, \textit{Chapter 8} and \cite{HaerdleWerwatzMuellerSperlich}, \textit{Chapter 3} for studies about optimal bandwidths for KDEs. Here we use the heuristic formula for the bandwidth given by
\begin{equation} \label{eq:bandwidth}
	h = 1.06 \frac{\sigma_N}{\sqrt[5]{N}},
\end{equation}
where $\sigma_N$ is the standard deviation of the sampling and $N$ is the number of samples. The idea is, that we compute the bandwidth depending on the variance of the sampling. This is stated and explained e.g. in \cite{Gramacki}, \textit{Chapter 4.2} and in \cite{Turlach}.
\end{remark}
%
%
%
%
For the previous bandwidth it holds $(h + (Nh)^{-1}) \rightarrow 0$ $\mathbb{P}$-almost surely for ${ N \rightarrow \infty }$.
Therefore, \cite{DevroyeGyorfi} (\textit{Chapter 6, Theorem 1}) provides the L$^1$-convergence of the estimator:
\begin{equation} \label{eq:Konv1D}
	\Vert \varrho_p - \varrho_{p,N} \Vert_{L^1}  \xrightarrow{N \rightarrow \infty} 0 \quad \mathbb{P}\text{-almost surely}.
\end{equation}
Thus, for an appropriate choice of the bandwidth $h$, we can use the KDE as an approximation for the exact probability density function of the pressure.
From \textit{Scheff\'e}'\textit{s lemma} (see \cite{DevroyeGyorfi}), it follows
\begin{equation} \label{Konv2D}
\big | \int_{p^{\min}}^{p^{\max}} \varrho_p(z) ~ dz -\int_{p^{\min}}^{p^{\max}} \varrho_{p,N} (z) ~ dz \big | \leq \frac{1}{2} \Vert \varrho_p - \varrho_{p,N} \Vert_{L^1} \xrightarrow{N \rightarrow \infty} 0 \quad \mathbb{P}\text{-almost surely}.
\end{equation}
 So with (\ref{eq:KDE1d}) we can approximate the probability for a random load vector to be feasible as follows:
\begin{equation} \label{eq:probabilityKDE}
	\mathbb{P}(b \in M) = \mathbb{P}(p_1 \in [ p^{\min}, p^{\max} ]) \approx   \int_{p^{\min}}^{p^{\max}} \varrho_{p,N}(z) dz =: \mathbb{P}_N(b \in M).
\end{equation}
In this example, we can compute the sampling set $\mathcal{P}_{\mathcal{B}}$ analytically, because the analytical solution of model \eqref{eq:isothermalEulerStationary} is known.  If this is not the case, e.g., for more complex systems like the stationary full Euler equations or Navier-Stokes equations (see \cite{Gas-Modell:DomschkeHillerLangTischendorf2017}), one can use numerical methods to solve the PDE and to get an approximated sampling set $\mathcal{P}_{\mathcal{B}}$.

\paragraph*{Example 1:}\ We give an example to illustrate that the results of the KDE and the SRD are similar. Therefore we use the values (without units) from \hyperref[tab:oneEdgeValues]{\textit{Table \ref*{tab:oneEdgeValues}}}.
\begin{table}  [h]
	\centering
	\begin{tabular}{| c | c | c | c | c | c | c |}
			\hline
			$p_0$ 		& $p^{\min}$	& $p^{\max}$	& $\mu$		& $\sigma$	& $\phi$ \\	
			\hline
			$60$		& $40$ 			& $60$ 			& $4$		& $0.5$ 	
			& $100$	\\		
			\hline
	\end{tabular}
	\caption{Values for the example with one edge}
	\label{tab:oneEdgeValues}
\end{table}
With the inequalities (\ref{eq:oneEdgeInequalities}) and the values given in \hyperref[tab:oneEdgeValues]{\textit{Table \ref*{tab:oneEdgeValues}}} we can see that $b$ is feasible iff $b \in [0, \sqrt{20} ]$. For comparing the probability we use a classical Monte Carlo (MC) method, in which we check the percentage number of points inside $[0, \sqrt{20}]$. For both, the MC method and the KDE approach, we use the same sampling of $5 \cdot 10^4$ points. The bandwidth for the KDE is given by (\ref{eq:bandwidth}). The result for $8$ tests is shown in \hyperref[tab:oneEdgeResults]{\textit{Table \ref*{tab:oneEdgeResults}}}.

\begin{table} [h]
	\centering
	\begin{tabular}{| c | c | c | c | c | c | c | c | c |}
		\hline
			& Test 1	& Test 2	& Test 3	& Test 4	& Test 5	& Test 6	& Test 7	& Test 8	\\
		\hline
		MC	& $82.99\%$	& $82.88\%$	& $82.83\%$	& $82.95\%$	& $83.09\%$	& $82.77\%$	& $82.83\%$	& $82.58\%$	\\
		\hline		
		KDE	& $82.83\%$ & $82.75\%$	& $82.69\%$	& $82.74\%$	& $82.91\%$	& $82.58\%$	& $82.70\%$	& $82.43\%$	\\
		\hline
		SRD & \multicolumn{8}{c |}{$82.75\%$} \\
		\hline
	\end{tabular}
	\caption{Results for the example with one edge}
	\label{tab:oneEdgeResults}
\end{table}

The sphere of the SRD in this example is finite ($\mathbb{S}^0 = \{-1,1\}$), so the SRD gives always the exact probability (except numerical errors due to quadrature) of $82.75\%$. Also, the probabilities computed by MC and the KDE are quite similar. The mean probability in MC resp. KDE is $82.87\%$ resp. $82.40\%$ and the variances are $0.0238$ resp. $0.0218$, which is very close to the (exact) result of the SRD. Further we provide confidence intervals with confidence level $95\%$ for the results stated in \hyperref[tab:oneEdgeResults]{\textit{Table \ref*{tab:oneEdgeResults}}}. The confidence level of $95\%$ is quite common in statistics. An introduction to confidence intervals can be found in \cite[Chapter 8]{Linde}. The confidence interval for the MC probabilities is $[82.74\%, 82.99\%]$ and the confidence interval for the KDE probability is $[82.58\%, 82.83\%]$. One can see that the (exact) probability (computed via SRD) is contained in both intervals. Thus one can see, that the KDE with the heuristic bandwidth (\ref{eq:bandwidth}) provides a suitable method to compute the desired probability. \\

A remaining question is how to choose the sample size s.t. the result is sufficiently accurate. In general it holds the larger the sample size the more accurate is the solution. A good sample size can be found by observing the sample error. In \cite{Roache} and \cite{Schwer} the authors suggest comparing the numerical result with two other numerical results for larger sample sizes. This mesh-to-mesh comparison is often used in numerics to determine the rate of convergence of the solution. With this approach an appropriate sample size can be chosen for the desired accuracy.

\subsection{Uncertain loads on tree-structured graphs} \label{sec:stationaryTree}

For this subsection, we consider a tree-structured graph (i.e., the graph does not contain cycles) $G = (\mathcal{V}, \mathcal{E})$ with one entry $v_0$, as introduced in the beginning of \hyperref[sec:stationary]{\textit{Section \ref*{sec:stationary}}}. Let the graph be numbered from the root $v_0$ by breadth-first search or depth-first search and let the model (\ref{eq:stationaryModel}) holds on every edge.

First we rewrite the system (\ref{eq:stationaryModel}) using the solution of the isothermal Euler equations. We use the incidence matrix $A^+ \in \mathbb{R}^{(n+1) \times n}$ of the graph. For an edge $e_\ell  \in \mathcal{E}$, which connects the nodes $v_i$ and $v_j$ starting from node $v_i$, we have
\begin{equation*}
	A^+_{k,\ell} = \begin{cases} \quad -1 & \text{ if } k = i, \\ \quad 1 & \text{ if } k = j, \\ \quad 0 & \text{ else}. \end{cases}
\end{equation*}
A formal definition of the incidence matrix can be found e.g., in \cite{WeierstrassInstitut, GugatSchuster, GugatSchultzSchuster}. We set $A \in \mathbb{R}^{n \times n}$ as $A^+$ without the first row (which corresponds to the root resp. the only entry node of the graph). The equation for mass conservation (\ref{eq:kirchhoff}) can equally be written as
\begin{equation} \label{eq:kirchhoffNew}
	A\ q = b,
\end{equation}
where $q_j$ is the (constant) flow on the edge $e_j$ and $b_i$ is the load at node $v_i$ ($i,j = 1, \cdots, n$). Due to the tree-structuredness of the graph, $A$ is a square matrix with full rank and thus invertible. Numbering the graph by breadth-first search or depth-first search implies that the matrix $A$ is upper triangular, just like its inverse. Further in \cite[Section 3.2]{WeierstrassInstitut}, the authors mention, that $A^{-1}_{i,j}$ is one if and only if the edge $e_j$ is on the (unique) path from the root to node $v_i$, otherwise $A^{-1}_{i,j}$ is zero. Motivated by (\ref{eq:isothermalEulerStationarySolution}) we define the function
\begin{equation*}
	g : \mathbb{R}^n \rightarrow \mathbb{R}^n,\quad g : b \mapsto (A^\top)^{-1} \Phi ~ \left( \left( A^{-1} b \right) \circ \vert A^{-1} b \vert \right), \\
\end{equation*}
where $\Phi \in \mathbb{R}^{n \times n}$ is a diagonal matrix with the values $\phi^e= \frac{\lambda^e}{(c^e)^2 D^e} (R_S T)^2 L^e $ ($e \in \mathcal{E}$) at its diagonal. The product on the right has to be understood componentwise. The $i$-th component of this function states the pressure loss from the root $v_0$ to node $v_i$. The term $(A^{-1} b)$ comes from the equation for mass conservation (\ref{eq:kirchhoffNew}) and contains the (constant) flows at the edges. With this function we get the following equivalence for feasible loads:

\begin{lemma} \label{lemma:inequalitiesStationary}
A load vector $b \in \mathbb{R}_{\geq 0}^n$ is feasible, i.e., there exists a solution of (\ref{eq:stationaryModel}), iff the following system of inequalities holds for all $k = 1, \cdots, n$:
\begin{equation*} \begin{aligned}
	p_0^2 &\leq (p_k^{\max})^2 + g_k(b), \\
	p_0^2 &\geq (p_k^{\min})^2 + g_k(b). \\
\end{aligned} \end{equation*}
\end{lemma}


\begin{proof}
The result follows from \cite[Corollary 1]{WeierstrassInstitut}. In their setting, the inlet pressure $p_0$ is not given explicitly, it is given inside a range $[p_0^{\min}, p_0^{\max}]$. Then \textit{Corollary 1} in \cite{WeierstrassInstitut} states, that a load vector $b \in \mathbb{R}_{\geq 0}^n$ is feasible, iff the following system of inequalities holds:
\begin{equation*} \begin{aligned}
	p_0^{\min} &\leq \min_{k=1, \cdots, n} \left[ p_k^{\max} + g_k(b) \right] \\
	p_0^{\max} &\geq \max_{k=1, \cdots, n} \left[ p_k^{\min} + g_k(b) \right] \\
	\max_{k=1, \cdots, n} \left[ p_k^{\min} + g_k(b) \right] &\leq \min_{k=1, \cdots, n} \left[ p_k^{\max} + g_k(b) \right].
\end{aligned} \end{equation*}
In our setting the inlet pressure $p_0$ is explicitly given, which means $p_0^{\min} = p_0^{\max}$. Then the third inequality directly follows from the first one and the second one and thus, our lemma is a special case of \textit{Corollary 1} in \cite{WeierstrassInstitut}.
\end{proof}

\begin{lemma} \label{lemma:convexityStationary}
If $p_i^{\min} \leq p_0 \leq p_j^{\max}$ for all $i,j = 1, \cdots, n$, then the set of feasible loads $M$ is convex.
\end{lemma}

\begin{proof}
The proof is equal to the proof of \textit{Theorem 11} in \cite{GugatSchultzSchuster} for linear graphs, i.e., graphs without junctions. Here, the proof also works for tree-structured networks, because the pressure at the inflow node is explicitly given. Thus, the inequality, where convexity for tree-structured networks breaks in \cite{GugatSchultzSchuster}, is redundant here (see \hyperref[lemma:inequalitiesStationary]{\textit{Lemma \ref*{lemma:inequalitiesStationary}}}).
\end{proof}

Now we use the SRD to compute the probability for a random load vector to be feasible. For this setting, the SRD approach is explained in detail in \cite{WeierstrassInstitut, GugatSchuster}. Let $b \sim \mathcal{N}(\mu, \Sigma)$ with mean value $\mu \in \mathbb{R}_+^n$ and positive definite covariance matrix $\Sigma \in \mathbb{R}^{n \times n}$ for an appropriate probability space $(\Omega, \mathcal{A}, \mathbb{P})$ be given. The next steps are similar to the case of one edge. For a point $s$ of a sample $\mathcal{S} := \{s_1, \cdots, s_N \} \subseteq \mathbb{S}^{n-1}$ of $N \in \mathbb{N}$ uniformly distributed points at the unit sphere $\mathbb{S}^{n-1}$, we set
\begin{equation*}
	b_s(\hat{r}) := \hat{r} \mathcal{L} s + \mu,
\end{equation*}
where $\mathcal{L}$ is s.t. $\mathcal{L} \mathcal{L}^\top = \Sigma$. Because we only consider outflows, $b_s(\hat{r})$ must be positive. We define the regular range
\begin{equation*}
	R_{s, \text{reg}} := \{ \hat{r} \geq 0 \vert b_s(\hat{r}) \geq 0 \}.
\end{equation*}
From this, it follows that $b_s(\hat{r}) \in R_{s, \text{reg}}$ is feasible, iff $b_s(\hat{r})$ satisfies the inequalities in \hyperref[lemma:inequalitiesStationary]{\textit{Lemma \ref*{lemma:inequalitiesStationary}}}. The inequalities are quadratic in the variable $r$ and we can write the sets $M_s$ as unions of disjoint intervals. Thus, (\ref{eq:computeProbabilityOneEdge}) gives us the probability for a random load vector to be feasible. \\
As in the previous subsection, we consider the stochastic equation corresponding to \eqref{eq:isothermalEulerStationary} with random load $b$ and coupling conditions \eqref{eq:kirchhoff}, \eqref{eq:momentumConservation}. The $n$-dimensional random vector $p$ denotes the pressure at the nodes $v_1, \cdots, v_n$.
The probability that $b$ is feasible is equal to the probability that the pressure $p$ at the nodes is in the pressure bounds, which we denote by $\mathbb{P}(p \in P^{\max}_{\min})$ with $P^{\max}_{\min} := \bigotimes_{i=1}^n [p_i^{\min}, p_i^{\max}]$.
We assume that the covariance matrix of the random vector $p$ is positive definite and that its distribution is absolutely continuous with probability density function $\varrho_p$.
Now, it holds
\begin{equation} \label{eq:ProbMultiKDE}
\mathbb{P}(b \in M ) = \mathbb{P}(p \in P^{\max}_{\min}) = \int_{P^{\max}_{\min}}  \varrho_p(z) ~ dz .
\end{equation}
However the exact probability density function $\varrho_p$ is not known. But we can approximate the function using a multidimensional kernel density estimation. \\
Let $\mathcal{B} = \{b^{\mathcal{S},1}, \cdots, b^{\mathcal{S},N}\}$ be independent and identically distributed nonnegative samples of the random load vector $b \sim \mathcal{N}(\mu, \Sigma)$. Then, let $ \mathcal{P}_{\mathcal{B}} = \{ p(b^{\mathcal{S},1}), \cdots, p(^{\mathcal{S},N}) \} \subseteq \mathbb{R}^n  $ be the pressures at the nodes  $v_1, \cdots, v_n$ for the different loads $b^{S,i} \in \mathcal{B}$ $(i = 1,\cdots,N)$. These samples are also independent and identically distributed. Note that $p_j(b^{\mathcal{S},i})$ ($i \in \{1, \cdots, N\}$, $j \in \{1, \cdots, n\}$) is the pressure at node $v_j$ and the $j$-th component of the pressure vector $p(b^{\mathcal{S},i})$. We introduce the general $n$-dimensional multivariate kernel density estimator (see e.g. \cite{Gramacki}):
\begin{equation*}
	\varrho_{p,N}(z) = \frac{1}{N \det(H)^{\frac{1}{2}}} \sum_{i=1}^N K \left( H^{-\frac{1}{2}} \left( z - p(b^{\mathcal{S},i}) \right) \right)
\end{equation*}
with a symmetric positive definite bandwidth matrix $H  \in \mathbb{R}^{n \times n}$ and a \mbox{$n$-variate} density function $K:\mathbb{R}^n \rightarrow \mathbb{R}_{\geq 0}$ called kernel.
%
%
We choose the the standard multivariate normal density function as kernel. This kernel can be written as the product
$K(x) = \prod_{i=1}^n \mathcal{K}(x_i)$, where the univariate kernel ${ \mathcal{K}:\mathbb{R} \rightarrow \mathbb{R} }$ is the standard univariate normal density function.
Let $\sigma_{N,i}^2$ denote the positive sample variance of the $i$th variable. Moreover, let $V_N$ denote the diagonal matrix $V_N=\text{diag}(\sigma_{N,1}^2, \cdots, \sigma_{N,n}^2)$. As suggested in \cite{Gramacki}, we use the bandwidth matrix
\begin{equation} \label{eq:multiBandwidth}
	H = h_y^2 V_N \quad  \text{with} \quad h_y= \left( \frac{4}{(n+2)N}	 \right)^{\frac{1}{n+4}}  .
\end{equation}
This choice simplifies the estimator to the following form
\begin{equation}\label{eq:multiKDE}
\begin{aligned}
	\varrho_{p,N}(z) &= \frac{1}{N \prod_{j=1}^n h_y \sigma_{N,j}} \sum_{i=1}^N \prod_{j=1}^n \mathcal{K} \left( \frac{z_j - p_j(b^{\mathcal{S},i})}{h_y \sigma_{N,j}} \right) \\
	&= \frac{1}{N \prod_{j=1}^n h_y \sigma_{N,j}} \sum_{i=1}^N \prod_{j=1}^n \frac{1}{\sqrt{2 \pi}} \exp \left( -\frac{1}{2} \left( \frac{z_j - p_j(b^{\mathcal{S},i})}{h_y \sigma_{N,j}} \right)^2 \right).
\end{aligned}
\end{equation}
As mentioned in \cite{Scott, WandJones} such product kernels are recommended and adequate in practice. But, in some situations using only diagonal bandwidth matrices could be insufficient and then general full bandwidth matrices are needed, e.g., depending on the sample covariance matrix.  \\
In order to show the convergence of the multivariate KDE (\ref{eq:multiKDE}), we transform the random vector $p$ via $y = V_N^{-1/2} p$. Thus we get the transformed sampling set
\begin{equation*}
	\mathcal{Y}_\mathcal{B} = \{ y_1^S, \cdots, y_N^S \} := \{ V_N^{-1/2} p(b^{\mathcal{S},1}), \cdots, V_N^{-1/2} p(b^{\mathcal{S},N}) \} = V_N^{-\frac{1}{2}} \mathcal{P}_\mathcal{B}.
\end{equation*}
For the approximation of the probability density function of the transformed variable we use the multivariate KDE with the previous settings adapted to the data $\mathcal{Y}_B$. Thus we get the bandwidth matrix $H = h_y^2 I_{n \times n}$, because the transformed data have unit variance. This leads to a KDE with one single bandwidth $h_y$ given by
\begin{equation} \label{eq:multiKDEtrans}
	\varrho_{y,N}(z) = \frac{1}{N h_y^n} \sum_{i=1}^N \prod_{j=1}^n \mathcal{K} \left( \frac{z_j - y_{i,j}^S}{h_y} \right) .
\end{equation}
%
Due to $(h_y + (N h_y^n)^{-1} ) \xrightarrow{N \rightarrow \infty} 0$ $\mathbb{P}\text{-almost surely}$, we can apply \textit{Chapter 6, Theorem 1} in \cite{DevroyeGyorfi} to the estimator $\varrho_{y,N}$. Thus, it holds
\begin{equation*}
	\Vert \varrho_{y,N} - \varrho_y \Vert_{L^1} \xrightarrow{N \rightarrow \infty} 0 \quad \mathbb{P}\text{-almost surely},
\end{equation*}
where $\varrho_y$ is the exact probability density function of the random variable $y$. This density function is given by $\varrho_y(z) = \varrho_p ( V_N^{1/2} z ) \vert \det (V_N) \vert^{1/2}$ according to the transformation. There is also a similar relation between the estimators: $\varrho_{y,N}(z) =  \varrho_{p,N}( V_N^{1/2} z)  \det(V_N)^{1/2}$.
Using the previous relations the \mbox{$L^1$-convergence} for the KDE $\varrho_{p,N}$ follows:
\begin{equation} \label{KonvMehrdim}
	\Vert \varrho_{p,N} - \varrho_p \Vert_{L^1} = \Vert \varrho_{y,N} - \varrho_y \Vert_{L^1} \xrightarrow{N \rightarrow \infty} 0 \quad \mathbb{P}\text{-almost surely}.
\end{equation}
Applying \textit{Scheff\'e}'s \textit{lemma} (see \cite{DevroyeGyorfi}) yields
\begin{equation}  \label{ScheffeMehrdim}
	 \bigg \vert \int_{P^{\max}_{\min}} \varrho_p(z) ~ dz - \int_{P^{\max}_{\min}} \varrho_{p,N}(z) ~ dz \bigg \vert \leq \frac{1}{2} \Vert \varrho_{p,N} - \varrho_p \Vert_{L^1} \xrightarrow{N \rightarrow \infty} 0 \quad \mathbb{P}\text{-almost surely}	.
\end{equation}
Thus, the integral of the estimator over the pressure bounds converges $\mathbb{P}\text{-almost}$ surely to the probability $\mathbb{P}(p \in P^{\max}_{\min})$. Now, we can use this integral as an approximation for the probability $\mathbb{P}( b \in M)$ in (\ref{eq:ProbMultiKDE}):
\begin{equation}
\mathbb{P}(b \in M) = \mathbb{P}(p \in P^{\max}_{\min}) \approx  \int_{P^{\max}_{\min}} \varrho_{p,N}(z) =: \mathbb{P}_N(b \in M) .
\end{equation}
This multidimensional integral has some useful properties, which we will need when we derive the necessary optimality conditions.
We want to mention again, that we can compute the sampling set $\mathcal{P}_{\mathcal{B}}$ analytically, because we know the solution of the stationary isothermal Euler equations.

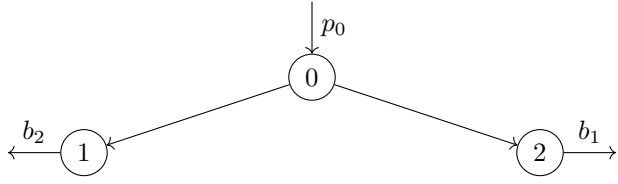
\begin{figure}[htbp]
	\centering
	\begin{tikzpicture}
		\node [minimum size=0.5cm] (A) at (0,0) [circle, draw] {$0$};
		\node [minimum size=0.5cm] (B) at (-3,-1) [circle, draw] {$1$};
		\node [minimum size=0.5cm] (C) at (3,-1) [circle, draw] {$2$};

		\draw[->] (A) to node[above] {$$} (B);
		\draw[->] (A) to node[above] {$$} (C);

		\draw[->] (0,1) to node[right] {$p_0$} (A);
		\draw[->] (C) to node[above] {$b_1$} (4,-1);
		\draw[->] (B) to node[above] {$b_2$} (-4,-1);
	\end{tikzpicture}
	\caption{Minimal tree-structured graph}
	\label{figure:ExampleMinimalTree}
\end{figure}

\paragraph*{Example 2:}\ This short example should illustrate the results of both approaches applied to the tree-structured graph with three nodes and two edges shown in \hyperref[figure:ExampleMinimalTree]{\textit{Figure \ref*{figure:ExampleMinimalTree}}}. The values (without units) are given in \hyperref[tab:twoEdgeValues]{\textit{Table \ref*{tab:twoEdgeValues}}}.

\begin{table}  [h]
	\centering
	\begin{tabular}{| c | c | c | c | c | c | c |}
			\hline
			$p_0$ 		& $p^{\min}$	& $p^{\max}$	& $\mu$		& $\Sigma$	& $\phi$ \\	
			\hline
			$60$		& $\begin{pmatrix} 40 \\ 30 \end{pmatrix}$	& $\begin{pmatrix} 60 \\ 50 \end{pmatrix}$ & $\begin{pmatrix} 4 \\ 4 \end{pmatrix}$ & $\begin{pmatrix} 0.25 & 0 \\ 0 & 0.25 \end{pmatrix}$	
			& $\begin{pmatrix} 100 \\ 100 \end{pmatrix}$ 	\\		
			\hline
	\end{tabular}
	\caption{Values for the example with two edges}
	\label{tab:twoEdgeValues}
\end{table}

From the inequalities in \hyperref[lemma:inequalitiesStationary]{\textit{Lemma \ref*{lemma:inequalitiesStationary}}} we get
\begin{equation*}
	M = \left\{ b \in \mathbb{R}^2_{\geq 0}\ \bigg \vert \begin{pmatrix} 0 \\ \sqrt{11}	\end{pmatrix} \leq \begin{bmatrix} b_1 \\ b_2 \end{bmatrix}	 \leq \begin{pmatrix} \sqrt{20} \\ \sqrt{27} \end{pmatrix} \right\}.
\end{equation*}
As in the last example, we compare both methods, the SRD and the KDE, with a classical Monte Carlo (MC) method. For the MC method and the KDE approach, we use the same sampling of $1 \cdot 10^5$ points. The result for $8$ tests is shown in \hyperref[tab:twoEdgesResults]{\textit{Table \ref*{tab:twoEdgesResults}}}.

\begin{table} [h]
	\centering
	\begin{tabular}{| c | c | c | c | c | c | c | c | c |}
		\hline
			& Test 1	& Test 2	& Test 3	& Test 4	& Test 5	& Test 6	& Test 7	& Test 8	\\
		\hline
		MC	& $75.07\%$ & $75.15\%$ & $75.22\%$ & $75.08\%$ & $74.88\%$ & $75.02\%$ & $74.99\%$ & $75.39\%$ \\
		\hline
		KDE	& $74.75\%$ & $74.84\%$ & $74.89\%$ & $74.78\%$ & $74.55\%$ & $74.76\%$ & $74.67\%$ & $75.04\%$ \\
		\hline
		SRD & \multicolumn{8}{c |}{$74.95\%$} \\
		\hline
	\end{tabular}
	\caption{Results for the example with two edges}
	\label{tab:twoEdgesResults}
\end{table}
The sampling of the SRD consists of $1 \cdot 10^4$ points uniform distribution of the sphere $\mathbb{S}^1$. Thus the SRD gives always the same (good) result of $74.95\%$, when rounded to $4$ digits. Again, the MC method and the KDE approach are quite close. The mean probability in MC resp. KDE is $75.10\%$ resp. $74.79\%$ and the variance is $0.0235$ resp. $0.0215$. As in the example with one edge we provide the confidence intervals for confidence level $95\%$ here. For the MC probability the confidence interval is $[74.97\%, 75.23\%]$ and for the KDE probability it is $[74.66\%, 74.91\%]$. The (good) result of the SRD close to both intervals. Thus, also in the two dimensional case, the KDE approach is quite good for computing the desired probability. The computing time in every test is quite reasonable. The computation time of the MC method and the KDE approach needs less than one second, while the SRD needs almost two seconds, but the focus of the implementation was on correctness, not on efficiency. So the computing time of the implementation of course can be improved.

\subsection{Stochastic optimization on stationary gas networks}

In this subsection, we formulate necessary conditions for optimization problems with approximated probabilistic constraints. Both, MC and SRD, give algorithmic ways to compute the probability for a random load vector to be feasible. With a KDE approach, which provides a sufficiently good approximation of the probability (if the sample size is sufficiently large), we can get necessary optimality conditions for certain optimization problems with approximated probabilistic constraints. Define the set
\begin{equation*}
	\mathcal{P}_0 := \bigotimes_{i=1}^n\ [p_i^{\min}, \infty) \ \subseteq \mathbb{R}^n,
\end{equation*}
and let a function
\begin{equation*}
	f : \mathbb{R}^n \times \mathbb{R} \rightarrow \mathbb{R},\quad (p^{\max}, p_0) \mapsto f(p^{\max}, p_0)
\end{equation*}
be given. For a probability level $\alpha \in (0,1)$ consider the optimization problems
\begin{equation} \label{eq:optimizationPressureBoundsStationary} \left\{ \quad \begin{aligned}
	\min_{p^{\max} \in \mathcal{P}_0} \quad &f(p_0, p^{\max}) \\
	\text{s.t.} \quad & \mathbb{P}(b \in M(p^{\max})) \geq \alpha,
\end{aligned} \right. \end{equation}
and
\begin{equation} \label{eq:optimizationInputPressureStationary} \left\{ \quad \begin{aligned}
	\min_{p_0 \in \mathbb{R}_{\geq 0}} \quad &f(p_0, p^{\max}) \\
	\text{s.t.} \quad & \mathbb{P}(b \in M(p_0)) \geq \alpha.
\end{aligned} \right. \end{equation}
Normally, $\alpha$ is chosen large, s.t. $\alpha$ is almost $1$. As mentioned before, our aim here is to write down the necessary optimality conditions in an appropriate way. In \hyperref[sec:stationarySingleEdge]{\textit{Section \ref*{sec:stationarySingleEdge}}} and \hyperref[sec:stationaryTree]{\textit{Section \ref*{sec:stationaryTree}}} we stated the $\mathbb{P}$-almost surely convergence of the KDE to the exact probability density. Further the numerical examples provide good and accurate results. In fact, we formulate the optimality conditions for the approximated optimization problems
\begin{equation} \label{eq:optimizationPressureBoundsStationaryApprox} \left\{ \quad \begin{aligned}
	\min_{p^{\max} \in \mathcal{P}_0} \quad &f(p_0, p^{\max}) \\
	\text{s.t.} \quad & \mathbb{P}_N(b \in M(p^{\max})) \geq \alpha,
\end{aligned} \right. \end{equation}
and
\begin{equation} \label{eq:optimizationInputPressureStationaryApprox} \left\{ \quad \begin{aligned}
	\min_{p_0 \in \mathbb{R}_{\geq 0}} \quad &f(p_0, p^{\max}) \\
	\text{s.t.} \quad & \mathbb{P}_N(b \in M(p_0)) \geq \alpha.
\end{aligned} \right. \end{equation}
We mention again that due to the convergence results stated before, the approximated probabilistic constraint converges $\mathbb{P}$-almost surely to the exact probabilistic constraint for $N \rightarrow \infty$. We define
\begin{equation*}
	P^{\max}_{\min} := \bigotimes_{i=1}^n [p_i^{\min}, p_i^{\max}].
\end{equation*}
Our aim is now to integrate the kernel density estimator of the pressure at the nodes $v_1, \cdots, v_n$ over the pressure bounds. Since the following computations hold for the probability in (\ref{eq:optimizationPressureBoundsStationaryApprox}) as well as in (\ref{eq:optimizationInputPressureStationaryApprox}), we neglect the argument of $M$ from here on. We have
\begin{equation*} \begin{aligned}
	\mathbb{P}_N(b \in M) &= \int_{P^{\max}_{\min}} \varrho_{p,N}(z) dz \\
	&= \frac{1}{N \prod_{j=1}^n h_j} \sum_{i=1}^N \int_{P^{\max}_{\min}} \prod_{j=1}^n \frac{1}{\sqrt{2 \pi}} \exp \left( -\frac{1}{2} \left( \frac{z_j - p_j(b^{\mathcal{S},i})}{h_j} \right)^2 \right) d z,
\end{aligned} \end{equation*}
with $\varrho_{p,N}$ as in (\ref{eq:multiKDE}). Since $P^{\max}_{\min}$ is a $n$-dimensional cuboid and $\varrho_{p,N}(z)$ is continuous we can use \textit{Fubini's Theorem}. Thus we have
\begin{gather*}
	\mathbb{P}_N(b \in M) = \\
	= \frac{1}{N \prod_{j=1}^n h_j} \sum_{i=1}^N \int_{p_1^{\min}}^{p_1^{\max}} \cdots \int_{p_n^{\min}}^{p_n^{\max}} \prod_{j=1}^n \frac{1}{\sqrt{2 \pi}} \exp \left( -\frac{1}{2} \left( \frac{z_j - p_j(b^{\mathcal{S},i})}{h_j} \right)^2 \right) d z_n \cdots d z_1,
\end{gather*}
and as the density estimation of the pressure is a product of an exponential function in every dimension, we can exchange the integral and the product. It follows
\begin{equation*} \begin{aligned}
\mathbb{P}_N(b \in M) 	
	&= \frac{1}{N \prod_{j=1}^n h_j} \sum_{i=1}^N \prod_{j=1}^n \int_{p_j^{\min}}^{p_j^{\max}} \frac{1}{\sqrt{2 \pi}} \exp \left( -\frac{1}{2} \left( \frac{z_j - p_j(b^{\mathcal{S},i})}{h_j} \right)^2 \right) d z_j.
\end{aligned} \end{equation*}
We define
\begin{equation*}
	\varphi_{i,j} : \mathbb{R} \rightarrow \mathbb{R}, \quad \varphi_{i,j} : x \mapsto \left( \frac{x - p_j(b^{\mathcal{S},i})}{\sqrt{2} h_j} \right), \\
\end{equation*}
and we set $\tau_{i,j} := \varphi_{i,j}(z_j)$ and use integration by substitution. Then with $\varphi'_{i,j}(x) = (\sqrt{2}\ h_j)^{-1}$ we get
\begin{equation*} \begin{aligned}
\mathbb{P}_N(b \in M) 	
	&= \frac{1}{N \prod_{j=1}^n h_j} \sum_{i=1}^N \prod_{j=1}^n \int_{p_j^{\min}}^{p_j^{\max}} \frac{1}{\sqrt{2 \pi}} \exp \left( -\varphi^2_{i,j}(z_j) \right) d z_j \\
	&= \frac{1}{N \prod_{j=1}^n h_j} \sum_{i=1}^N \prod_{j=1}^n \int_{\varphi_{i,j}(p_j^{\min})}^{\varphi_{i,j}(p_j^{\max})} \frac{1}{\sqrt{2 \pi}} \exp \left( -\tau^2_{i,j} \right) \sqrt{2}\ h_j\ d \tau_{i,j} \\
	&= \frac{1}{N} \sum_{i=1}^N \prod_{j=1}^n \int_{\varphi_{i,j}(p_j^{\min})}^{\varphi_{i,j}(p_j^{\max})} \frac{1}{\sqrt{\pi}} \exp \left( -\tau^2_{i,j} \right) d \tau_{i,j}.
\end{aligned} \end{equation*}
This formula contains the \textit{Gauss error function} (see e.g. \cite{Andrews}):
\begin{equation} \label{eq:GaussErrorFunction}
	\erf(x) := \frac{2}{\sqrt{\pi}} \int_0^x \exp \left( -t^2 \right) dt.
\end{equation}

We insert the Gauss error function in the previous integral term and we obtain
\begin{equation} \begin{gathered} \label{eq:probabilityKDEStationaryErf}
\begin{aligned}
	\mathbb{P}_N(b \in M) &= \int_{P_{\min}^{\max}} \varrho_{p,N}(z) dz  \\ &= \frac{1}{N} \frac{1}{2^n} \sum_{i=1}^N \prod_{j=1}^n \left[ \erf \left( \varphi_{i,j}(p_j^{\max}) \right) - \erf \left( \varphi_{i,j}(p_j^{\min}) \right) \right].
	\end{aligned}
\end{gathered} \end{equation}
Now we consider problem (\ref{eq:optimizationPressureBoundsStationaryApprox}). For $\alpha \in (0,1)$ we define
\begin{equation*}
	g_{\alpha} : \mathbb{R}^n \rightarrow \mathbb{R},\quad p^{\max} \mapsto \alpha - \mathbb{P}_N(b \in M(p^{\max})).
\end{equation*}
Thus we have
\begin{equation*}
	g_\alpha(p^{\max}) = \alpha - \frac{1}{N} \sum_{i=1}^N \prod_{j=1}^n \int_{ \varphi_{i,j}(p_j^{\min}) }^{ \varphi_{i,j}(p_j^{\max}) } \frac{1}{\sqrt{\pi}} \exp\left( -\tau_{i,j}^2 \right) d \tau_{i,j}.
\end{equation*}
We compute the partial derivatives of $g_\alpha$. For $k \in \{1, \cdots, n\}$ we have
\begin{equation*} \begin{aligned}
	\frac{\partial}{\partial p_k^{\max}} g_\alpha(p^{\max}) = - \frac{1}{N} \sum_{i=1}^N \left[ \prod_{j=1, j \neq k}^n \int_{ \varphi_{i,j}(p_j^{\min}) }^{ \varphi_{i,j}(p_j^{\max}) } \frac{1}{\sqrt{\pi}} \exp\left( -\tau_{i,j}^2 \right) d \tau_{i,j} \right. \\
	\left. \cdot \frac{1}{\sqrt{\pi}} \exp\left( - \varphi^2_{i,k}(p_k^{\max}) \right) \frac{1}{\sqrt{2} h_k} \right],
\end{aligned} \end{equation*}
and with the Gauss error function (\ref{eq:GaussErrorFunction}) it follows
\begin{equation} \begin{aligned} \label{eq:partialDerivativePressureBoundStationary}
	\frac{\partial}{\partial p_k^{\max}} g_\alpha(p^{\max}) = - \frac{1}{N} \frac{1}{2^n} \sum_{i=1}^N \left[ \prod_{j=1, j\neq k}^n \left[ \erf \left( \varphi_{i,j}(p_j^{\max}) \right) - \erf \left( \varphi_{i,j}(p_j^{\min}) \right) \right] \right. \\
	\left. \cdot \frac{\sqrt{2}}{\sqrt{\pi} h_k} \exp \left( - \varphi_{i,k}^2(p_k^{\max}) \right) \right].
\end{aligned} \end{equation}
Then, the $k$-th component of the gradient $\nabla g_\alpha(p^{\max}) \in \mathbb{R}^n$ is given by (\ref{eq:partialDerivativePressureBoundStationary}). Note that for $b^{S,1}, \cdots, b^{S,N} \in \mathbb{R}^n$ ($N > 1$) and $p_i^{\max} > p_i^{\min}$ ($i = 1,\cdots,n$), the partial derivatives in (\ref{eq:partialDerivativePressureBoundStationary}) are negative for all $p^{\max} \in \mathbb{R}^n$.
\begin{remark} \label{remark:mfcq}
Since (\ref{eq:optimizationPressureBoundsStationaryApprox}) has only one constraint, the \textit{linear independent constraint qualification} (LICQ) holds for every $\tilde{p}^{\max} > p^{\min}$ (componentwise) with $g_\alpha(\tilde{p}^{\max}) = 0$. \\
\end{remark}

Now we can state necessary optimality conditions for the optimization problem (\ref{eq:optimizationPressureBoundsStationaryApprox}):
\begin{corollary} \label{corollary:pMax}
Let $p^{*,\max} \in \mathbb{R}^n$ be a (local) optimal solution of (\ref{eq:optimizationPressureBoundsStationaryApprox}). Since the LICQ holds in $p^{*,\max}$, there exists a multiplier $\mu^* \geq 0$, s.t.
\begin{equation*} \begin{aligned}
	\nabla_{p^{\max}} f(p^{*,\max}, p_0) + \mu^* \nabla g_\alpha(p^{*,\max}) &= 0, \\
	g_\alpha(p^{*,\max}) &\leq 0, \\
	\mu^* g_\alpha(p^{*,\max}) &= 0.
\end{aligned} \end{equation*}
Thus, $(p^{*,\max}, \mu^*) \in \mathbb{R}^{n+1}$ is a \textit{Karush-Kuhn-Tucker point}.
\end{corollary}

Now we consider problem (\ref{eq:optimizationInputPressureStationaryApprox}). We slightly change the notation to add the explicit dependence on $p_0$, so we write $p(b^{\mathcal{S},i},p_0)$ ($i = 1, \cdots, N$) instead of $p(b^{\mathcal{S},i})$ for the samples in the set $\mathcal{P}_{\mathcal{B}}$. We redefine the function $\varphi_{i,j}$ ($i = 1, \cdots, N$, $j = 1, \cdots, n$) as
\begin{equation*}
	\varphi_{i,j}: \mathbb{R} \times \mathbb{R} \rightarrow \mathbb{R} \quad (x,y) \mapsto \left( \frac{x - p_j(b^{\mathcal{S},i},y)}{\sqrt{2} h_j} \right),
\end{equation*}
and we define the constraint of (\ref{eq:optimizationInputPressureStationaryApprox}) as
\begin{equation*}
	\gamma_\alpha : \mathbb{R} \rightarrow \mathbb{R} \quad p_0 \mapsto \alpha - \mathbb{P}_N(b \in M(p_0)).
\end{equation*}
Thus we have
\begin{equation*}
	\gamma_\alpha(p_0) = \alpha - \frac{1}{N \prod_{j=1}^n h_j} \sum_{i=1}^N \prod_{j=1}^n \int_{p_j^{\min}}^{p_j^{\max}} \frac{1}{\sqrt{2 \pi}} \exp\left( - \varphi_{i,j}^2(z_j, p_0) \right) d z_j.
\end{equation*}
For the derivative with respect to $p_0$, it follows
\begin{equation*} \begin{aligned}
	\frac{d}{d p_0} \gamma_\alpha(p_0) &= - \frac{1}{N \prod_{j=1}^n h_j} \sum_{i=1}^N \frac{d}{d p_0} \left( \prod_{j=1}^n \int_{p_j^{\min}}^{p_j^{\max}} \frac{1}{\sqrt{2 \pi}} \exp\left( - \varphi_{i,j}^2(z_j, p_0) \right) d z_j \right) \\
	&= - \frac{1}{N \prod_{j=1}^n h_j} \sum_{i=1}^N \sum_{k=1}^n \prod_{j=1, j\neq k}^n \int_{p_j^{\min}}^{p_j^{\max}} \frac{1}{\sqrt{2 \pi}} \exp\left( - \varphi_{i,j}^2(z_j, p_0) \right) d z_j \\
	&\hspace{4cm}\cdot \frac{d}{d p_0} \int_{p_k^{\min}}^{p_k^{\max}} \frac{1}{\sqrt{2 \pi}} \exp\left( - \varphi_{i,k}^2(z_k, p_0) \right) d z_k.
\end{aligned} \end{equation*}
Due to the \textit{dominated convergence theorem} we can exchange the integral and the derivative, thus we have
\begin{equation*} \begin{aligned}
	\frac{d}{d p_0} \gamma_\alpha(p_0) &= - \frac{1}{N \prod_{j=1}^n h_j} \sum_{i=1}^N \sum_{k=1}^n \prod_{j=1, j\neq k}^n \int_{p_j^{\min}}^{p_j^{\max}} \frac{1}{\sqrt{2 \pi}} \exp\left( - \varphi_{i,j}^2(z_j, p_0) \right) d z_j \\
	&\hspace{4cm}\cdot \int_{p_k^{\min}}^{p_k^{\max}} \frac{d}{d p_0} \frac{1}{\sqrt{2 \pi}} \exp\left( - \varphi_{i,k}^2(z_k, p_0) \right) d z_k \\
	&= - \frac{1}{N \prod_{j=1}^n h_j} \sum_{i=1}^N \sum_{k=1}^n \prod_{j=1, j\neq k}^n \int_{p_j^{\min}}^{p_j^{\max}} \frac{1}{\sqrt{2 \pi}} \exp\left( - \varphi_{i,j}^2(z_j, p_0) \right) d z_j \\
	&\hspace{.5cm}\cdot \int_{p_k^{\min}}^{p_k^{\max}} \frac{1}{\sqrt{\pi} h_k} \exp\left( - \varphi_{i,k}^2(z_k, p_0) \right) \varphi_{i,k}(z_k, p_0) \frac{d}{d p_0} p_k(b_i, p_0) d z_k.
\end{aligned} \end{equation*}
We define $\tau_{i,j} := \varphi_{i,j}(z_j, p_0)$ and since $p_k(b_i, p_0)$ is independent of $z_k$, it follows
\begin{equation*} \begin{aligned}
	\frac{d}{d p_0} \gamma_\alpha(p_0) = - \frac{1}{N \prod_{j=1}^n h_j} \sum_{i=1}^N \sum_{k=1}^n \prod_{j=1, j\neq k}^n \int\limits_{\varphi_{i,j}(p_j^{\min},p_0)}^{\varphi_{i,j}(p_j^{\max},p_0)} \frac{h_j}{\sqrt{\pi}} \exp\left( - \tau_{i,j}^2 \right) d \tau_{i,j} \\
	\cdot \frac{d}{d p_0} p_k(b_i, p_0) \int\limits_{\varphi_{i,k}(p_k^{\min},p_0)}^{\varphi_{i,k}(p_k^{\max},p_0)} \frac{\sqrt{2}}{\sqrt{\pi}} \exp\left( - \tau_{i,k}^2 \right) \tau_{i,k} d \tau_{i,k}.
\end{aligned} \end{equation*}
The second integral can be solved analytically and yields:
\begin{equation*}
	\left[ \frac{d}{dx} \left( -\frac{1}{\sqrt{2 \pi}} \exp(-x^2) \right) = \frac{\sqrt{2}}{\sqrt{\pi}} \exp(-x^2) x \right].
\end{equation*}
Hence, we have
\begin{equation*} \begin{aligned}
	\frac{d}{d p_0} \gamma_\alpha(p_0) = -\frac{1}{N} \frac{1}{2^n} \sum_{i=1}^N \sum_{k=1}^n \left[ \prod_{j=1,j \neq k}^n \left[ \erf\left( \varphi_{i,j}(p_j^{\max},p_0) \right) - \erf\left( \varphi_{i,j}(p_j^{\min},p_0) \right) \right] \right. \\
	\cdot \left. \frac{\sqrt{2}}{\sqrt{\pi} h_k} \frac{d}{d p_0} p_k(b_i, p_0) \left[ -\exp\left( -\varphi^2_{i,k}(p_k^{\max},p_0) \right) + \exp\left( -\varphi^2_{i,k}(p_k^{\min},p_0) \right) \right] \right].
\end{aligned} \end{equation*}
In the setting of the stationary gas networks it is true that
\begin{equation*}
	\frac{d}{d p_0} p_k(b^{\mathcal{S},i},p_0) = \frac{p_0}{p_k(b^{\mathcal{S},i},p_0)}.
\end{equation*}

\begin{corollary} \label{corollary:p0}
	Let $p^*_0 \in \mathbb{R}$ be a (local) optimal solution of (\ref{eq:optimizationInputPressureStationaryApprox}). Since the LICQ holds in $p^*_0$ (cf. \hyperref[remark:mfcq]{\textit{Remark \ref*{remark:mfcq}}}), then there exists a multiplier $\mu^* \geq 0$, s.t.
\begin{equation*} \begin{aligned}
	\nabla_{p_0} f(p^{\max}, p^*_0) + \mu^* \nabla \gamma_\alpha(p^*_0) &= 0, \\
	\gamma_\alpha(p^*_0) &\leq 0, \\
	\mu^* \gamma_\alpha(p^*_0) &= 0.
\end{aligned} \end{equation*}	
Thus the point $(p^*_0, \mu^*) \in \mathbb{R}^2$ is a Karush-Kuhn-Tucker point.
\end{corollary}

If the objective function $f$ is strictly convex and the feasible set is convex, then all necessary conditions stated here are sufficient. In this case, \hyperref[corollary:pMax]{\textit{Corollary \ref*{corollary:pMax}}} and \hyperref[corollary:p0]{\textit{Corollary \ref*{corollary:p0}}} give a characterization of the (unique) optimal solution of the approximated problems (\ref{eq:optimizationPressureBoundsStationaryApprox}) and (\ref{eq:optimizationInputPressureStationaryApprox}).

\begin{remark}
The question whether the solutions of the approximated problems (\ref{eq:optimizationPressureBoundsStationaryApprox}) and (\ref{eq:optimizationInputPressureStationaryApprox}) converge to the solutions of (\ref{eq:optimizationPressureBoundsStationary}) and (\ref{eq:optimizationInputPressureStationary}), is out of scope of this work but all numerical results and tests hypothesize the convergence if the sample size goes to infinity.
\end{remark}

\subsection{Application to a realisitic gas network} \label{sec:GasLib11}

The \textit{GasLib}\footnote[1]{\url{http://gaslib.zib.de/}} promotes research on gas networks by providing realistic benchmark instances. We use the \textit{GasLib-11} as a meaningful example. A scheme of the GasLib-11 is shown in \hyperref[figure:GasLib-11]{\textit{Figure \ref*{figure:GasLib-11}}} and more information can be found at \url{http://gaslib.zib.de/testData.html}. \\

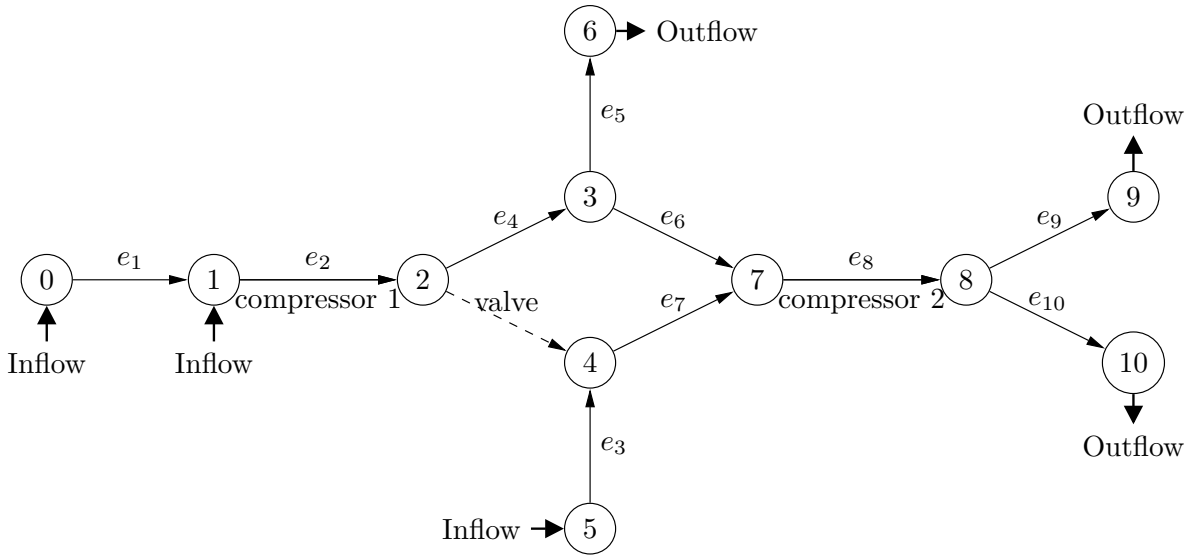
\begin{figure}[htbp]
	\centering
	\resizebox{\textwidth}{!}{
		\begin{tikzpicture}
			\node [minimum size=0.5cm] (A) at (0,0) [circle, draw] {$0$};
			\node [minimum size=0.5cm] (B) at (2,0) [circle, draw] {$1$};
			\node [minimum size=0.5cm] (C) at (4.5,0) [circle, draw] {$2$};
			\node [minimum size=0.5cm] (D) at (6.5,1) [circle, draw] {$3$};
			\node [minimum size=0.5cm] (E) at (6.5,-1) [circle, draw] {$4$};
			\node [minimum size=0.5cm] (F) at (6.5,-3) [circle, draw] {$5$};
			\node [minimum size=0.5cm] (G) at (6.5,3) [circle, draw] {$6$};
			\node [minimum size=0.5cm] (H) at (8.5,0) [circle, draw] {$7$};
			\node [minimum size=0.5cm] (I) at (11,0) [circle, draw] {$8$};
			\node [minimum size=0.5cm] (J) at (13,1) [circle, draw] {$9$};
			\node [minimum size=0.5cm] (K) at (13,-1) [circle, draw] {$10$};
		
			\draw[->, >={Triangle[length=0pt 2*8,width=0pt 8]}] (A) to node[above] {$e_1$} (B);
			\draw[->, >={Triangle[length=0pt 2*8,width=0pt 8]}] (B) to node[above] {$e_2$} (C);
			\draw[->, >={Triangle[length=0pt 2*8,width=0pt 8]}] (F) to node[right] {$e_3$} (E);
			\draw[->, >={Triangle[length=0pt 2*8,width=0pt 8]}] (C) to node[above] {$e_4$} (D);
			\draw[->, >={Triangle[length=0pt 2*8,width=0pt 8]}] (D) to node[right] {$e_5$} (G);
			\draw[->, >={Triangle[length=0pt 2*8,width=0pt 8]}] (D) to node[above] {$e_6$} (H);
			\draw[->, >={Triangle[length=0pt 2*8,width=0pt 8]}] (E) to node[above] {$e_7$} (H);
			\draw[->, >={Triangle[length=0pt 2*8,width=0pt 8]}] (H) to node[above] {$e_8$} (I);
			\draw[->, >={Triangle[length=0pt 2*8,width=0pt 8]}] (I) to node[above] {$e_9$} (J);
			\draw[->, >={Triangle[length=0pt 2*8,width=0pt 8]}] (I) to node[above] {$e_{10}$} (K);
			\draw[->, >={Triangle[length=0pt 2*8,width=0pt 8]}] (B) to node[below] {compressor 1} (C);
			\draw[->, >={Triangle[length=0pt 2*8,width=0pt 8]}] (H) to node[below] {compressor 2} (I);
			\draw[->, >={Triangle[length=0pt 2*8,width=0pt 8]}, dashed] (C) to node[above] {valve} (E);
		
			\node [minimum size=0.5cm] (L) at (0,-1) {Inflow};
			\node [minimum size=0.5cm] (M) at (2,-1) {Inflow};
			\node [minimum size=0.5cm] (N) at (5.2,-3) {Inflow};
			\draw[->, >={Triangle[length=0pt 1*8,width=0pt 8]}, thick] (L) to (A);
			\draw[->, >={Triangle[length=0pt 1*8,width=0pt 8]}, thick] (M) to (B);
			\draw[->, >={Triangle[length=0pt 1*8,width=0pt 8]}, thick] (N) to (F);
			\node [minimum size=0.5cm] (O) at (7.9,3) {Outflow};
			\node [minimum size=0.5cm] (P) at (13,2) {Outflow};
			\node [minimum size=0.5cm] (Q) at (13,-2) {Outflow};
			\draw[->, >={Triangle[length=0pt 1*8,width=0pt 8]}, thick] (G) to (O);
			\draw[->, >={Triangle[length=0pt 1*8,width=0pt 8]}, thick] (J) to (P);
			\draw[->, >={Triangle[length=0pt 1*8,width=0pt 8]}, thick] (K) to (Q);
		\end{tikzpicture}
	}
	\caption{A scheme of the GasLib-11}
	\label{figure:GasLib-11}
\end{figure}

The GasLib-11 consists in $11$ nodes and $11$ edges. Two of the edges represent compressor stations and one edge represents a valve. Compressor stations counteract the pressure loss caused by friction in the pipes. Here the compressor stations are modeled as frictionless pipes, that satisfy the equation
\begin{equation*}
	\frac{(p_{\text{in}})^2}{(p_{\text{out}})^2} = u,
\end{equation*}
as it is done in \cite{GugatSchuster}. This model for compressor stations is also suggested in \cite{EvaluatingGasNetworkCapacities}, where one gets an excellent overview about the details on how to model a compressor station. For our system we assume that the compressor at edge $e_2$ is switched off, i.e., $u_{e_2} = 1$ (so this edge can be modeled as frictionless pipe) and that the compressor station at edge $e_8$ increases the pressure by $20\%$, i.e., $u_{e_8} = 1.2$. The valve is also modeled as a frictionless pipe in which gas can be transported if the valve is opened and which cannot be used for gas transportation if the valve is closed. We assume that the valve is closed, so this edge vanishes in our implementation. For the remaining edges ($e_1, e_3, e_4, e_5, e_6, e_7, e_9, e_{10}$) we assume $\phi_{e_i} = 1$. \\
Further gas enters the network at the nodes $v_0$, $v_1$, $v_5$ and is transported through the network to the nodes $v_6$, $v_9$ and $v_{10}$. The values for the inlet pressure $p_0 = [p_{v_0}, p_{v_1}, p_{v_5}]$, the lower pressure bound $p^{\text{min}} = [p^{\text{min}}_{v_6}, p^{\text{min}}_{v_9}, p^{\text{min}}_{v_{10}}]$ and the probability distribution at the exit nodes $\mu = [\mu_{v_6}, \mu_{v_9}, \mu_{v_{10}}]$ and $\Sigma$ with $\text{diag}(\Sigma) = [\sigma^2_{v_6}, \sigma^2_{v_9}, \sigma^2_{v_{10}}]$ are given in \hyperref[table:GasLib-11]{\textit{Table \ref*{table:GasLib-11}}}.\\
\begin{table}[h!]
	\centering
	\begin{tabular}{| c | c | c | c |}
			\hline
			$p_0$ 		& $p^{\min}$	& $\mu$		& $\Sigma$ \\
			\hline
			$\begin{pmatrix} 60 \\ 58 \\ 60 \end{pmatrix} $	& $\begin{pmatrix} 40 \\ 40 \\ 40 \end{pmatrix}$	& $\begin{pmatrix} 20 \\ 15 \\ 18 \end{pmatrix}$	& $\begin{pmatrix} 2 & 0 & 0 \\ 0 & 2 & 0 \\ 0 & 0 & 2 \end{pmatrix}$  \\	
			\hline
	\end{tabular}
	\caption{Values for the GasLib-11}
	\label{table:GasLib-11}
\end{table}\\[0.1mm]
Consider the linear function
\begin{equation*}
	f : \mathbb{R}^3 \rightarrow \mathbb{R},\quad f : p^{\text{max}} \mapsto c^\top p^{\text{max}},
\end{equation*}
with $c = \mathbb{1}_3$. We first solve the deterministic problem
\begin{equation} \label{eq:optGasLib11Deter} \begin{aligned}
	\min_{p^{\text{max}} \in \mathcal{P}_0} \quad &f(p^{\text{max}}) \\
	\text{s.t.} \quad &b \in M(p^{\text{max}}),
\end{aligned} \end{equation}
where the load vector $b$ is given by the mean value $\mu$. We use default setting of the MATLAB$\textsuperscript{\textregistered}$-routine \textit{fmincon} to solve (\ref{eq:optGasLib11Deter}), which is an interior-point algorithm. It returns
\begin{equation*}
	p^{\text{max}}_{\text{det}} = \begin{bmatrix} 46.10 \\ 52.04 \\ 51.08
	\end{bmatrix},
\end{equation*}
as optimal deterministic solution, i.e., as the lowest upper pressure bound for the nodes $v_6$, $v_9$ and $v_{10}$. Now we consider the uncertain outflow at the nodes $v_6$, $v_9$ and $v_{10}$. We compute the probability for a random load vector to be feasible with respect to the optimal deterministic pressure bounds by using (\ref{eq:probabilityKDEStationaryErf}). The probability $\mathbb{P}(b \in M(p^{\max}_{\text{det}}))$ for $8$ tests (each with $1 \cdot 10^5$ samples) is shown in \hyperref[table:GasLib-11probability]{\textit{Table \ref*{table:GasLib-11probability}}}.
\begin{table} [h]
	\centering
	\begin{tabular}{| c | c | c | c | c | c | c | c | c |}
		\hline
			& Test 1	& Test 2	& Test 3	& Test 4	& Test 5	& Test 6	& Test 7	& Test 8	\\
		\hline
		MC	& $36.02\%$	& $35.66\%$	& $35.91\%$	& $35.86\%$	& $35.34\%$	& $35.48\%$	& $35.98\%$	& $35.90\%$	\\
		\hline	
		KDE & $35.72\%$ & $35.41\%$ & $35.48\%$ & $35.39\%$ & $34.92\%$ & $35.08\%$ & $35.75\%$ & $35.47\%$ \\
		\hline
	\end{tabular}
	\caption{Probability $\mathbb{P}(b \in M(p^{\max}_{\text{deter}}))$ for the optimal deterministic upper pressure bounds}
	\label{table:GasLib-11probability}
\end{table}
The probabilities for the deterministic optimal pressure bounds are unsatisfactory. The mean MC probability is $35.77\%$ and the mean KDE probability is $35.40\%$. For a confidence level of $95\%$ the confidence interval for the MC probability is $[35.56\%, 35.98\%]$ and the confidence interval for the KDE probability is $[35.16\%, 35.64\%]$. So if the boundary data (i.e., the gas demand) is uncertain, the optimal deterministic pressure bounds are unserviceable in the sense that these bounds do not provide a good operating gas network for uncertain gas demand. \\

Next we consider the probabilistic constrained optimization problem (\ref{eq:optimizationPressureBoundsStationaryApprox}). We set
\begin{equation*}
	\alpha := 0.75.
\end{equation*}
For arbitrary starting points the MATLAB$\textsuperscript{\textregistered}$-routine \textit{fmincon} sometimes struggles with finding a solution of (\ref{eq:optimizationPressureBoundsStationaryApprox}) but the optimal deterministic solution appears to be a good choice for the starting point of the routine. The results of $8$ Tests with $1 \cdot 10^5$ sampling points, i.e., the optimal upper pressure bounds $p^{\text{max}}$ at the nodes $v_6$, $v_9$ and $v_{10}$, are shown in \hyperref[table:GasLib-11Results]{\textit{Table \ref*{table:GasLib-11Results}}}. In $8$ more Tests we solve (\ref{eq:optimizationPressureBoundsStationaryApprox}) by using \hyperref[corollary:pMax]{\textit{Corollary \ref*{corollary:pMax}}}. The points that satisfy the necessary optimality conditions are always good candidates for the optimal solution. To be more precise on that we solve the following optimization problem using again \textit{fmincon}: \\
\begin{equation*} \begin{aligned}
	\min_{p^{\text{max}}, \mu} \quad f(p^{\text{max}}) \hspace{4.25cm} & \\
	\text{s.t.} \quad \nabla_{p^{\max}} f(p^{\max}, p_0) + \mu \nabla g_\alpha(p^{\max}) &= 0, \\
	g_\alpha(p^{\max}) &\leq 0, \\
	\mu g_\alpha(p^{\max}) &= 0, \\
	 \mu &\geq 0.
\end{aligned} \end{equation*}
Here, we get values that are almost equal to the optimal solution stated in \hyperref[table:GasLib-11Results]{\textit{Table \ref*{table:GasLib-11Results}}}, they vary in a range of $1 \cdot 10^{-6}$. From this fact one could expect that the necessary optimality conditions stated in \hyperref[corollary:pMax]{\textit{Corollary \ref*{corollary:pMax}}} are sufficient but we do not analyze this here. \\
\begin{table} [h]
	\centering
	\begin{tabular}{| c | c | c | c | c | c | c | c |}
		\hline
		Test 1	& Test 2	& Test 3	& Test 4	& Test 5	& Test 6	& Test 7	& Test 8	\\
		\hline
		$\begin{bmatrix} 47.51 \\ 53.33 \\ 52.44 \end{bmatrix}$	& $\begin{bmatrix} 47.51 \\ 53.34 \\ 52.45 \end{bmatrix}$	& $\begin{bmatrix} 47.52 \\ 53.33 \\ 52.46 \end{bmatrix}$	& $\begin{bmatrix} 47.52 \\ 53.34 \\ 52.46 \end{bmatrix}$	& $\begin{bmatrix} 47.51 \\ 53.35 \\ 52.46 \end{bmatrix}$	& $\begin{bmatrix} 47.51 \\ 53.35 \\ 52.45 \end{bmatrix}$	& $\begin{bmatrix} 47.53 \\ 53.33 \\ 52.44 \end{bmatrix}$	& $\begin{bmatrix} 47.52 \\ 53.33 \\ 52.46 \end{bmatrix}$	\\
		\hline		
	\end{tabular}
	\caption{Stochastic optimal upper pressure bounds $p^{\max}_{\text{stoch}}$}
	\label{table:GasLib-11Results}
\end{table}
One can see, that all results are almost equal. The optimal upper pressure bounds of the stochastic optimization problem (\ref{eq:optimizationPressureBoundsStationaryApprox}) are slightly larger than the optimal upper pressure bounds of the deterministic optimization problem (\ref{eq:optGasLib11Deter}) but the probability for a random load vector to be feasible is $75\%$, as it is required in the probabilistic constraint. The computation time for a single test was about $20$ minutes, where the optimization time was much less than a second. The $20$ minutes were almost only needed to solve the GasLib-11 $1 \cdot 10^5$ times. The solution of the necessary optimality conditions needed a little bit more time than the direct solution using \textit{fmincon} and (\ref{eq:probabilityKDEStationaryErf}), but solving the necessary optimality conditions leads to a solution more often even if the starting point of \textit{fmincon} is badly chosen.

\section{Dynamic flow networks} \label{sec:dynamic}

In this section, we extend the in \hyperref[sec:stationary]{\textit{Section \ref*{sec:stationary}}} introduced methods to dynamic systems. We first discuss probabilistic constraints in a dynamic setting and time dependent random boundary data. Then we consider a model, which we can solve analytically to use the idea of the SRD for dynamic systems and compare it with the idea of the KDE. Last we also formulate necessary optimality conditions for optimization problems with probabilistic constraints in a dynamic setting. \\

\subsection{Time dependent probabilistic constraints and random boundary data}

In \hyperref[sec:stationary]{\textit{Section \ref*{sec:stationary}}}, we computed the probability for a random vector to be feasible. So if we fix a point in time $t^* \in [0,T]$, we can use a similar procedure. But if we do not fix a point in time, we need an extension to the probabilistic constraint how it has to be understood for a time period. For a time dependent uncertain boundary function $b(t)$ and a (time dependent) feasible set $M(t)$, a possible formulation (the one, that we will use later) for the probabilistic constraint is
\begin{equation} \label{eq:timeDependentProbabilisticConstraint}
	\mathbb{P}(b \in M(t)\ \forall t \in [0,T]) \geq \alpha.
\end{equation}
That means, we want to guarantee, that a percentage $\alpha$ of all possible random boundary functions (in an appropriate probability space ($\Omega, \mathcal{A}, \mathcal{P})$) is feasible in every point in time $t \in [0,T]$. This is a very strong condition. In fact (\ref{eq:timeDependentProbabilisticConstraint}) is a so-called probust constraint, which means it is a mix between a probabilistic constraint and a robust constraint. This class of constraints has been developed recently and is currently of big interest in research (see e.g. \cite{VanAckooijHenrionPerezAros, AdelhuetteEtAl}). Another possibility is
\begin{equation*}
	\mathbb{P}(b \in M(t)) \geq \alpha \quad \forall t \in [0,T],
\end{equation*}
which means, that a random boundary function must be feasible with a percentage of $\alpha$ in every time point $t \in [0,T]$. For our applications, this might not make sense, because we want to guarantee that problems for a gas consumer only occur in worst case scenarios. This probabilistic constraint only states, that even in these worst case scenarios, the problems for a consumer stay small, but these small problems can occur in every point in time. Probabilistic constraints of this type have been discussed in \cite{VanAckooijFrangioniOliveira}. A third possibility for the time dependent probabilistic constraint is an ergodic formulation:
\begin{equation*}
	\frac{1}{T} \int_0^T \mathbb{P}(b \in M(t))\ dt \geq \alpha.
\end{equation*}
That means the ergodic probability during the time period $[0,T]$ must be large enough. This formulation might make sense in other applications, but not for the flow problems which are considered here (with the same argument as before). Thus we use the formulation (\ref{eq:timeDependentProbabilisticConstraint}) for time dependent probabilistic constraints. \\

Next we discuss the uncertain boundary data. For a random boundary data, we use a representation as Fourier series as it is done in \cite{FarshbafShakerGugatHeitschHenrion}. So for a deterministic boundary function $b_D: [0,T] \rightarrow \mathbb{R}$ with $b_D(0) = 0$ and for $m = 0, 1, 2, \cdots$, we define the orthonormal series
\begin{equation} \label{eq:FourierBasis}
	\psi_m(t) := \frac{\sqrt{2}}{\sqrt{T}} \sin \left( \left( \frac{\pi}{2} + m \pi \right) \frac{t}{T} \right),
\end{equation}
and the coefficients
\begin{equation} \label{eq:FourierCoefficients}
	a^0_m := \int_0^T b_D(t) \psi_m(t) dt.
\end{equation}
Then we can write the boundary function $b_D(t)$ in a series representation
\begin{equation} \label{eq:BoundaryAsFourier}
	b_D(t) = \sum_{m=0}^\infty a^0_m \psi_m(t).
\end{equation}
Now for $m \in \mathbb{N}_0$, we consider the Gaussian distributed random variables $a_m \sim \mathcal N (1, \sigma^2)$ for a mean value $1$ and a standard deviation $\sigma \in \mathbb{R}_+$ on an appropriate probability space $(\Omega,\mathcal{A},\mathbb{P})$. Then we consider the random boundary data
\begin{equation} \label{eq:FourierBoundary}
	b(t,\omega) = \sum_{m=0}^\infty a_m(\omega) a^0_m \psi_m(t).
\end{equation}
Since the random variables $a_m$ are all independent and identically distributed, we can use the fact that for $b_D \in L^2(0,T)$, we have also $b \in L^2(0,T)$ $\mathbb{P}$-almost surely. In \cite{MarcusPisier, Hill, FarshbafShakerGugatHeitschHenrion} the authors state that this approach even guarantees better regularity and it also holds for a larger class of random variables.

\begin{remark} \label{remark:FourierSeries}
For the numerical tests, we truncate the Fourier series after $N_F \in \mathbb{N}$ terms. Thus we use
\begin{equation*}
	b_D^{N_F}(t) = \sum_{m=0}^{N_F} a^0_m \psi_m(t)
\end{equation*}
instead of (\ref{eq:BoundaryAsFourier}) for the implementation of $b_D$. Because it holds
\begin{equation*}
	\lim_{N_F \rightarrow \infty} b_D^{N_F} = b_D,
\end{equation*}
this truncated Fourier series is a sufficient good expression for $b_D(t)$ for $N_F$ large enough. The question how to choose $N_F$ strongly depends on the data $b_D$ and on the desired accuracy of the Fourier series. In general one has to guarantee, that the truncation error is small. One criteria for finding a sufficient large number $N_F$ is to state a bound for the $L^2$-truncation error. For $\vartheta \in (0,1)$ we require that $N_F$ is chosen large enough, s.t.
\begin{equation*}
	\left\Vert b_D(t) - b_D^{N_F}(t) \right\Vert_{L^2}^2 \leq \vartheta \left\Vert b_D(t) \right\Vert_{L^2}^2.
\end{equation*}
Due to the convergence of the Fourier series it is always possible to find $N_F$ large enough, s.t. the $L^2$-error bound is satisfied for all $\vartheta \in (0,1)$. Another criteria for an sufficient large number $N_F$ is a bound for the $L^\infty$-truncation error. For $\vartheta \in (0,1)$ we require that $N_F$ is chosen large enough s.t.
\begin{equation*}
	\left\Vert b_D(t) - b_D^{N_F}(t) \right\Vert_{L^\infty} \leq \vartheta \left[ \sup_{\tau \in [0,T]} b_D(\tau) - \inf_{\tau \in [0,T]} b_D(\tau) \right].
\end{equation*}
The Gibbs phenomenon might cause problems regarding the $L^\infty$-error if $b_D$ contains discontinuities (see e.g. \cite{Thompson}), so in this case the $L^2$-error is the better choice. For continuous functions $b_D$ both estimates can be used to find a sufficient large number $N_F$. Usually $\vartheta$ is chosen small, even close to zero, i.e., $\vartheta = 1\%$ or $\vartheta = 0.1\%$, but this choice depends on the operator. Similarly we use
\begin{equation*}
	b^{N_F}(t,\omega) := \sum_{m=0}^{N_F} a_m(\omega) a^0_m \psi_m(t),
\end{equation*}
with $\mathcal{N}(1, \sigma^2)$-distributed random variables $a_0, \cdots, a_{N_F}$ instead of (\ref{eq:FourierBoundary}) as random boundary data for the implementation.
\end{remark}
This representation of a random boundary function as Fourier series requires $b_D(0) = 0$. If this is not given, i.e., if $b_D(0) \neq 0$, one can shift $b_D$ by $b_D(0)$, get the representation as Fourier series and shift this Fourier representation back by $b_D(0)$, as we do later in \hyperref[example3]{\textit{Example 3}}.

\subsection{Deterministic loads for a scalar PDE}

For $(t,x) \in [0,T] \times [0,L]$ and constants $d < 0$, $m \leq 0$, we consider the deterministic scalar linear PDE with initial condition and boundary condition
\begin{equation} \label{eq:linearScalarPDE} \left\{ \begin{aligned}
	&r_t(t,x) + d r_x(t,x) = m r(t,x), \\
	&r(0,x) = r_0(x), \\
	&r(t,L) = b(t).
\end{aligned} \right. \end{equation}
Here, $r$ is the concentration of the contamination. The term $d r_x$ describes the transport of the contamination according to the water flow and the term $mr$ describes the decay of the contamination.
This equation models the flow of contamination in water along a pipe or in a network (see \cite{GugatWater, FuegenschuhGoettlichHerty}). Assume $C^0$-compatibility between the initial and the boundary condition, which is $r_0(L) = b(0)$. We will specify the boundary condition later. We state $b(t) \geq 0$, if the water gets polluted and $b(t) < 0$ if the water gets cleaned.

For initial data $r_0 \in L^2(0,L)$ and boundary data $b \in L^2(0,T)$, a solution of (\ref{eq:linearScalarPDE}) is in $C([0,T], L^2(0,L))$ and it is analytically given by
\begin{equation*}
	r(t,x) = \begin{cases} \exp(mt)\ r_0(x - dt) & \text{if } x \leq L + dt, \\ \exp \left(m \frac{x - L}{d} \right) b(t - \frac{x - L}{d}) & \text{if } x > L + dt. \end{cases}
\end{equation*}

Now, we consider a linear graph $G = (\mathcal{V}, \mathcal{E})$ with vertex set $\mathcal{V} := \{v_0, \cdots, v_n \}$ and the set of edges $\mathcal{E} = \{e_1, \cdots, e_n\} \subseteq \mathcal{V} \times \mathcal{V}$. Every edge $e_i \in \mathcal{E}$ has a positive length $L_i$.
Linear means here, that every node has at most one outgoing edge (see \hyperref[figure:linearGraph]{\textit{Figure \ref*{figure:linearGraph}}}). For a formal definition see \cite{GugatSchultzSchuster}.

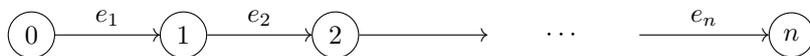
\begin{figure}[htbp]
	\centering
	\begin{tikzpicture}
		\node [minimum size=0.5cm] (A) at (0,0) [circle, draw] {$0$};
		\node [minimum size=0.5cm] (B) at (2,0) [circle, draw] {$1$};
		\node [minimum size=0.5cm] (C) at (4,0) [circle, draw] {$2$};
		\node [minimum size=0.5cm] (D) at (7,0) {$\cdots$};
		\node [minimum size=0.5cm] (E) at (10,0) [circle, draw] {$n$};

		\draw[->] (A) to node[above] {$e_1$} (B);
		\draw[->] (B) to node[above] {$e_2$} (C);
		\draw[->] (C) to (6,0);
		\draw[->] (8,0) to node[above] {$e_n$} (E);
	\end{tikzpicture}
	\caption{Linear graph with $n+1$ nodes}
	\label{figure:linearGraph}
\end{figure}

Equation (\ref{eq:linearScalarPDE}) holds on every edge. We assume conservation of the flow at the nodes, i.e.
\begin{equation*}
	r_i(t,L_i) = r_{i+1}(t,0) + b_i(t) \quad \forall i = 1, \cdots, n-1 \quad \forall t \in [0,T],
\end{equation*}
where $r_i$ denotes the contamination concentration on edge $e_i$ and $b_i$ denotes the boundary data at node $v_i$.

For constants $d_k < 0$, $m_k \leq 0$, the full model can be written as follows (with $(t,x) \in [0,T] \times [0,L_k]$ on the $k$-th edge and $k = 1, \cdots, n$):

\begin{equation} \label{eq:lineaScalarModelOnLinearGraphWithOutflows}
	\left\{ \quad \begin{aligned}
		&r_k(0,x) = r_{k,0}(x), \\
		&(r_k)_t(t,x) + d_k (r_k)_x(t,x) = m_k r_k(t,x), \\
		&r_k(t,L_k) = \begin{cases} b_n(t) &\text{if } k = n, \\ r_{k+1}(t,0) + b_k(t) &\text{else}. \end{cases}
\end{aligned} \right. \end{equation}

The model can be interpreted as follows: The graph represents a water network, where the water is contaminated at the nodes $v_i$ ($i = 1, \cdots n$). This contamination is distributed in the graph in a negative way (due to $d_k < 0$). We want to know the contamination rate at node $v_0$. Later, we assume the contamination rate at the nodes to be Gaussian distributed. Then for a time $t^* \in [0,T]$, we want to compute the probability for the contamination rate at node $v_0$ to fulfill box constraints using both, the SRD and the KDE. In both cases, we also consider the general time dependent chance constraints discussed before. The next theorem states an analytical solution for the model (\ref{eq:lineaScalarModelOnLinearGraphWithOutflows}).

\begin{satz} \label{theorem:scalarModelOnLinearGraph}
Let initial states $r_{k,0} \in H^1(0,L_k)$ and boundary conditions $b_k \in H^1(0,T)$ for $k = 1, \cdots, n$ be given. Then the solution of the $k$-th edge of (\ref{eq:lineaScalarModelOnLinearGraphWithOutflows}) is in $C^1([0,T], H^1(0,L_k))$ and it is analytically given for $x \geq d_k t + d_k \sum_{j=k}^n \frac{L_j}{d_j}$ by
\begin{equation*} \begin{aligned}
	r_k(t,x) &= \sum_{i=k}^n \exp \left( m_k \frac{x}{d_k} - \sum_{j=k}^i m_j \frac{L_j}{d_j} \right) b_i \left( t - \frac{x}{d_k} + \sum_{j=k}^i \frac{L_j}{d_j} \right).
\end{aligned} \end{equation*}
For the case $x < d_k t + d_k \sum_{j=k}^n \frac{L_j}{d_j}$ the solution is given by
\begin{equation*} \begin{aligned}
	r_k(t,x) = & \sum_{i=k}^{\ell-1} \exp \left( m_k \frac{x}{d_k} - \sum_{j=k}^i m_j \frac{L_j}{d_j} \right) b_i \left( t - \frac{x}{d_k} + \sum_{j=k}^i \frac{L_j}{d_j} \right) \\
	&+ \exp \left( m_\ell t - (m_\ell - m_k) \frac{x}{d_k} + \sum_{j=k}^{\ell-1} (m_\ell - m_j) \frac{L_j}{d_j} \right) r_{\ell,0} \left( -d_\ell t + d_\ell \frac{x}{d_k} - d_\ell \sum_{j=k}^{\ell-1} \frac{L_j}{d_j} \right)
\end{aligned} \end{equation*}
for $d_k t + d_k \sum_{j=k}^{\ell-1} \frac{L_j}{d_j} \leq x < d_k t + d_k \sum_{j=k}^{\ell} \frac{L_j}{d_j}$ and $\ell \in \{k, \cdots, n\}$ (with $d_k < 0$ for $k = 1, \cdots, n$).
\end{satz}

\begin{remark} \label{remark:informationTransport}
We set $t^* := \sum_{j=1}^n \frac{L_j}{\vert d_j \vert}$. For points in time $t \leq t^*$ the solution can depend explicitly on the initial condition. If we assume that $\vert d_i \vert$ are absolute velocities and $L_i$ are lengths, the information from the right boundary needs $\frac{L_n}{ \vert d_n \vert}$ seconds to travel along the $n$-th edge. Then after $\frac{L_n}{\vert d_n \vert}$ seconds, the solution of the $n$-th edge only depends on the boundary data, but the solution of edge $n-1$ can still depend on the initial condition of the $n$-th edge. A scheme of characteristics for a graph with $4$ edges is shown in \hyperref[fig:characteristicsOnFourEdges]{\textit{Figure \ref*{fig:characteristicsOnFourEdges}}}.
\end{remark}

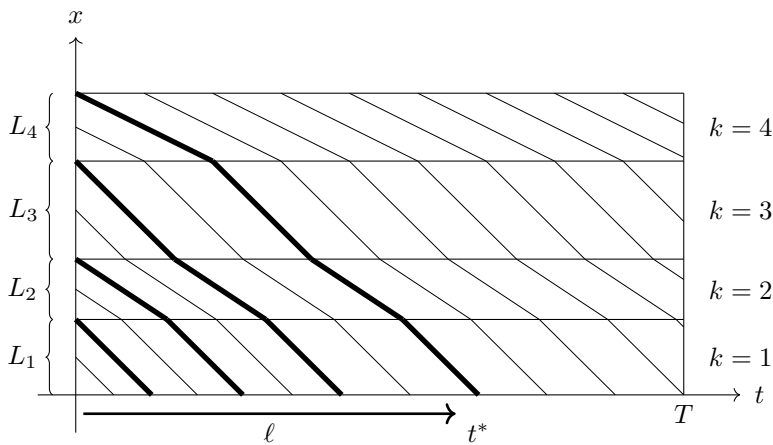
\begin{figure}[htbp]
	\centering
	\begin{tikzpicture}
		\node[minimum size=0.5cm] (A) at (9,0) {$t$};
		\node[minimum size=0.5cm] (B) at (0,5) {$x$};
		\draw[->] (-.5,0) to (A);
		\draw[->] (0,-.5) to (B);
		\draw (8,0) to (8,4);
		\draw (0,4) to (8,4);
		\node[minimum size=0.5cm] (C) at (8,-.25) {$T$};
		\draw (0,3.1) to (8,3.1);
		\draw (0,1.8) to (8,1.8);
		\draw (0,1) to (8,1);
		\draw[decorate, decoration={brace}, xshift=-2ex]  (0,0) -- node[left=0.5ex] {$L_1$}  (0,1);
		\draw[decorate, decoration={brace}, xshift=-2ex]  (0,1) -- node[left=0.5ex] {$L_2$}  (0,1.8);
		\draw[decorate, decoration={brace}, xshift=-2ex]  (0,1.8) -- node[left=0.4ex] {$L_3$}  (0,3.1);
		\draw[decorate, decoration={brace}, xshift=-2ex]  (0,3.1) -- node[left=0.4ex] {$L_4$}  (0,4);
		\node[minimum size=0.5cm] (D) at (8.75,3.55) {$k=4$};
		\node[minimum size=0.5cm] (E) at (8.75,2.45) {$k=3$};
		\node[minimum size=0.5cm] (F) at (8.75,1.4) {$k=2$};
		\node[minimum size=0.5cm] (G) at (8.75,0.5) {$k=1$};
		\draw[line width=.25pt] (0,3.55) to (.9,3.1);
		\draw[line width=2pt] (0,4) to (1.8,3.1);
		\draw[line width=.25pt] (.9,4) to (2.7,3.1);
		\draw[line width=.25pt] (1.8,4) to (3.6,3.1);
		\draw[line width=.25pt] (2.7,4) to (4.5,3.1);
		\draw[line width=.25pt] (3.6,4) to (5.4,3.1);
		\draw[line width=.25pt] (4.5,4) to (6.3,3.1);
		\draw[line width=.25pt] (5.4,4) to (7.2,3.1);
		\draw[line width=.25pt] (6.3,4) to (8,3.15);
		\draw[line width=.25pt] (7.2,4) to (8,3.6);
		\draw[line width=.25pt] (0,2.45) to (0.65,1.8);
		\draw[line width=2pt] (0,3.1) to (1.3,1.8);
		\draw[line width=.25pt] (.9,3.1) to (2.2,1.8);
		\draw[line width=2pt] (1.8,3.1) to (3.1,1.8);
		\draw[line width=.25pt] (2.7,3.1) to (4,1.8);
		\draw[line width=.25pt] (3.6,3.1) to (4.9,1.8);
		\draw[line width=.25pt] (4.5,3.1) to (5.8,1.8);
		\draw[line width=.25pt] (5.4,3.1) to (6.7,1.8);
		\draw[line width=.25pt] (6.3,3.1) to (7.6,1.8);
		\draw[line width=.25pt] (7.2,3.1) to (8,2.3);
		\draw[line width=.25pt] (0,1.4) to (0.6,1);
		\draw[line width=2pt] (0,1.8) to (1.2,1);
		\draw[line width=.25pt] (0.65,1.8) to (1.85,1);
		\draw[line width=2pt] (1.3,1.8) to (2.5,1);
		\draw[line width=.25pt] (2.2,1.8) to (3.4,1);
		\draw[line width=2pt] (3.1,1.8) to (4.3,1);
		\draw[line width=.25pt] (4,1.8) to (5.2,1);
		\draw[line width=.25pt] (4.9,1.8) to (6.1,1);
		\draw[line width=.25pt] (5.8,1.8) to (7,1);
		\draw[line width=.25pt] (6.7,1.8) to (7.9,1);
		\draw[line width=.25pt] (7.6,1.8) to (8,1.53);
		\draw[line width=.25pt] (0,0.5) to (0.5,0);
		\draw[line width=2pt] (0,1) to (1,0);
		\draw[line width=.25pt] (0.6,1) to (1.6,0);
		\draw[line width=2pt] (1.2,1) to (2.2,0);
		\draw[line width=.25pt] (1.85,1) to (2.85,0);
		\draw[line width=2pt] (2.5,1) to (3.5,0);
		\draw[line width=.25pt] (3.4,1) to (4.4,0);
		\draw[line width=2pt] (4.3,1) to (5.3,0);
		\node[minimum size=0.5cm] (E) at (5.3,-0.5) {$t^*$};
		\draw[line width=.25pt] (5.2,1) to (6.2,0);
		\draw[line width=.25pt] (6.1,1) to (7.1,0);
		\draw[line width=.25pt] (7,1) to (8,0);
		\draw[line width=.25pt] (7.9,1) to (8,0.9);
		\draw[->, line width=1pt] (0.1,-0.25) to node[below] {$\ell$} (5,-0.25);
	\end{tikzpicture}
	\caption{Characteristics of (\ref{eq:lineaScalarModelOnLinearGraphWithOutflows}) on a graph with $4$ edges}
	\label{fig:characteristicsOnFourEdges}
\end{figure}

\begin{remark}
In the analytical solution of (\ref{eq:lineaScalarModelOnLinearGraphWithOutflows}) we only distinguish if the solution depends on the initial or the boundary condition, depending on the time and the location in the pipe. So for the solution of edge $4$ in \hyperref[fig:characteristicsOnFourEdges]{\textit{Figure \ref*{fig:characteristicsOnFourEdges}}}, we distinguish between
\begin{equation*}
	x \geq L_4 + d_4 t \quad \text{and} \quad x < L_4 + d_4 t.
\end{equation*}
That means, at the beginning of edge 4 (for $x = 0$), for $0 \leq t \leq -\frac{L_4}{d_4}$, the solution of edge $4$ depends on the initial condition of edge $4$. From this, it follows, that the solution of edge $3$ in \hyperref[fig:characteristicsOnFourEdges]{\textit{Figure \ref*{fig:characteristicsOnFourEdges}}} depends only on the initial condition of edge $3$ for
\begin{equation*}
	x < L_3 + d_3 t.
\end{equation*}
It depends on the boundary conditions of edge $3$ and the initial condition of edge $4$ (due to the coupling condition) for
\begin{equation*}
	L_3 + d_3 t \leq x < L_3 + d_3 t + d_3 \frac{L_4}{d_4},
\end{equation*}
and it depends on the boundary conditions of edge $3$ and edge $4$ for
\begin{equation*}
	x \geq L_3 + d_3 t + d_3 \frac{L_4}{d_4}.
\end{equation*}
This leads to the differentiation in \hyperref[theorem:scalarModelOnLinearGraph]{\textit{Theorem \ref*{theorem:scalarModelOnLinearGraph}}} in the case $x < d_k t + d_k \sum_{j=k}^n \frac{L_j}{d_j}$ ($k = 1, \cdots, n$).
\end{remark}

\begin{remark}
With the result of \hyperref[theorem:scalarModelOnLinearGraph]{\textit{Theorem \ref*{theorem:scalarModelOnLinearGraph}}} one can also derive analytical solutions of (\ref{eq:lineaScalarModelOnLinearGraphWithOutflows}) for tree-structured graphs, but one has to take into account, that the flow at the end of an edge (due to coupling conditions) can depend on more than one outgoing edges. That means the solution on a tree-structured graph is basically the sum over all paths of the solution stated in \hyperref[theorem:scalarModelOnLinearGraph]{\textit{Theorem \ref*{theorem:scalarModelOnLinearGraph}}}.
\end{remark}

\begin{proof}[\textbf{Proof of \hyperref[theorem:scalarModelOnLinearGraph]{\textit{Theorem \ref*{theorem:scalarModelOnLinearGraph}}}.}]
We define the following functions:
\begin{equation*} \begin{aligned}
	\alpha_{k,i}(x) &:= m_k \frac{x}{d_k} - \sum_{j=k}^i m_j \frac{L_j}{d_j}, \\
	\beta_{k,i}(t,x) &:= t - \frac{x}{d_k} + \sum_{j=k}^i \frac{L_j}{d_j}, \\
	\gamma_{k,\ell}(t,x) &:= m_\ell t - (m_\ell - m_k) \frac{x}{d_k} + \sum_{j=k}^{\ell-1} (m_\ell - m_j) \frac{L_j}{d_j}, \\
	\delta_{k,\ell}(t,x) &:= -d_\ell t + d_\ell \frac{x}{d_k} - d_\ell \sum_{j=k}^{\ell-1} \frac{L_j}{d_j}.
\end{aligned} \end{equation*}

We consider the $k$-th edge in a linear graph with $n$ edges ($k \in \{1, \cdots, n\}$.
\paragraph*{\textbf{Step 1: The PDE in (\ref{eq:lineaScalarModelOnLinearGraphWithOutflows}) holds: \\}}
For $x \geq d_k t + d_k \sum_{j=k}^n \frac{L_j}{d_j}$ we have
\begin{equation*}
	\frac{\partial}{\partial t} r_k(t,x) = \sum_{i=k}^n \exp \left( \alpha_{k,i}(x) \right) b'_i \left( \beta_{k,i}(t,x) \right)
\end{equation*}
and
\begin{equation*} \begin{aligned}
	d_k \frac{\partial}{\partial x}	r_k(t,x) &= d_k \sum_{i=k}^n \exp \left( \alpha_{k,i}(x) \right) \frac{m_k}{d_k} b_i \left( \beta_{k,i}(t,x) \right) \\
	&+ d_k \sum_{i=k}^n \exp \left( \alpha_{k,i}(x) \right) b'_i \left( \beta_{k,i}(t,x) \right) \left(-\frac{1}{d_k}\right).
\end{aligned} \end{equation*}
Thus it follows
\begin{equation*}
	\frac{\partial}{\partial t} r_k(t,x) + d_k \frac{\partial}{\partial x} r_k(t,x) = m_k \sum_{i=k}^n \exp \left( \alpha_{k,i}(x) \right) b_i \left( \beta_{k,i}(t,x) \right) = m_k r_k(t,x).
\end{equation*}
So the PDE in the system (\ref{eq:lineaScalarModelOnLinearGraphWithOutflows}) holds in the marked area in \hyperref[fig:characteristicsOnFourEdgesAreas]{\textit{Figure \ref*{fig:characteristicsOnFourEdgesAreas} (a)}}. For $x < d_k t + d_k \sum_{j=k}^n \frac{L_j}{d_j}$ and $\ell \in \{k, \cdots, n\}$, we have
\begin{equation*} \begin{aligned}
	\frac{\partial}{\partial t} r_k(t,x) = \sum_{i=k}^{\ell-1} \exp \left( \alpha_{k,i}(x) \right) b'_i \left( \beta_{k,i}(t,x) \right) &+ \exp \left( \gamma_{k,\ell}(t,x) \right) m_\ell\ r_{\ell,0} \left( \delta_{k,\ell}(t,x) \right) \\
	&+ \exp \left( \gamma_{k,\ell}(t,x) \right) r'_{\ell,0} \left( \delta_{k,\ell}(t,x) \right) (-d_\ell)
\end{aligned} \end{equation*}
and
\begin{equation*} \begin{aligned}
	d_k \frac{\partial}{\partial x} &=&& d_k \sum_{i=k}^{\ell-1} \exp \left( \alpha_{k,i}(x) \right) \frac{m_k}{d_k} b_i \left( \beta_{k,i}(t,x) \right) \\
	& &&+ d_k \sum_{i=k}^{\ell-1} \exp \left( \alpha_{k,i}(x) \right) b'_i \left( \beta_{k,i}(t,x) \right) \left(-\frac{1}{d_k}\right) \\
	& &&+ d_k \exp \left( \gamma_{k,\ell}(t,x) \right) \left( \frac{-m_\ell + m_k}{d_k} \right) r_{\ell,0} \left( \delta_{k,\ell}(t,x) \right) \\
	& &&+ d_k \exp \left( \gamma_{k,\ell}(t,x) \right) r'_{\ell,0} \left( \delta_{k,\ell}(t,x) \right) \frac{d_\ell}{d_k}.
\end{aligned} \end{equation*}
It follows
\begin{equation*} \begin{aligned}
	\frac{\partial}{\partial t} r_k(t,x) + d_k \frac{\partial}{\partial x} r_k(t,x) &=&& m_k \sum_{i=k}^{\ell-1} \exp \left( \alpha_{k,i}(x) \right) b_i \left( \beta_{k,i}(t,x) \right) \\
	& &&+ m_k \exp \left( \gamma_{k,\ell}(t,x) \right) r_{\ell,0} \left( \delta_{k,\ell}(t,x) \right) \\
	&=&& m_k r_k(t,x),
\end{aligned} \end{equation*}
and the PDE in system (\ref{eq:lineaScalarModelOnLinearGraphWithOutflows}) also holds in the marked area in \hyperref[fig:characteristicsOnFourEdgesAreas]{\textit{Figure \ref*{fig:characteristicsOnFourEdgesAreas} (b)}}.

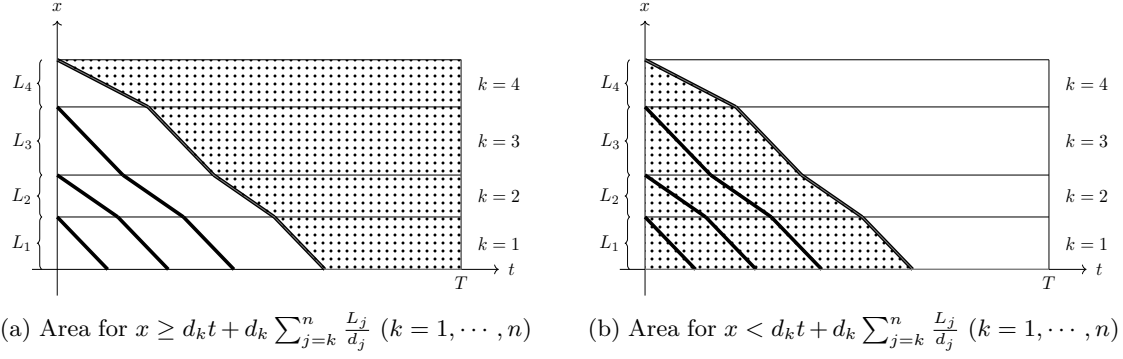
\begin{figure}[t]
	\begin{subfigure}[c]{7cm}
		\centering
		\resizebox{7cm}{4cm}{
			\begin{tikzpicture}
				\node[minimum size=0.5cm] (A) at (9,0) {$t$};
				\node[minimum size=0.5cm] (B) at (0,5) {$x$};
				\draw[->] (-.5,0) to (A);
				\draw[->] (0,-.5) to (B);
				\draw (8,0) to (8,4);
				\draw (0,4) to (8,4);
				\node[minimum size=0.5cm] (C) at (8,-.25) {$T$};
				\draw (0,3.1) to (8,3.1);
				\draw (0,1.8) to (8,1.8);
				\draw (0,1) to (8,1);
				\draw[decorate, decoration={brace}, xshift=-2ex]  (0,0) -- node[left=0.5ex] {$L_1$}  (0,1);
				\draw[decorate, decoration={brace}, xshift=-2ex]  (0,1) -- node[left=0.5ex] {$L_2$}  (0,1.8);
				\draw[decorate, decoration={brace}, xshift=-2ex]  (0,1.8) -- node[left=0.4ex] {$L_3$}  (0,3.1);
				\draw[decorate, decoration={brace}, xshift=-2ex]  (0,3.1) -- node[left=0.4ex] {$L_4$}  (0,4);
				\node[minimum size=0.5cm] (D) at (8.75,3.55) {$k=4$};
				\node[minimum size=0.5cm] (E) at (8.75,2.45) {$k=3$};
				\node[minimum size=0.5cm] (F) at (8.75,1.4) {$k=2$};
				\node[minimum size=0.5cm] (G) at (8.75,0.5) {$k=1$};
				\draw[line width=2pt] (0,4) to (1.8,3.1);
				\draw[line width=2pt] (0,3.1) to (1.3,1.8);
				\draw[line width=2pt] (1.8,3.1) to (3.1,1.8);
				\draw[line width=2pt] (0,1.8) to (1.2,1);
				\draw[line width=2pt] (1.3,1.8) to (2.5,1);
				\draw[line width=2pt] (3.1,1.8) to (4.3,1);
				\draw[line width=2pt] (0,1) to (1,0);
				\draw[line width=2pt] (1.2,1) to (2.2,0);
				\draw[line width=2pt] (2.5,1) to (3.5,0);
				\draw[line width=2pt] (4.3,1) to (5.3,0);
				\draw[gray, postaction={pattern=dots}] (0,4) -- (1.8,3.1) -- (3.1,1.8) -- (4.3,1) -- (5.3,0) -- (8,0) -- (8,4);
			\end{tikzpicture}
		}
		\caption{Area for $x \geq d_k t + d_k \sum_{j=k}^n \frac{L_j}{d_j}$ ($k = 1, \cdots, n$)}
	\end{subfigure}
	\hspace{.5cm}
	\begin{subfigure}[c]{7cm}
		\centering
		\resizebox{7cm}{4cm}{
			\begin{tikzpicture}
				\node[minimum size=0.5cm] (A) at (9,0) {$t$};
				\node[minimum size=0.5cm] (B) at (0,5) {$x$};
				\draw[->] (-.5,0) to (A);
				\draw[->] (0,-.5) to (B);
				\draw (8,0) to (8,4);
				\draw (0,4) to (8,4);
				\node[minimum size=0.5cm] (C) at (8,-.25) {$T$};
				\draw (0,3.1) to (8,3.1);
				\draw (0,1.8) to (8,1.8);
				\draw (0,1) to (8,1);
				\draw[decorate, decoration={brace}, xshift=-2ex]  (0,0) -- node[left=0.5ex] {$L_1$}  (0,1);
				\draw[decorate, decoration={brace}, xshift=-2ex]  (0,1) -- node[left=0.5ex] {$L_2$}  (0,1.8);
				\draw[decorate, decoration={brace}, xshift=-2ex]  (0,1.8) -- node[left=0.4ex] {$L_3$}  (0,3.1);
				\draw[decorate, decoration={brace}, xshift=-2ex]  (0,3.1) -- node[left=0.4ex] {$L_4$}  (0,4);
				\node[minimum size=0.5cm] (D) at (8.75,3.55) {$k=4$};
				\node[minimum size=0.5cm] (E) at (8.75,2.45) {$k=3$};
				\node[minimum size=0.5cm] (F) at (8.75,1.4) {$k=2$};
				\node[minimum size=0.5cm] (G) at (8.75,0.5) {$k=1$};
				\draw[line width=2pt] (0,4) to (1.8,3.1);
				\draw[line width=2pt] (0,3.1) to (1.3,1.8);
				\draw[line width=2pt] (1.8,3.1) to (3.1,1.8);
				\draw[line width=2pt] (0,1.8) to (1.2,1);
				\draw[line width=2pt] (1.3,1.8) to (2.5,1);
				\draw[line width=2pt] (3.1,1.8) to (4.3,1);
				\draw[line width=2pt] (0,1) to (1,0);
				\draw[line width=2pt] (1.2,1) to (2.2,0);
				\draw[line width=2pt] (2.5,1) to (3.5,0);
				\draw[line width=2pt] (4.3,1) to (5.3,0);
				\draw[gray, postaction={pattern=dots}] (0,4) -- (1.8,3.1) -- (3.1,1.8) -- (4.3,1) -- (5.3,0) -- (8,0) -- (0,0);
			\end{tikzpicture}
		}
		\caption{Area for $x < d_k t + d_k \sum_{j=k}^n \frac{L_j}{d_j}$ ($k = 1, \cdots, n$)}
	\end{subfigure}
	\caption{Areas in which the PDE of system (\ref{eq:lineaScalarModelOnLinearGraphWithOutflows}) holds}
	\label{fig:characteristicsOnFourEdgesAreas}
\end{figure}

\paragraph*{\textbf{Step 2: The initial conditions in (\ref{eq:lineaScalarModelOnLinearGraphWithOutflows}) hold:} \\}
Next we show, that the initial conditions hold. For $x < d_k t + d_k \sum_{j=k}^n \frac{L_j}{d_j}$ and $\ell = k$, we have
\begin{equation*}
	r_k(0,x) = \sum_{i=k}^{k-1} \exp \left( \alpha_{k,i}(x) \right) b_i \left( \beta_{k,i}(0,x) \right) + \exp \left( \gamma_{k,k}(0,x) \right) r_{k,0} \left( \delta_{k,k}(0,x) \right).
\end{equation*}
Since sums from $k$ to $k-1$ are equal to $0$, this leads to $\gamma_{k,k}(0,x) = 0$ and $\delta_{k,k}(0,x) = x$. Thus the initial conditions are fulfilled (see \hyperref[fig:characteristicsOnFourEdgesBoundaryInitial]{\textit{Figure \ref*{fig:characteristicsOnFourEdgesBoundaryInitial} (a)}}).

\paragraph*{\textbf{Step 3: The boundary conditions in (\ref{eq:lineaScalarModelOnLinearGraphWithOutflows}) hold:} \\}
For checking the boundary condition we consider (i.e. $k = n$ and $x \geq d_n t + L_n$), we have
\begin{equation*}  \begin{aligned}
	r_n(t,L_n) &= \sum_{i=n}^n \exp \left( \alpha_{n,i}(L_n) \right) b_i(\beta_{n,i}(t,L_n)) \\ &= \exp \left( \alpha_{n,n}(L_n) \right) b_n(\beta_{n,n}(t,L_n)) = b_n(t),
\end{aligned} \end{equation*}
since $\alpha_{n,n}(L_n) = 1$ and $\beta_{n,n}(t,L_n) = t$ (see \hyperref[fig:characteristicsOnFourEdgesBoundaryInitial]{\textit{Figure \ref*{fig:characteristicsOnFourEdgesBoundaryInitial} (b)}}).

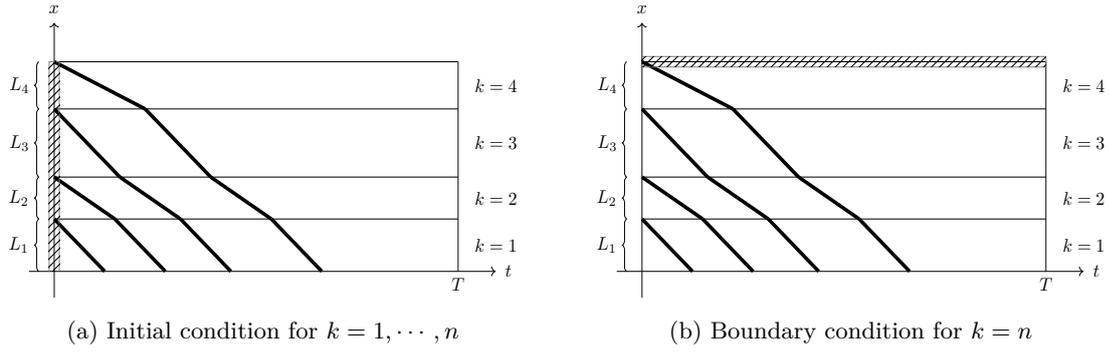
\begin{figure}[t]
	\begin{subfigure}[c]{7cm}
		\centering
		\resizebox{7cm}{4cm}{
			\begin{tikzpicture}
				\node[minimum size=0.5cm] (A) at (9,0) {$t$};
				\node[minimum size=0.5cm] (B) at (0,5) {$x$};
				\draw[->] (-.5,0) to (A);
				\draw[->] (0,-.5) to (B);
				\draw (8,0) to (8,4);
				\draw (0,4) to (8,4);
				\node[minimum size=0.5cm] (C) at (8,-.25) {$T$};
				\draw (0,3.1) to (8,3.1);
				\draw (0,1.8) to (8,1.8);
				\draw (0,1) to (8,1);
				\draw[decorate, decoration={brace}, xshift=-2ex]  (0,0) -- node[left=0.5ex] {$L_1$}  (0,1);
				\draw[decorate, decoration={brace}, xshift=-2ex]  (0,1) -- node[left=0.5ex] {$L_2$}  (0,1.8);
				\draw[decorate, decoration={brace}, xshift=-2ex]  (0,1.8) -- node[left=0.4ex] {$L_3$}  (0,3.1);
				\draw[decorate, decoration={brace}, xshift=-2ex]  (0,3.1) -- node[left=0.4ex] {$L_4$}  (0,4);
				\node[minimum size=0.5cm] (D) at (8.75,3.55) {$k=4$};
				\node[minimum size=0.5cm] (E) at (8.75,2.45) {$k=3$};
				\node[minimum size=0.5cm] (F) at (8.75,1.4) {$k=2$};
				\node[minimum size=0.5cm] (G) at (8.75,0.5) {$k=1$};
				\draw[line width=2pt] (0,4) to (1.8,3.1);
				\draw[line width=2pt] (0,3.1) to (1.3,1.8);
				\draw[line width=2pt] (1.8,3.1) to (3.1,1.8);
				\draw[line width=2pt] (0,1.8) to (1.2,1);
				\draw[line width=2pt] (1.3,1.8) to (2.5,1);
				\draw[line width=2pt] (3.1,1.8) to (4.3,1);
				\draw[line width=2pt] (0,1) to (1,0);
				\draw[line width=2pt] (1.2,1) to (2.2,0);
				\draw[line width=2pt] (2.5,1) to (3.5,0);
				\draw[line width=2pt] (4.3,1) to (5.3,0);
				\draw[gray, postaction={pattern=north east lines}] (-.1,0) -- (-.1,4) -- (.1,4) -- (.1,0);
			\end{tikzpicture}
		}
		\caption{Initial condition for $k = 1, \cdots, n$}
	\end{subfigure}
	\hspace{.5cm}
	\begin{subfigure}[c]{7cm}
		\centering
		\resizebox{7cm}{4cm}{
			\begin{tikzpicture}
				\node[minimum size=0.5cm] (A) at (9,0) {$t$};
				\node[minimum size=0.5cm] (B) at (0,5) {$x$};
				\draw[->] (-.5,0) to (A);
				\draw[->] (0,-.5) to (B);
				\draw (8,0) to (8,4);
				\draw (0,4) to (8,4);
				\node[minimum size=0.5cm] (C) at (8,-.25) {$T$};
				\draw (0,3.1) to (8,3.1);
				\draw (0,1.8) to (8,1.8);
				\draw (0,1) to (8,1);
				\draw[decorate, decoration={brace}, xshift=-2ex]  (0,0) -- node[left=0.5ex] {$L_1$}  (0,1);
				\draw[decorate, decoration={brace}, xshift=-2ex]  (0,1) -- node[left=0.5ex] {$L_2$}  (0,1.8);
				\draw[decorate, decoration={brace}, xshift=-2ex]  (0,1.8) -- node[left=0.4ex] {$L_3$}  (0,3.1);
				\draw[decorate, decoration={brace}, xshift=-2ex]  (0,3.1) -- node[left=0.4ex] {$L_4$}  (0,4);
				\node[minimum size=0.5cm] (D) at (8.75,3.55) {$k=4$};
				\node[minimum size=0.5cm] (E) at (8.75,2.45) {$k=3$};
				\node[minimum size=0.5cm] (F) at (8.75,1.4) {$k=2$};
				\node[minimum size=0.5cm] (G) at (8.75,0.5) {$k=1$};
				\draw[line width=2pt] (0,4) to (1.8,3.1);
				\draw[line width=2pt] (0,3.1) to (1.3,1.8);
				\draw[line width=2pt] (1.8,3.1) to (3.1,1.8);
				\draw[line width=2pt] (0,1.8) to (1.2,1);
				\draw[line width=2pt] (1.3,1.8) to (2.5,1);
				\draw[line width=2pt] (3.1,1.8) to (4.3,1);
				\draw[line width=2pt] (0,1) to (1,0);
				\draw[line width=2pt] (1.2,1) to (2.2,0);
				\draw[line width=2pt] (2.5,1) to (3.5,0);
				\draw[line width=2pt] (4.3,1) to (5.3,0);
				\draw[gray, postaction={pattern=north east lines}] (0,4.1) -- (8,4.1) -- (8,3.9) -- (0,3.9);
			\end{tikzpicture}
		}
		\caption{Boundary condition for $k=n$}
	\end{subfigure}
	\caption{Areas in which the initial and the boundary conditions of (\ref{eq:lineaScalarModelOnLinearGraphWithOutflows}) hold}
	\label{fig:characteristicsOnFourEdgesBoundaryInitial}
\end{figure}

\paragraph*{\textbf{Step 4: The coupling conditions in (\ref{eq:lineaScalarModelOnLinearGraphWithOutflows}) hold:} \\}
Finally, we have to check the coupling conditions. For $k = 1, \cdots, n-1$ and $x \geq d_k t + d_k \sum_{j=k}^n \frac{L_j}{d_j}$ it is $\alpha_{k,i}(L_k) = \alpha_{k+1,i}(0)$ and $\beta_{k,i}(t,L_k) = \beta_{k+1,i}(t,0)$. Thus we have
\begin{equation*} \begin{aligned}
	r_k(t,L_k) &=&& \sum_{i=k}^n \exp \left( \alpha_{k,i}(L_k) \right) b_i \left( \beta_{k,i}(t,L_k) \right) \\
	&=&& \sum_{i=k}^n \exp \left( \alpha_{k+1,i}(0) \right) b_i \left( \beta_{k+1,i}(t,0) \right) \\
	&=&& \exp \left( \alpha_{k+1,k}(0) \right) b_k \left( \beta_{k+1,k}(t,0) \right)  \\
	& &&+ \sum_{i=k+1}^n \exp \left( \alpha_{k+1,i}(0) \right) b_i \left( \beta_{k+1,i}(t,0) \right) \\
	&=&& b(t) + r_{k+1}(t,0),
\end{aligned} \end{equation*}
and the coupling conditions are fulfilled (see \hyperref[fig:characteristicsOnFourEdgesCoupling]{\textit{Figure \ref*{fig:characteristicsOnFourEdgesCoupling} (a)}}). For $k = 1, \cdots, n-1$, $x < d_k t + d_k \sum_{j=k}^n \frac{L_j}{d_j}$ and $\ell \in \{k+1, \cdots, n\}$ we have $\gamma_{k,\ell}(t,L_k) = \gamma_{k+1,\ell}(t,0)$ and $\delta_{k,\ell}(t,L_k) = \delta_{k+1,\ell}(t,0)$. It follows
\begin{equation*} \begin{aligned}
	r_k(t,L_k) &=&& \sum_{i=k}^{\ell-1} \exp \left( \alpha_{k,i}(L_k) \right) b_i \left( \beta_{k,i}(t,L_k) \right) + \exp \left( \gamma_{k,\ell}(t,L_k) \right) r_{\ell,0} \left( \delta_{k,\ell}(t,L_k) \right) \\
	&=&& b_k(t) \sum_{i=k+1}^{\ell-1} \exp \left( \alpha_{k+1,i}(0) \right) b_i \left( \beta_{k+1,i}(t,0) \right) \\
	& &&+ \exp \left( \gamma_{k+1,\ell}(t,0) \right) r_{\ell,0} \left( \delta_{k+1,\ell}(t,0) \right) \\
	&=&& b_k(t) + r_{k+1}(t,0).
\end{aligned} \end{equation*}
So the coupling conditions also hold in this case (see \hyperref[fig:characteristicsOnFourEdgesCoupling]{\textit{Figure \ref*{fig:characteristicsOnFourEdgesCoupling} (b)}}) and the theorem is proven.

\begin{figure}[htbp]
	\begin{subfigure}[c]{7cm}
		\centering
		\resizebox{7cm}{4cm}{
			\begin{tikzpicture}
				\node[minimum size=0.5cm] (A) at (9,0) {$t$};
				\node[minimum size=0.5cm] (B) at (0,5) {$x$};
				\draw[->] (-.5,0) to (A);
				\draw[->] (0,-.5) to (B);
				\draw (8,0) to (8,4);
				\draw (0,4) to (8,4);
				\node[minimum size=0.5cm] (C) at (8,-.25) {$T$};
				\draw (0,3.1) to (8,3.1);
				\draw (0,1.8) to (8,1.8);
				\draw (0,1) to (8,1);
				\draw[decorate, decoration={brace}, xshift=-2ex]  (0,0) -- node[left=0.5ex] {$L_1$}  (0,1);
				\draw[decorate, decoration={brace}, xshift=-2ex]  (0,1) -- node[left=0.5ex] {$L_2$}  (0,1.8);
				\draw[decorate, decoration={brace}, xshift=-2ex]  (0,1.8) -- node[left=0.4ex] {$L_3$}  (0,3.1);
				\draw[decorate, decoration={brace}, xshift=-2ex]  (0,3.1) -- node[left=0.4ex] {$L_4$}  (0,4);
				\node[minimum size=0.5cm] (D) at (8.75,3.55) {$k=4$};
				\node[minimum size=0.5cm] (E) at (8.75,2.45) {$k=3$};
				\node[minimum size=0.5cm] (F) at (8.75,1.4) {$k=2$};
				\node[minimum size=0.5cm] (G) at (8.75,0.5) {$k=1$};
				\draw[line width=2pt] (0,4) to (1.8,3.1);
				\draw[line width=2pt] (0,3.1) to (1.3,1.8);
				\draw[line width=2pt] (1.8,3.1) to (3.1,1.8);
				\draw[line width=2pt] (0,1.8) to (1.2,1);
				\draw[line width=2pt] (1.3,1.8) to (2.5,1);
				\draw[line width=2pt] (3.1,1.8) to (4.3,1);
				\draw[line width=2pt] (0,1) to (1,0);
				\draw[line width=2pt] (1.2,1) to (2.2,0);
				\draw[line width=2pt] (2.5,1) to (3.5,0);
				\draw[line width=2pt] (4.3,1) to (5.3,0);
				\draw[gray, postaction={pattern=north east lines}] (1.8,3) -- (1.8,3.2) -- (8,3.2) -- (8,3);
				\draw[gray, postaction={pattern=north east lines}] (3.1,1.7) -- (3.1,1.9) -- (8,1.9) -- (8,1.7);
				\draw[gray, postaction={pattern=north east lines}] (4.3,.9) -- (4.3,1.1) -- (8,1.1) -- (8,.9);
			\end{tikzpicture}
		}
		\caption{Coupling conditions for $x \geq d_k t + d_k \sum_{j=k}^n \frac{L_j}{d_j}$ ($k = 1, \cdots, n-1$)}
	\end{subfigure}
	\hspace{.5cm}
	\begin{subfigure}[c]{7cm}
		\centering
		\resizebox{7cm}{4cm}{
			\begin{tikzpicture}
				\node[minimum size=0.5cm] (A) at (9,0) {$t$};
				\node[minimum size=0.5cm] (B) at (0,5) {$x$};
				\draw[->] (-.5,0) to (A);
				\draw[->] (0,-.5) to (B);
				\draw (8,0) to (8,4);
				\draw (0,4) to (8,4);
				\node[minimum size=0.5cm] (C) at (8,-.25) {$T$};
				\draw (0,3.1) to (8,3.1);
				\draw (0,1.8) to (8,1.8);
				\draw (0,1) to (8,1);
				\draw[decorate, decoration={brace}, xshift=-2ex]  (0,0) -- node[left=0.5ex] {$L_1$}  (0,1);
				\draw[decorate, decoration={brace}, xshift=-2ex]  (0,1) -- node[left=0.5ex] {$L_2$}  (0,1.8);
				\draw[decorate, decoration={brace}, xshift=-2ex]  (0,1.8) -- node[left=0.4ex] {$L_3$}  (0,3.1);
				\draw[decorate, decoration={brace}, xshift=-2ex]  (0,3.1) -- node[left=0.4ex] {$L_4$}  (0,4);
				\node[minimum size=0.5cm] (D) at (8.75,3.55) {$k=4$};
				\node[minimum size=0.5cm] (E) at (8.75,2.45) {$k=3$};
				\node[minimum size=0.5cm] (F) at (8.75,1.4) {$k=2$};
				\node[minimum size=0.5cm] (G) at (8.75,0.5) {$k=1$};
				\draw[line width=2pt] (0,4) to (1.8,3.1);
				\draw[line width=2pt] (0,3.1) to (1.3,1.8);
				\draw[line width=2pt] (1.8,3.1) to (3.1,1.8);
				\draw[line width=2pt] (0,1.8) to (1.2,1);
				\draw[line width=2pt] (1.3,1.8) to (2.5,1);
				\draw[line width=2pt] (3.1,1.8) to (4.3,1);
				\draw[line width=2pt] (0,1) to (1,0);
				\draw[line width=2pt] (1.2,1) to (2.2,0);
				\draw[line width=2pt] (2.5,1) to (3.5,0);
				\draw[line width=2pt] (4.3,1) to (5.3,0);
				\draw[gray, postaction={pattern=north east lines}] (0,3.2) -- (0,3) -- (1.8,3) -- (1.8,3.2);
				\draw[gray, postaction={pattern=north east lines}] (0,1.9) -- (0,1.7) -- (3.1,1.7) -- (3.1,1.9);
				\draw[gray, postaction={pattern=north east lines}] (0,1.1) -- (0,.9) -- (4.3,.9) -- (4.3,1.1);
			\end{tikzpicture}
		}
		\caption{Coupling conditions for $x < d_k t + d_k \sum_{j=k}^n \frac{L_j}{d_j}$ ($k = 1, \cdots, n-1$)}
	\end{subfigure}
	\caption{Areas in which the coupling conditions of (\ref{eq:lineaScalarModelOnLinearGraphWithOutflows}) hold}
	\label{fig:characteristicsOnFourEdgesCoupling}
\end{figure}
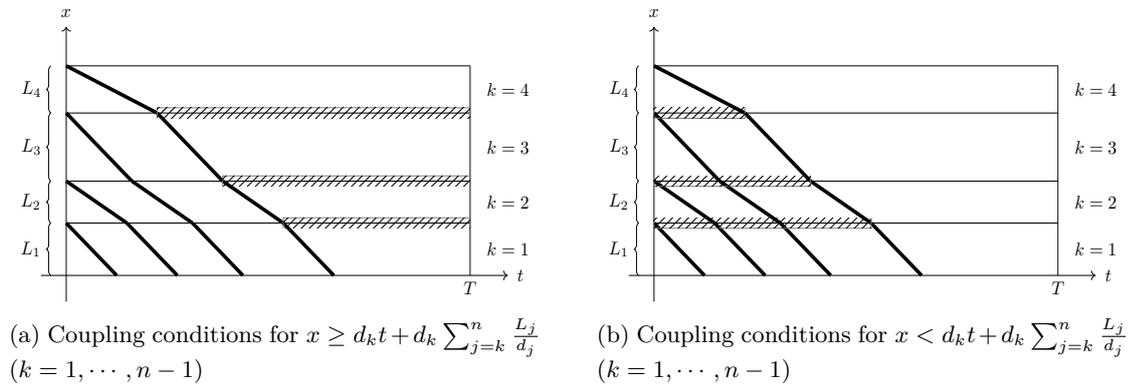

\end{proof}

In the next section, we consider model (\ref{eq:lineaScalarModelOnLinearGraphWithOutflows}) with uncertain boundary data.

\subsection{Stochastic loads for a scalar PDE}

The model (\ref{eq:lineaScalarModelOnLinearGraphWithOutflows}) describes the distribution of the water contamination in a network which is contaminated by the consumers at the nodes except $v_0$. At the end of the network (at node $v_0$), there are restrictions on the contamination rate. But because the contamination rate at the nodes cannot be known a priori, it can be seen as random. Of course one can expect a certain value from statistics or measurements, but this value is never exact. Therefore we use the random boundary data in a Fourier series representation, which was introduced before. In this section we first fix a point in time $t^* \in [0,T]$ and compute the probability, that the solution of (\ref{eq:lineaScalarModelOnLinearGraphWithOutflows}) with random boundary data satisfies box constraints at single point in time $t^* \in [0,T]$. Then we generalize this approach and compute the probability, that the box constraints are satisfied for all times $t \in [0,T]$ (cf. (\ref{eq:timeDependentProbabilisticConstraint})). \\

For $m \in \mathbb{N}_0$, we consider the Gaussian distributed random variables $a_m \sim \mathcal{N}(\mathbb{1}_n, \Sigma)$ with mean value $\mathbb{1}_n \in \mathbb{R}_+^n$ and positive definite covariance matrix $\Sigma \in \mathbb{R}^{n \times n}$ on an appropriate probability space $(\Omega, \mathcal{A}, \mathbb{P})$. Then the random boundary data at node $v_k$ (for $k \in \{1, \cdots, n\}$) is given by
\begin{equation*}
	b_k(t,\omega) = \sum_{m=0}^{\infty} a_{m,k}(\omega) a^0_{m,k} \psi_m(t),
\end{equation*}
with coefficients
\begin{equation*}
	a^0_{m,k} := \int_0^T (b_D)_k(t) \psi_m(t) dt,
\end{equation*}
and $\psi_m$ defined in (\ref{eq:FourierBasis}). Note again, that for the implementation, we cut the series after $N_F \in \mathbb{N}$ terms, which is a good approximation of $b$ for $N_F$ large enough (see \hyperref[remark:FourierSeries]{\textit{Remark \ref*{remark:FourierSeries}}}) and as mentioned before, if $(b_D)_k \in L^2(0,T)$, then $b_k \in L^2(0,T)$ $\mathbb{P}$-almost surely. Because water cannot get cleaned at the nodes, we are only interested in positive boundary values (cf. \hyperref[sec:stationary]{\textit{Section \ref*{sec:stationary}}}). Therefore, we assume that $b_D \in L^2(0,T)$ with $b_D \geq 0$ and that the parameter $\Sigma$ of the distribution of $a_m$ are chosen s.t. the probability that $b \geq 0$ is almost $1$. In practice this can be done as follows: Let $\gamma_k^* := \text{argmin}_{t \in [0,T]} b_k^D(t)$ and let $\mathfrak{I}_{k,_+}$ resp $\mathfrak{I}_{k,-}$ be the set of indices where $a_{m,k}^0 \psi_m(\gamma_k^*) \geq 0$ resp. where $a_{m,k}^0 \psi_m(\gamma_k^*) < 0$. We split the Fourier series in positive and negative terms, it follows
\begin{equation*}
	b_k^D(\gamma_k^*) = \sum_{m \in \mathfrak{I}_{k,+}} a_{m,k}^0 \psi_m(\gamma_k^*) + \sum_{m \in \mathfrak{I}_{k,-}} a_{m,k}^0 \psi_m(\gamma_k^*).
\end{equation*}
Mention that the $a_m$ all are identically distributed and $a_{m,k}$ has the variance $\sigma^2_k$. We use that fact, that a random Gaussian number $a_{m,k}(\omega)$ is in $[1 - 3\sigma_k, 1 + 3\sigma_k]$ with probability $99.73\%$. The worst case for a random scenario with random numbers $a_{k,m}(\omega) \in [1 - 3\sigma_k, 1 + 3\sigma_k]$ would be, if the positive terms get smaller and the negative terms get larger, i.e.,
\begin{equation*} \begin{aligned}
	b_k(\gamma_k^*) &= \sum_{m \in \mathfrak{I}_{k,+}} (1 - 3\sigma_k) a_{m,k}^0 \psi_m(\gamma_k^*) + \sum_{m \in \mathfrak{I}_{k,-}} (1 + 3\sigma_k) a_{m,k}^0 \psi_m(\gamma_k^*) \\
	&= b_k^D(\gamma_k^*) - 3\sigma_k \left( \sum_{m \in \mathfrak{I}_{k,+}} a_{m,k}^0 \psi_m(\gamma_k^*) - \sum_{m \in \mathfrak{I}_{k,-}} a_{m,k}^0 \psi_m(\gamma_k^*) \right).
\end{aligned} \end{equation*}
From this it follows, that $b_k(\gamma_k^*) \geq 0$, if
\begin{equation*}
\sigma_k \leq \frac{b_k^D(\gamma_k^*)}{3 \left( \sum_{m \in \mathfrak{I}_{k,+}} a_{m,k}^0 \psi_m(\gamma_k^*) - \sum_{m \in \mathfrak{I}_{k,-}} a_{m,k}^0 \psi_m(\gamma_k^*) \right)}.
\end{equation*}
For the implementation this is a quite cheap task since the terms of the Fourier series have to be computed anyway. When we would use a truncated Gaussian distribution for the $a_{m,k}$ bounded from below by $1 - 3 \sigma_k$ and bounded from above by $1 + 3 \sigma_k$, then we could guarantee that $b_k(t, \omega)$ is non negative on $[0,T]$.
As it is mentioned before, we want the solution at a time $t^* \in [0,T]$ at the node $v_0$ to satisfy box constraints, s.t.
\begin{equation} \label{eq:boxConstraints}
	r_1(t^*,0) \in \left[ r_0^{\min}, r_0^{\max} \right].
\end{equation}
 So the full model in this subsection is given in (\ref{eq:lineaScalarModelOnLinearGraphWithOutflows}). For this model, we define the set of feasible loads as
\begin{equation} \label{eq:feasibleSetScalarDynamic}
	M(t^*) := \left\{\ b \in L^2([0,T];\mathbb{R}^n_{\geq 0})\ \left\vert\ \begin{matrix} r_k(t,x) \text{ is a solution of } (\ref{eq:lineaScalarModelOnLinearGraphWithOutflows})\ (\text{for } k = 1, \cdots, n) \\ \text{such that } r_1(t^*,0) \in \left[ r_0^{\min}, r_0^{\max} \right] \end{matrix} \right.\ \right\}.
\end{equation}
Our aim in this subsection is, for a time $t^* \in [0,T]$, to compute the probability
\begin{equation*}
	\mathbb{P}(\ b \in M(t^*)\ ),
\end{equation*}
which is the probability, that for a random boundary function $b \in L^2(0,T)$, the solution of the linear system (\ref{eq:lineaScalarModelOnLinearGraphWithOutflows}) satisfies the box constraints (\ref{eq:boxConstraints}) at a point in time $t^* \in [0,T]$. From \hyperref[theorem:scalarModelOnLinearGraph]{\textit{Theorem \ref*{theorem:scalarModelOnLinearGraph}}} we know that
\begin{equation*}
	r_1(t,0) = \begin{cases}
		\sum_{i=1}^n \exp \left( - \sum_{j=1}^i m_j \frac{L_j}{d_j} \right) b_i^\omega \left( t + \sum_{j=1}^i \frac{L_j}{d_j} \right) & t \geq - \sum_{j=1}^n \frac{L_j}{d_j}, \\[10pt]
		\sum_{i=1}^{\ell-1} \exp \left( - \sum_{j=1}^i m_j \frac{L_j}{d_j} \right) b_i^\omega \left( t + \sum_{j=1}^i \frac{L_j}{d_j} \right) & t < - \sum_{j=1}^n \frac{L_j}{d_j} \\
		+ \exp \left( m_\ell t + \sum_{j=1}^{\ell-1} (m_\ell - m_j) \frac{L_j}{d_j} \right) r_{\ell,0} \left( - d_\ell t - d_\ell \sum_{j=1}^{\ell-1} \frac{L_j}{d_j} \right) & (\ell = 1, \cdots, n), \end{cases}
\end{equation*}
where $b^\omega$ denotes the realization $b(\omega)$ of the random boundary data for $\omega \in \Omega$.
For $i = 1, \cdots, n$, we define the (time dependent) values
\begin{equation*}
	\mathcal{C}_i := \exp \left( - \sum_{j=1}^i m_j \frac{L_j}{d_j} \right),
\end{equation*}
and
\begin{equation*}
	\mathcal{C}_i^0(t) := \exp \left( m_i t + \sum_{j=1}^{i-1} (m_i - m_j)\frac{L_j}{d_j} \right) r_{i,0} \left( -d_i t - d_i \sum_{j=1}^{i-1} \frac{L_j}{d_j} \right).
\end{equation*}
Then, $b$ is feasible at time $t^* \in [0,T]$, iff
\begin{equation} \label{eq:inequalityDynamicLargeTimes}
	r_0^{\min} \leq \sum_{i=1}^n \mathcal{C}_i b_i^\omega \left( t^* + \sum_{j=1}^i \frac{L_j}{d_j} \right) \leq r_0^{\max},
\end{equation}
for $t^* \geq - \sum_{j=1}^n \frac{L_j}{d_j}$ and it is feasible, iff
\begin{equation} \label{eq:inequalityDynamicSmallTimes}
	r_0^{\min} \leq \sum_{j=1}^{\ell-1} \mathcal{C}_i b_i^\omega \left( t^* + \sum_{j=1}^i \frac{L_j}{d_j} \right) + \mathcal{C}^0_\ell(t^*) \leq r_0^{\max},
\end{equation}
for $-\sum_{j=k}^{\ell-1} \frac{L_j}{d_j} \leq t^* < - \sum_{j=1}^n \frac{L_j}{d_j}$ ($\ell \in \{1, \cdots, n\}$). Due to the distribution of the random values $a_m$ ($m = 0, 1, \cdots $), we have
\begin{equation*}
	b \sim \mathcal{N}(\mu_b(t), \Sigma_b(t)),
\end{equation*}
with $\mu_b(\cdot) \in \mathbb{R}_+^n$ and $\Sigma_b(\cdot) \in \mathbb{R}^{n \times n}$ positive definite.
To compute the desired probability for a point in time $t^* \in [0,T]$, we use the idea of the SRD. For a point $s \in \mathbb{S}^{n-1}$ at the unit sphere, we set
\begin{equation*}
	b_s(\hat{r},t) = \hat{r} \mathcal{L}_b(t) s + \mu_b(t) = \hat{r} \pi_b(t) + \mu_b(t),
\end{equation*}
with $\pi_b(t) = \mathcal{L}_b(t) v$ and $\mathcal{L}$, s.t. $\mathcal{L}_b(t) \mathcal{L}_b^\top(t) = \Sigma_b(t)$. Because we are only interested in positive boundary values, we define the regular range as
\begin{equation*}
	R_{s,\text{reg}} := \{ \hat{r} \geq 0\ \vert\ b_s(\hat{r},t) \geq 0 \quad \forall t \in [0,T]\}.
\end{equation*}
Thus, similar to the stationary case, the (time dependent) one-dimensional sets
\begin{equation*}
	M_s(t^*) = \{ \hat{r} \in R_{s,\text{reg}}\ \vert\ b_s(\hat{r},\cdot) \in M(t^*) \}
\end{equation*}
at a point in time $t^*$ can be computed by intersecting the regular range with the inequality (\ref{eq:inequalityDynamicLargeTimes}) resp. (\ref{eq:inequalityDynamicSmallTimes}). If $t^* \geq - \sum_{j=1}^n \frac{L_j}{d_j}$, then from ({\ref{eq:inequalityDynamicLargeTimes}) it follows
\begin{equation} \label{eq:SRDDynIneq1}
	r_0^{\min} \leq \sum_{i=1}^n \mathcal{C}_i \left[ \hat{r} \pi_{b,i} \left( t^* + \sum_{j=1}^i \frac{L_j}{d_j} \right) + \mu_{b,i} \left( t^* + \sum_{j=1}^i \frac{L_j}{d_j} \right) \right] \leq r_0^{\max}.
\end{equation}
Define the values
\begin{equation*}
	a_1 := \frac{ r_0^{\min} - \sum_{i=1}^n C_i\ \mu_{b,i} \left( t^* + \sum_{j=1}^i \frac{L_j}{d_j} \right) }{ \sum_{i=1}^n C_i\ \pi_{b,i} \left( t^* + \sum_{j=1}^i \frac{L_j}{d_j} \right) },
\end{equation*}
and
\begin{equation*}
	a_2 := \frac{ r_0^{\max} - \sum_{i=1}^n C_i\ \mu_{b,i} \left( t^* + \sum_{j=1}^i \frac{L_j}{d_j} \right) }{ \sum_{i=1}^n C_i\ \pi_{b,i} \left( t^* + \sum_{j=1}^i \frac{L_j}{d_j} \right) }.
\end{equation*}
Then we have
\begin{equation*}
	M_s(t^*) = R_{s,\text{reg}} \cap \begin{cases} [a_1, a_2] & \text{if } \sum_{i=1}^n C_i\ \pi_{b,i} \left( t^* + \sum_{j=1}^i \frac{L_j}{d_j} \right) \geq 0, \\ [a_2, a_1] & \text{else}. \end{cases}
\end{equation*}
If $t^* < - \sum_{j=1}^n \frac{L_j}{d_j}$ ($\ell = 1, \cdots, n$), then from (\ref{eq:inequalityDynamicSmallTimes}) it follows
\begin{equation} \label{eq:SRDDynIneq2}
	r_0^{\min} \leq \sum_{j=1}^{\ell-1} \mathcal{C}_i \left[ \hat{r} \pi_{b,i} \left( t^* + \sum_{j=1}^i \frac{L_j}{d_j} \right) + \mu_{b,i} \left( t^* + \sum_{j=1}^i \frac{L_j}{d_j} \right) \right] + \mathcal{C}^0_\ell(t^*) \leq r_0^{\max}.
\end{equation}
Define the values
\begin{equation*}
	a_1 := \frac{ r_0^{\min} - C^0_\ell(t^*) - \sum_{i=1}^{\ell-1} C_i\ \mu_{b,i} \left( t^* + \sum_{j=1}^i \frac{L_j}{d_j} \right) }{ \sum_{i=1}^{\ell-1} C_i\ \pi_{b,i} \left( t^* + \sum_{j=1}^i \frac{L_j}{d_j} \right) },
\end{equation*}
and
\begin{equation*}
	a_2 := \frac{ r_0^{\max} - C^0_\ell(t^*) - \sum_{i=1}^{\ell-1} C_i\ \mu_{b,i} \left( t^* + \sum_{j=1}^i \frac{L_j}{d_j} \right) }{ \sum_{i=1}^{\ell-1} C_i\ \pi_{b,i} \left( t^* + \sum_{j=1}^i \frac{L_j}{d_j} \right) }.
\end{equation*}
Then we have
\begin{equation*}
	M_s(t^*) = R_{s,\text{reg}} \cap \begin{cases} [a_1, a_2] & \text{if } \sum_{i=1}^{\ell-1} C_i\ \pi_{b,i} \left( t^* + \sum_{j=1}^i \frac{L_j}{d_j} \right) \geq 0, \\ [a_2, a_1] & \text{else}. \end{cases}
\end{equation*}
Thus, for every point in time $t^* \in [0,T]$, the set $M_s(t^*)$ can be represented as a union of disjoint intervals, which we can use to compute the probability for the random boundary function $b$ to be feasible, like in (\ref{eq:computeProbabilityOneEdge}). \\

Next, we approximate the probability $\mathbb{P}( b \in M(t^*))$ by using the KDE approach introduced in subsection \ref{sec:stationarySingleEdge}. We consider the stochastic equation corresponding to \eqref{eq:lineaScalarModelOnLinearGraphWithOutflows} with random boundary data. Note that this equation has also a solution $\mathbb{P}$-almost surely. We assume that the distribution of the random variable $r_1(t^*,0)$ is absolutely continuous with probability density function $\varrho_{r,t^*}$ for $t^* \in [0,T]$. Thus the point in time $t^*$ has to be large enough, so that $r_1(t^*,0)$ depends on the random boundary data as it is explained in \hyperref[remark:informationTransport]{\textit{Remark \ref*{remark:informationTransport}}}. Similar to section \ref{sec:stationarySingleEdge}, it holds
\begin{equation*}
\mathbb{P} \left(b \in M(t^*) \right) = \mathbb{P} \left( r_1(t^*,0) \in [r_0^{\min}, r_0^{\max}] \right) = \int_{r_0^{\min}}^{r_0^{\max}} \varrho_{r,t^*}(z) ~ dz	.
\end{equation*}
In order to approximate the unknown probability density function by the KDE, we need a sampling set for the random variable $r_1(t^*,0)$. Therefore, let
\begin{equation*}
	\mathcal{A}_m := \{  a_{m}^{\mathcal{S},1}, \cdots, a_{m}^{\mathcal{S},N} \}
\end{equation*}
for $m=1,\cdots,N_F$ be independent and identically distributed sampling sets of $a_m \sim \mathcal{N}(\mu, \Sigma)$ (with mean value $\mu \in \mathbb{R}_+^n$ and $\Sigma \in \mathbb{R}^{n \times n}$ positive definite). Let
\begin{equation*}
	\mathcal{B}_{\mathcal{A}} := \{ b^{\mathcal{S},1}, \cdots, b^{\mathcal{S},N} \}
\end{equation*}
with $b^{S,i} = \sum_{m=0}^{N_F} a_m^{S,i} a_m^0 \phi_m$ be the corresponding sampling of the random boundary function, where $b^{S,i} \in L^2(0,T)$ with $b^{S,i} \geq 0$ ($i = 1, \cdots, N$) $\mathbb{P}$-almost surely.
 With this and \hyperref[theorem:scalarModelOnLinearGraph]{\textit{Theorem \ref*{theorem:scalarModelOnLinearGraph}}} (or an appropriate numerical method for solving the system (\ref{eq:lineaScalarModelOnLinearGraphWithOutflows})), we define the sample
\begin{equation*}
	\mathcal{R}_* := \{ r_1(t^*,0,b^{\mathcal{S},1}), \cdots, r_1(t^*,0,b^{\mathcal{S},N}) \},
\end{equation*}
where $r_1(t^*,0,b^{\mathcal{S},i})$ ($i = 1, \cdots, N$) is the solution of (\ref{eq:lineaScalarModelOnLinearGraphWithOutflows}) at node $v_0$ at time $t^* \in [0,T]$ with boundary function $b^{S,i} \in \mathcal{B}_{\mathcal{A}}$. The samplings $\mathcal{B}_{\mathcal{A}}$ and $\mathcal{R}_*$ are also independent and identically distributed. Then for a bandwidth $h \in \mathbb{R}_+$, the probability density function $\varrho_{r,t^*}$ is approximately given by
\begin{equation*}
	\varrho_{r,t^*,N}(z) = \frac{1}{N\ h} \sum_{i=1}^N \frac{1}{\sqrt{2 \pi}} \exp \left( -\frac{1}{2} \left( \frac{z - r_1(t^*,0,b^{\mathcal{S},i})}{h} \right)^2 \right).
\end{equation*}
We choose the bandwidth according to \eqref{eq:bandwidth}. Therefore, we get the same convergence results for the KDE resp. the approximated probability as in \eqref{eq:Konv1D} resp. \eqref{Konv2D}.
So we can approximate the probability for a random boundary function $b$ to be feasible at time $t^* \in [0,T]$ by
\begin{equation} \label{eq:probabilityKDEScalarDynamic}
	\mathbb{P}(\ b \in M(t^*)\ ) \approx \int_{r_0^{\min}}^{r_0^{\max}} \varrho_{r,t^*,N}(z) dz =: \mathbb{P}_N \left( b \in M(t^*) \right)
\end{equation}

So far, the computation of the desired probability in this subsection was only for box constraints at a certain point in time $t^* \in [0,T]$, e.g., the end time $t^* = T$. As mentioned before, we are interested in box constraints for the full time period, which leads to a probabilistic constraint given in (\ref{eq:timeDependentProbabilisticConstraint}). The idea is, that $r_1(t,0)$ satisfies the box constraints for all $t \in [0,T]$, iff the maximum and the minimum value of $r_1(t,0)$ in $[0,T]$ satisfies the box constraints:
\begin{equation*} \begin{gathered}
	r_1(t,0) \in [ r_0^{\min}, r_0^{\max} ] \quad \forall t \in [0,T] \\
	\Updownarrow \\
	\begin{bmatrix} r_0^{\min} \\ r_0^{\min} \end{bmatrix} \leq \begin{bmatrix} \max_{t \in [0,T]} r_1(t,0) \\ \min_{t \in [0,T]} r_1(t,0) \end{bmatrix} \leq \begin{bmatrix} r_0^{\max} \\ r_0^{\max} \end{bmatrix}.
\end{gathered} \end{equation*}

We define the values
\begin{equation*}
	\underline{t} := \text{argmin}_{t \in [0,T]} r_1(t,0),
\end{equation*}
and
\begin{equation*}
	\overline{t} := \text{argmax}_{t \in [0,T]} r_1(t,0).
\end{equation*}
For the SRD we use a similar procedure as above. We have
\begin{equation*}
	\mathbb{P}(b \in M(t)\ \forall t \in [0,T]) \quad \Leftrightarrow \quad \mathbb{P}(b \in M(\underline{t}) \text{ and } b \in M(\overline{t}) ).
\end{equation*}
That means, we need to intersect the regular range with two inequalities of the form (\ref{eq:SRDDynIneq1}) resp. (\ref{eq:SRDDynIneq2}) to get the set $M_s(\underline{t}) \cap M_s(\overline{t})$ and to compute the desired probability. This is only possible, if one compute $\underline{t}$ and $\overline{t}$ for the deterministic boundary function $b_D(t)$. Otherwise, due to the randomness of the boundary functions $b(t)$, the \textit{argmin} and the \textit{argmax} can be shifted and thus, the $\underline{t}$ and $\overline{t}$ depend on this uncertainty. So a general $\underline{t}$ and $\overline{t}$ does not exist and it is not clear, where to evaluate the mean and the variance in (\ref{eq:SRDDynIneq1}) resp. (\ref{eq:SRDDynIneq2}). But if we use the deterministic boundary function, the $\underline{t}$ and the $\overline{t}$ does not meet the minimal and maximal values of the random boundary functions and thus, the result may not be significant. \\

The KDE uses only the solution of the model for estimating an analytical probability density function, which we can use later for the optimization. To extend the KDE to probabilistic constraints like (\ref{eq:timeDependentProbabilisticConstraint}), we consider the minimal and maximal contamination concentration at node $v_0$ in the time period $[0,T]$.
We assume that the distribution of the random vector $R:= \left( \min_{t \in [0,T]}r_1(t,0), \max_{t \in [0,T]} r_1(t,0) \right)^T$ is absolutely continuous with probability density function $\varrho_R$. Using this time independent variable, we get
\begin{equation*}
\begin{aligned}
\mathbb{P} \left(  b \in M(t)\ \forall t \in [0,T] \right)
&= \mathbb{P} \left( R \in [r_0^{\min}, r_0^{\max}] \times [r_0^{\min}, r_0^{\max}] \right) \\
&=\int_{[r_0^{\min}, r_0^{\max}] \times [r_0^{\min}, r_0^{\max}]} \varrho_R(z) ~ dz	.
\end{aligned}
\end{equation*}
We use now a two dimensional KDE like in the stationary case for tree-structured graphs. For the set $\mathcal{B}_{\mathcal{A}}$ we define the sampling set of the random variable $R$ by
\begin{equation*}
\mathcal{R} = \left\{ (\underline{r}_{1}^{S,1},\overline{r}_{1}^{S,1})^T, \cdots, (\underline{r}_{1}^{S,N},\overline{r}_{1}^{S,N})^T  \right\}
\end{equation*}
with
\begin{equation*}
	\underline{r}_1^{S,i} := \min_{t \in [0,T]} r_1(t,0,b^{\mathcal{S},i}(t)) \quad (i = 1, \cdots, N)
\end{equation*}
and
\begin{equation*}
	\overline{r}_1^{S,i} := \max_{t \in [0,T]} r_1(t,0,b^{\mathcal{S},i}(t)) \quad (i = 1, \cdots, N).
\end{equation*}
This sampling is also independent and identically distributed. Using the KDE given in \eqref{eq:multiKDE} with bandwidth \eqref{eq:multiBandwidth}, the probability density function of the minimal and maximal contamination rate at node $v_0$ in the time period $[0,T]$ is approximately given by
\begin{equation*} \begin{aligned}
	\varrho_{R,N}(z) = \frac{1}{N h_y^2 \sigma_{N,1}^{\min} \sigma_{N,1}^{\max}} \sum_{i=1}^N \frac{1}{2 \pi} \exp \left( -\frac{1}{2} \left( \frac{z_1 - \underline{r}_1^{S,i}}{h_y \sigma_{N,1}^{\min}} \right)^2 \right) \\
	\cdot \exp \left( -\frac{1}{2} \left( \frac{z_2 - \overline{r}_1^{S,i}}{h_y \sigma_{N,1}^{\max}} \right)^2 \right),
\end{aligned} \end{equation*}
where $(\sigma_{N,1}^{\min})^2$ and $(\sigma_{N,1}^{\max})^2$ are the variances of $\underline{r}_1^{S,i}$ and $\overline{r}_1^{S,i}$ ($i = 1, \cdots, N$). For this estimator we get the same convergence results for the KDE resp. the approximated probability as in \eqref{KonvMehrdim} and \eqref{ScheffeMehrdim}. Thus we approximate the desired probability as follows:
\begin{equation*}
	\mathbb{P} \left( b \in M(t) \ \forall t \in [0,T] \right) \approx \int\limits_{[r_0^{\min}, r_0^{\max}] \times [r_0^{\min}, r_0^{\max}]} \varrho_{R,N}(z) dz
	=: \mathbb{P}_N \left( b \in M(t)\ \forall t \in [0,T] \right).
\end{equation*}

\begin{remark} \label{remark:minmax}
If $\underline{r}_1^{S,i}$ or $\overline{r}_1^{S,i}$ is taken in the time period, in which the solution depends on the initial data, then the distribution function of the minimal resp. maximal contamination rate contains a discontinuity. Thus, a probability density function in the classical sense does not even exist. So one has to guarantee, that $\underline{r}_1^{S,i}$ and $\overline{r}_1^{S,i}$ is not taken in the beginning of the time period, e.g., by excluding this part from the probabilistic constraint. As it is mentioned in \hyperref[remark:informationTransport]{\textit{Remark \ref*{remark:informationTransport}}}, for times $t \geq t^*$ with
\begin{equation*}
	t^* = \sum_{j=1}^n \frac{L_j}{\vert d_j \vert}
\end{equation*}
the solution does not depend on the initial condition anymore. So instead of solving (\ref{eq:timeDependentProbabilisticConstraint}) one can solve
\begin{equation*}
	\mathbb{P}(b \in M(t)\ \forall t \in [t^*, T]) \geq \alpha.
\end{equation*}
Motivated by the application this makes sense since the initial state is either given a priori or can be chosen a priori s.t. all bounds are satisfied for small times.
\end{remark}

\label{example3}
\paragraph*{Example 3:}\ Consider the graph with one edge shown in \hyperref[figure:exampleDynamic]{\textit{Figure \ref*{figure:exampleDynamic}}}.

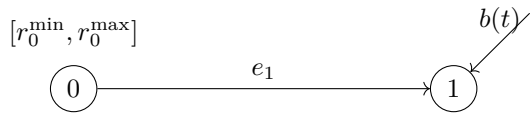
\begin{figure}[htbp]
	\centering
	\begin{tikzpicture}
		\node [minimum size=0.5cm] (A) at (0,0) [circle, draw] {$0$};
		\node [minimum size=0.5cm] (B) at (5,0) [circle, draw] {$1$};
		\node (C) at (0,.75) {$[r_0^{\min}, r_0^{\max}]$};

		\draw[->] (A) to node[above] {$e_1$} (B);
		\draw[->] (6,1) to node[above] {$b(t)$} (B);
	\end{tikzpicture}
	\caption{Example graph with $2$ nodes}
	\label{figure:exampleDynamic}
\end{figure}

For the computation, we use $r_0(x) = 5 \exp \left( \frac{m}{d} (x-L) \right)$ as initial function, $b_D(t) = -2 \sin(2t) + 5$ and $b(t,\omega) = \sum_{m=0}^{\infty} a_m(\omega) a^0_m \psi_m(t) + 5$ as random boundary function. The coefficients $a^0_m$ for the shifted function $b_D(t) - 5$ are given by (\ref{eq:FourierCoefficients}). The initial function is chosen s.t. the solution is constant at $x = 0$ for small times, so we guarantee, that all minimal values are below this constant value and all maximal values are above this constant value (see \hyperref[remark:minmax]{\textit{Remark \ref*{remark:minmax}}}). The other values are given in \hyperref[tab:valuesDynamicExample]{\textit{Table \ref*{tab:valuesDynamicExample}}}.

\begin{table} [htbp]
	\centering
	\begin{tabular}{| c | c | c | c | c | c | c | c |}
		\hline
		$r_0^{\min}$ 	& $r_0^{\max}$ 	& $\mu$	& $\sigma$	& $d$	& $m$	& $L$	& $T$	\\ \hline
		$2$				& $6$			& $1$	& $0.25$	& $-5$	& $-1$ 	& $1$	& $4$	\\ \hline
	\end{tabular}
	\caption{Values for the dynamic example}
	\label{tab:valuesDynamicExample}
\end{table}

Further, we use $101$ points for the time discretization and we cut the Fourier series of the random boundary data after $30$ terms. A sampling of $10$ random boundary functions and the corresponding solutions are shown in \hyperref[figure:randomBoundaryAndSolution]{\textit{Figure \ref*{figure:randomBoundaryAndSolution}}}.

\begin{figure} [htbp]
	\centering
	\begin{subfigure}[c]{7cm}
		\centering
		\includegraphics[width=7cm]{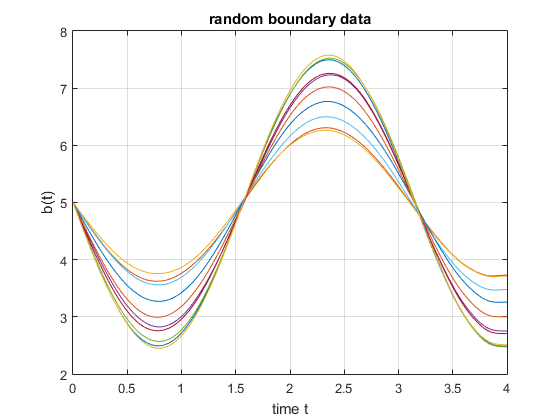}
		\caption{$10$ random boundary functions\\ \phantom{a}}
	\end{subfigure}
	\begin{subfigure}[c]{7cm}
		\centering
		\includegraphics[width=7cm]{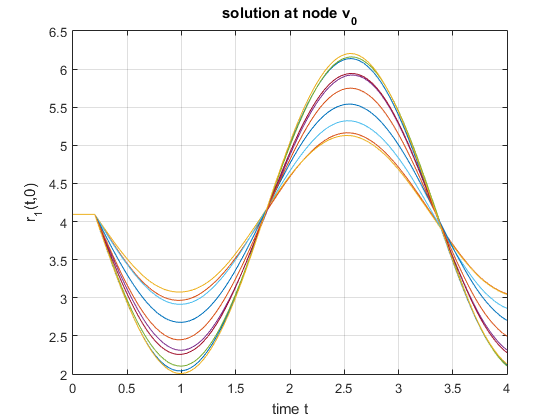}
		\caption{$r_1(t,0)$ for the random boundary data (scaling of $y$-axis is different from (a))}
	\end{subfigure}
	\caption{Sampling of $10$ random boundary functions and the corresponding solutions at node $v_0$}
	\label{figure:randomBoundaryAndSolution}
\end{figure}

For the MATLAB$\textsuperscript{\textregistered}$ implementation, we use a sampling of $1 \cdot 10^5$ boundary functions in terms of Fourier series. We compare the probabilities of the KDE again with a classical Monte Carlo method (MC). The MC method checks for a random boundary function, if the bounds are satisfied for every point of the time discretization. If this is not the case, this boundary function is not feasible. The results of the tests are shown in \hyperref[tab:oneEdgeDynamic]{\textit{Table \ref*{tab:oneEdgeDynamic}}}.

\begin{table} [h]
	\centering
	\begin{tabular}{| c | c | c | c | c | c | c | c | c |}
		\hline
			& Test 1	& Test 2	& Test 3	& Test 4	& Test 5	& Test 6	& Test 7	& Test 8	\\
		\hline
		MC	& $74.32\%$ & $74.39\%$ & $74.22\%$ & $74.32\%$ & $74.24\%$ & $74.33\%$ & $74.51\%$ & $74.40\%$ \\
		\hline
		KDE	& $74.33\%$ & $74.39\%$ & $74.21\%$ & $74.31\%$ & $74.24\%$ & $74.32\%$ & $74.51\%$ & $74.40\%$ \\
		\hline
	\end{tabular}
	\caption{Results for the dynamic example with one edge}
	\label{tab:oneEdgeDynamic}
\end{table}

One can see, that the results of MC and the KDE are almost equal. The mean probability in MC resp. KDE is $74.38\%$ resp. $74.37\%$ and the variance is $0.0141$ resp. $0.0142$. For a confidence level of $95\%$ the confidence interval for the MC probability is $[74.26\%, 74.42\%]$, which is also the confidence interval of the KDE probability. Both methods are still quite fast. MATLAB$\textsuperscript{\textregistered}$ needs much more time ($\sim 1$ minute) for the sampling and computing the random boundary data than for computing the probabilities, which is less then one second. \\

In this subsection, we have considered the SRD and the KDE in a dynamic setting based on the results from \hyperref[sec:stationary]{\textit{Section \ref*{sec:stationary}}}. First the box constraints only hold for a certain point in time $t^* \in [0,T]$ and then, the box constraints hold for the full time period $[0,T]$.

\subsection{Stochastic optimization on dynamic flow networks}

In this subsection, we formulate necessary optimality conditions for the dynamic hyperbolic system introduced before. In the subsection before, we introduced different ways to compute the probability for a random boundary function to be feasible. Using the KDE gives us a good approximation of this probability. Define the set
\begin{equation*}
	\mathcal{R}_0 := [r_0^{\min}, \infty),
\end{equation*}
and the function
\begin{equation*}
	f: \mathbb{R} \rightarrow \mathbb{R}, \quad (r_0^{\max}) \mapsto f(r_0^{\max}).
\end{equation*}
For a probability level $\alpha \in (0,1)$, consider the optimization problem with the approximated probabilistic constraints
\begin{equation} \label{eq:optimizationBoundsDynamic} \left\{ \begin{aligned}
	\min_{r_0^{\max} \in \mathcal{R}_0} \quad & f(r_0^{\max}) \\
	\text{s.t.} \quad &\mathbb{P}_N \left( b \in M(t) \ \forall t \in [0,T] \right) \geq \alpha
\end{aligned} \right. . \end{equation}
Similar to the stationary case, for $k = 1, \cdots, n$ and $i = 1, \cdots, N$, we define
\begin{equation*}
	\varphi_{i,k}^{\min}(x) := \frac{x - \underline{r}_{k}^{S,i}}{\sqrt{2} h_y \sigma_{N,k}^{\min}} \quad \text{and} \quad \varphi_{i,k}^{\max}(x) := \frac{x - \overline{r}_{k}^{S,i}}{\sqrt{2} h_y \sigma_{N,k}^{\max}},
\end{equation*}
where
\begin{equation*}
	\underline{r}_{k}^{S,i} := \min_{t \in [0,T]} r_k(t,0,b^{\mathcal{S},i}) \quad \text{and} \quad \overline{r}_{k}^{S,i} := \max_{t \in [0,T]} r_k(t,0,b^{\mathcal{S},i}),
\end{equation*}
with samples $b^{\mathcal{S},i} \in \mathcal{B}_{\mathcal{A}}$. Then we can rewrite the computation of the desired probability using the error function as
\begin{equation*} \begin{aligned}
	\mathbb{P}_N \left( b \in M(t) \ \forall t \in [0,T] \right) =&
	\int_{[r_0^{\min},r_0^{\max}] \times [r_0^{\min},r_0^{\max}]} \varrho_{R,N}(z)~ dz \\
	 =&	\frac{1}{4N} \sum_{i=1}^N \left[ \erf \left( \varphi_{i,1}^{\min}(r_0^{\max})  \right) - \erf \left( \varphi_{i,1}^{\min}(r_0^{\min}) \right) \right] \\
	& \cdot \left[ \erf \left( \varphi_{i,1}^{\max}(r_0^{\max}) \right) - \erf \left( \varphi_{i,1}^{\max}(r_0^{\min}) \right) \right].
\end{aligned} \end{equation*}

We define the function
\begin{equation*}
	g_\alpha : \mathbb{R} \rightarrow \mathbb{R}, r^{\max} \mapsto \alpha - \mathbb{P}_N( b \in M(t,r_0^{\max}) \ \forall t \in [0,T] ).
\end{equation*}
Difference to the stationary case is, that the terms in both dimensions depend on the same upper bound, thus we have to use product rule to compute the derivative. It follows
\begin{equation*} \begin{aligned}
\frac{d}{d r_0^{\max}} g_\alpha(r_0^{\max}) = - \frac{1}{4 N} \sum_{i=1}^N &\left[ \erf \left( \varphi_{i,1}^{\min}(r_0^{\max}) \right) - \erf \left( \varphi_{i,1}^{\min}(r_0^{\min}) \right) \right] \\
	&\cdot \frac{\sqrt{2}}{\sqrt{\pi} h_y \sigma_{N,1}^{\max}} \exp \left( - (\varphi_{i,1}^{\max}(r_0^{\max}))^2 \right) \\
	+ &\left[ \erf \left( \varphi_{i,1}^{\max}(r_0^{\max}) \right) - \erf \left( \varphi_{i,1}^{\max}(r_0^{\min}) \right) \right] \\
	&\cdot \frac{\sqrt{2}}{\sqrt{\pi} h_y \sigma_{N,1}^{\min}} \exp \left( -(\varphi_{i,1}^{\min}(r_0^{\max}))^2 \right).
\end{aligned} \end{equation*}
As it is mentioned in \hyperref[remark:mfcq]{\textit{Remark \ref*{remark:mfcq}}}, the LICQ is always fulfilled and we can state the necessary optimality conditions for the approximated problem (\ref{eq:optimizationBoundsDynamic}):
\begin{corollary} \label{corollary:dynamicKKT}
Let $r_0^{*,\max} \in \mathbb{R}$ be a (local) optimal solution of (\ref{eq:optimizationBoundsDynamic}). Since the LICQ holds in $r_0^{*,\max}$, there exists a multiplier $\mu^* \geq 0$, s.t.
\begin{equation*} \begin{aligned}
	f'(r_0^{*,\max}) + \mu^*  g'_\alpha(r_0^{*,\max}) &= 0, \\
	g_\alpha(r_0^{*,\max}) &\leq 0, \\
	\mu^* g_\alpha(r_0^{*,\max}) &= 0.
\end{aligned} \end{equation*}
Thus, $(r_0^{*,\max},\mu^*) \in \mathbb{R}^2$ is a Karush-Kuhn-Tucker point.
\end{corollary}

\begin{remark}
In the $n$-dimensional case, in which we have bounds at $n$ nodes, the computation of the desired probability is the following:
\begin{equation*} \begin{aligned}
	\mathbb{P}_N \left( b \in M(t) \ \forall t \in [0,T] \right) = \frac{1}{4^n N} \sum_{i=1}^N \prod_{j=1}^n &\left[ \erf \left( \varphi_{i,j}^{\min}(r_j^{\max}) \right) - \erf \left( \varphi_{i,j}^{\min}(r_j^{\min}) \right) \right] \\
	\cdot &\left[ \erf \left( \varphi_{i,j}^{\max}(r_j^{\max}) \right) - \erf \left( \varphi_{i,j}^{\max}(r_j^{\min}) \right) \right].
\end{aligned} \end{equation*}

Mention that this is an $2n$-dimensional KDE since every boundary function provides two samples, one for the minimal values and one for the maximal values. The partial derivatives with respect to $r_j^{\max}$ are given by
\begin{equation*} \begin{aligned}
	\frac{\partial}{\partial r_j^{\max}} \mathbb{P}_N \left( b \in M(t) \ \forall t \in [0,T] \right) = \\
	= \frac{1}{4^n N} \sum_{i=1}^N \prod_{j=1, j \neq k}^n &\left[ \erf \left( \varphi_{i,j}^{\min}(r_j^{\max}) \right) - \erf \left( \varphi_{i,j}^{\min}(r_j^{\min}) \right) \right] \\
	\cdot &\left[ \erf \left( \varphi_{i,j}^{\max}(r_j^{\max}) \right) - \erf \left( \varphi_{i,j}^{\max}(r_j^{\min}) \right) \right]. \\
	\cdot \bigg[ &\left[ \erf \left( \varphi_{i,k}^{\min}(r_k^{\max}) \right) - \erf \left( \varphi_{i,k}^{\min}(r_k^{\min}) \right) \right] \\
	 & \qquad \cdot \frac{\sqrt{2}}{\sqrt{\pi} h_y \sigma_{N,k}^{\max}} \exp \left( -( \varphi_{i,k}^{\max}(r_k^{\max}) )^2 \right) \\
	+ &\left[ \erf \left( \varphi_{i,k}^{\max}(r_k^{\max}) \right) - \erf \left( \varphi_{i,k}^{\max}(r_k^{\min}) \right) \right] \\
	& \qquad \cdot \frac{\sqrt{2}}{\sqrt{\pi} h_y \sigma_{N,k}^{\min}} \exp \left( -( \varphi_{i,k}^{\min}(r_k^{\max}) )^2 \right) \bigg],
\end{aligned} \end{equation*}
where $(\sigma_{N,k}^{\max})^2$ resp. $(\sigma_{N,k}^{\min})^2$ are the variances of the sampling of the maximal values resp. of the minimal values.
\end{remark}

\subsection{Application to a realistic network}

For this section we basically use the graph of the GasLib-11 but since this graph was designed for gas transportation we slightly vary it. First we assume that the compressor edges are normal edges. As in section \hyperref[sec:GasLib11]{\textit{Section \ref*{sec:GasLib11}}} we assume that the valve is closed, s.t. the edge between node $v_2$ and $v_4$ vanishes. The water is contaminated at the nodes $v_6$, $v_9$ and $v_{10}$ and the pollution distributes in the graph. We assume that the pollution equally distributes at node $v_7$, i.e., half of the pollution distributes in edge $e_6$, the other half in $e_7$. We define pollution bounds $r_0^{\min}, r_0^{\max} \in \mathbb{R}_{\geq 0}$ for node $v_0$, $r_1^{\min}, r_1^{\max} \in \mathbb{R}_{\geq 0}$ for node $v_1$ and $r_5^{\min}, r_5^{\max} \in \mathbb{R}_{\geq 0}$ for node $v_5$. We want these bounds to be satisfied. A scheme of this network is shown in \hyperref[figure:pollutionNetwork]{\textit{Figure \ref*{figure:pollutionNetwork}}}. For simplicity we assume that $m = -0.1$, $d = -1$ and $L = 1$ for every edge.

\begin{figure}[htbp]
	\centering
	\resizebox{\textwidth}{!}{
		\begin{tikzpicture}
			\node [minimum size=0.5cm] (A) at (0,0) [circle, draw] {$0$};
			\node [minimum size=0.5cm] (B) at (2,0) [circle, draw] {$1$};
			\node [minimum size=0.5cm] (C) at (4.5,0) [circle, draw] {$2$};
			\node [minimum size=0.5cm] (D) at (6.5,1) [circle, draw] {$3$};
			\node [minimum size=0.5cm] (E) at (6.5,-1) [circle, draw] {$4$};
			\node [minimum size=0.5cm] (F) at (6.5,-3) [circle, draw] {$5$};
			\node [minimum size=0.5cm] (G) at (6.5,3) [circle, draw] {$6$};
			\node [minimum size=0.5cm] (H) at (8.5,0) [circle, draw] {$7$};
			\node [minimum size=0.5cm] (I) at (11,0) [circle, draw] {$8$};
			\node [minimum size=0.5cm] (J) at (13,1) [circle, draw] {$9$};
			\node [minimum size=0.5cm] (K) at (13,-1) [circle, draw] {$10$};
		
			\draw[->, >={Triangle[length=0pt 2*8,width=0pt 8]}] (A) to node[above] {$e_1$} (B);
			\draw[->, >={Triangle[length=0pt 2*8,width=0pt 8]}] (B) to node[above] {$e_2$} (C);
			\draw[->, >={Triangle[length=0pt 2*8,width=0pt 8]}] (F) to node[right] {$e_3$} (E);
			\draw[->, >={Triangle[length=0pt 2*8,width=0pt 8]}] (C) to node[above] {$e_4$} (D);
			\draw[->, >={Triangle[length=0pt 2*8,width=0pt 8]}] (D) to node[right] {$e_5$} (G);
			\draw[->, >={Triangle[length=0pt 2*8,width=0pt 8]}] (D) to node[above] {$e_6$} (H);
			\draw[->, >={Triangle[length=0pt 2*8,width=0pt 8]}] (E) to node[above] {$e_7$} (H);
			\draw[->, >={Triangle[length=0pt 2*8,width=0pt 8]}] (H) to node[above] {$e_8$} (I);
			\draw[->, >={Triangle[length=0pt 2*8,width=0pt 8]}] (I) to node[above] {$e_9$} (J);
			\draw[->, >={Triangle[length=0pt 2*8,width=0pt 8]}] (I) to node[above] {$e_{10}$} (K);
			\draw[->, >={Triangle[length=0pt 2*8,width=0pt 8]}, dashed] (C) to node[above] {valve} (E);
		
			\node [minimum size=0.5cm] (L) at (0,-1) {};
			\node [minimum size=0.5cm] (M) at (2,-1) {};
			\node [minimum size=0.5cm] (N) at (5.2,-3) {};
			\draw[<-, >={Triangle[length=0pt 1*8,width=0pt 8]}, thick] (L) to (A);
			\draw[<-, >={Triangle[length=0pt 1*8,width=0pt 8]}, thick] (M) to (B);
			\draw[<-, >={Triangle[length=0pt 1*8,width=0pt 8]}, thick] (N) to (F);
			\node [minimum size=0.5cm] (O) at (8.5,3) {Contamination};
			\node [minimum size=0.5cm] (P) at (13,2) {Contamination};
			\node [minimum size=0.5cm] (Q) at (13,-2) {Contamination};
			\draw[<-, >={Triangle[length=0pt 1*8,width=0pt 8]}, thick] (G) to (O);
			\draw[<-, >={Triangle[length=0pt 1*8,width=0pt 8]}, thick] (J) to (P);
			\draw[<-, >={Triangle[length=0pt 1*8,width=0pt 8]}, thick] (K) to (Q);
		\end{tikzpicture}
	}
	\caption{A network scheme for water pollution}
	\label{figure:pollutionNetwork}
\end{figure}
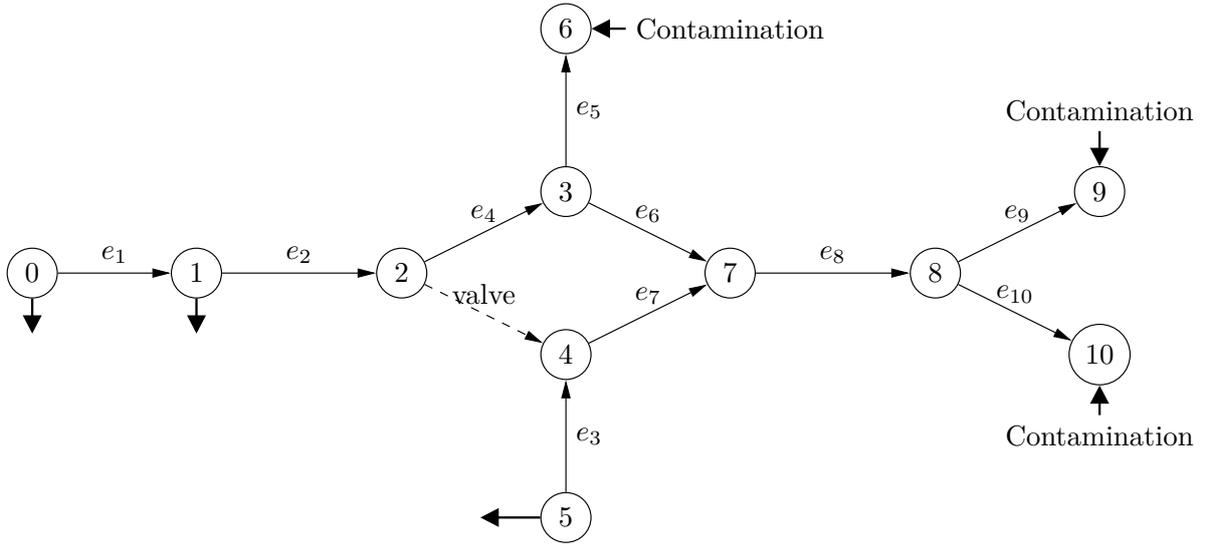

The boundary functions are given by
\begin{equation} \label{eq:deterministicBoundaryFunctions} \begin{aligned}
	b_6(t) 		&= \sin(t)+5, \\
	b_9(t) 		&= \frac{1}{4} \vert t-3 \vert +2, \\
	\text{and} \quad b_{10}(t)	&= \frac{1}{(t-1)^2 + \frac{1}{2}} + 3.
\end{aligned} \end{equation}
The initial conditions are given by
\begin{equation*} \begin{aligned}
	&r_{e_{10},0}(x) = \frac{11}{3} \exp \left( \frac{m}{d} ( x - L ) \right), \\
	&r_{e_9,0}(x) = \frac{11}{4} \exp \left( \frac{m}{d} ( x - L ) \right), \\
	&r_{e_8,0}(x) = \big( r_{e_9,0}(0) + r_{e_{10},0}(0) \big) \exp \left( \frac{m}{d} ( x - L ) \right), \\
	&r_{e_7,0}(x) = \frac{1}{2} r_{e_8,0}(0) \exp \left( \frac{m}{d} ( x - L ) \right), \\
	&r_{e_6,0}(x) = \frac{1}{2} r_{e_8,0}(0) \exp \left( \frac{m}{d} ( x - L ) \right), \\
	&r_{e_5,0}(x) = 5 \exp \left( \frac{m}{d} ( x - L ) \right), \\
	&r_{e_4,0}(x) = \big( r_{e_5,0}(0) + r_{e_6,0}(0) \big) \exp \left( \frac{m}{d} ( x - L ) \right), \\
	&r_{e_3,0}(x) = r_{e_7,0}(0) \exp \left( \frac{m}{d} ( x - L ) \right), \\
	&r_{e_2,0}(x) = r_{e_4,0}(0) \exp \left( \frac{m}{d} ( x - L ) \right), \\
	\text{and} \quad &r_{e_1,0}(x) = r_{e_2,0}(0) \exp \left( \frac{m}{d} ( x - L ) \right).
\end{aligned} \end{equation*}
The initial conditions are chosen, s.t. the solution at the nodes is constant as long as information from the boundary nodes needes to reach the nodes and the initial conditions satisfy the $C^0$-compatibility with the boundary conditions, which is
\begin{equation*}
	r_{e_5,0}(L) = b_6(0), \quad r_{e_9,0}(L) = b_9(0) \quad \text{and} \quad r_{e_{10},0}(L) = b_{10}(0).
\end{equation*}
The boundary functions and the solution at the nodes $v_0$, $v_1$ and $v_5$ are shown in \hyperref[figure:deterministicData]{\textit{Figure \ref*{figure:deterministicData}}}.
\begin{figure}[htbp]
	\centering
	\begin{subfigure}[c]{7.5cm}
		\centering
		\includegraphics[width=7.5cm]{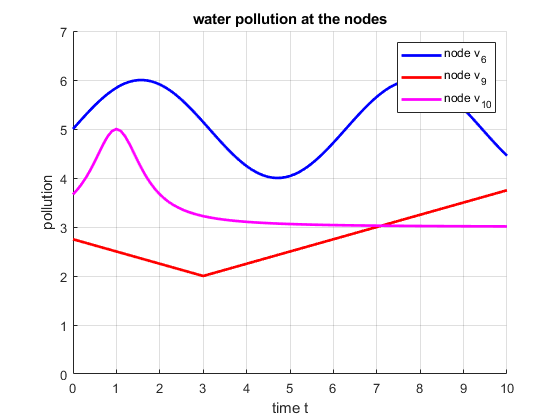}
		\caption{Deterministic boundary data at node $v_6$, $v_9$, $v_{10}$}
	\end{subfigure} \hspace{.5cm}
	\begin{subfigure}[c]{7.5cm}
		\centering
		\includegraphics[width=7.5cm]{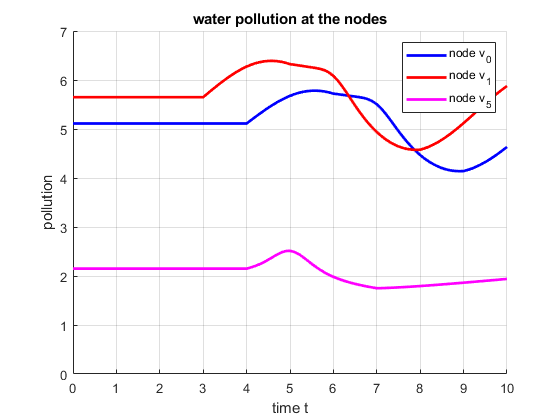}
		\caption{Solution at node $v_0$, $v_1$ and $v_5$}
	\end{subfigure}
	\caption{Deterministic boundary data and deterministic solution}
	\label{figure:deterministicData}
\end{figure}
Since the boundary functions do not satisfy $b_i(0) = 0$ ($i \in \{6,9,10\}$), we compute the Fourier series for the functions $(b_i(t)-b_i(0))$, randomize them by multiplying a Gaussian distributed random number (with $\mu = 1$ and $\sigma^2 = 0.1$) to every summand of the series and then we add the constants $b_i(0)$ to the random Fourier series. For the implementation we use the first $30$ terms of the Fourier series, i.e., $N_F = 30$. Some random scenarios for the boundary functions are shown in  \hyperref[figure:randomScenarios]{\textit{Figure \ref*{figure:randomScenarios}}}.
\begin{figure}[t!]
	\centering
	\begin{subfigure}[c]{7cm}
		\centering
		\includegraphics[width=7cm]{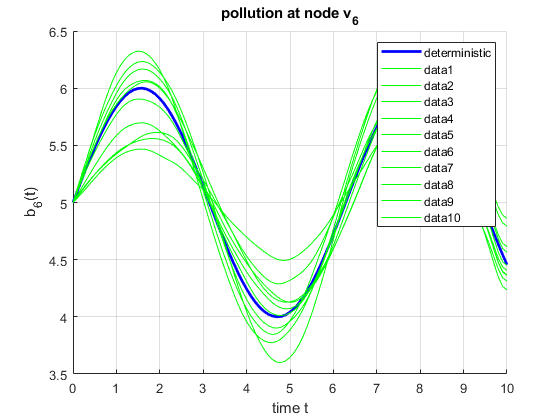}
		\caption{Random scenarios for boundary data $b_6(t)$}
	\end{subfigure} \hspace{.5cm}
	\begin{subfigure}[c]{7cm}
		\centering
		\includegraphics[width=7cm]{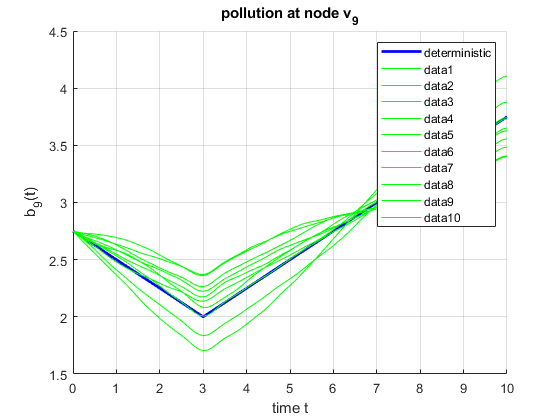}
		\caption{Random scenarios for boundary data $b_9(t)$}
	\end{subfigure} \hspace{.5cm}\\
	\begin{subfigure}[c]{7cm}
		\centering
		\includegraphics[width=7cm]{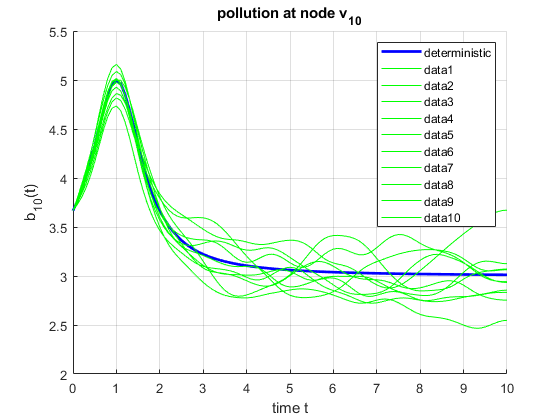}
		\caption{Random scenarios for boundary data $b_{10}(t)$}
	\end{subfigure}
	\caption{Random boundary scenarios at $v_6$, $v_9$ and $v_{10}$}
	\label{figure:randomScenarios}
\end{figure}
The lower pollution bounds are given by
\begin{equation*}
	r^{\min} = \begin{bmatrix} r_0^{\min} \\ r_1^{\min} \\ r_5^{\min} \end{bmatrix} = \begin{bmatrix} 3.5 \\ 4 \\ 1 \end{bmatrix}.
\end{equation*}
Consider the linear function
\begin{equation*}
	f : \mathbb{R}^3 \rightarrow \mathbb{R}, \quad f : r^{\max} \mapsto c^T r^{\max},
\end{equation*}
with $c = \mathbb{1}_3$. We first solve the deterministic problem
\begin{equation} \label{eq:optGasLib11DynDeter} \begin{aligned}
	\min_{r^{\max} \geq r^{\min}} \quad &f(r^{\max}) \\
	\text{s.t.} \quad &b(t) \in M(r^{\max})\ \forall t \in [0,T],
\end{aligned} \end{equation}
with $T = 10$ and the boundary functions are given in (\ref{eq:deterministicBoundaryFunctions}). As in the stationary case we use the default setting of the MATLAB$\textsuperscript{\textregistered}$-routine \textit{fmincon} to solve (\ref{eq:optGasLib11DynDeter}), which is an interior-point algorithm. It returns
\begin{equation*}
	r^{\max}_{\text{det}} = \begin{bmatrix} r_0^{\max} \\ r_1^{\max} \\ r_5^{\max} \end{bmatrix} = \begin{bmatrix} 5.78 \\ 6.39 \\ 2.51 \end{bmatrix},
\end{equation*}
as optimal deterministic solution, i.e., as the lowest upper pollution bound for the nodes $v_0$, $v_1$ and $v_5$. Now we consider uncertain water contamination at the nodes $v_6$, $v_9$ and $v_{10}$. We compute the probability that the random contamination satisfies the bounds $[r^{\min}, r^{\max}_{\text{det}}]$ at the nodes $v_0$, $v_1$ and $v_5$, which is $\mathbb{P}(b(t) \in M(r^{\max}_{\text{det}})\ \forall t \in [0,T])$. This is shown for $8$ tests (each with $1 \cdot 10^5$ samples) in \hyperref[table:GasLib-11dynamicProbability]{\textit{Table \ref*{table:GasLib-11dynamicProbability}}}. \\
\begin{table} [h]
	\centering
	\begin{tabular}{| c | c | c | c | c | c | c | c | c |}
		\hline
			& Test 1	& Test 2	& Test 3	& Test 4	& Test 5	& Test 6	& Test 7	& Test 8	\\
		\hline
		MC	& $37.71\%$	& $37.77\%$	& $37.83\%$	& $37.65\%$	& $37.80\%$	& $37.54\%$	& $37.92\%$	& $38.01\%$	\\
		\hline	
		KDE & $37.57\%$ & $37.62\%$ & $37.68\%$ & $37.51\%$ & $37.67\%$ & $37.39\%$ & $37.76\%$ & $37.88\%$ \\
		\hline
	\end{tabular}
	\caption{Probability $\mathbb{P}(b \in M(p^{\max}_{\text{det}}))$ for the optimal deterministic upper pollution bounds}
	\label{table:GasLib-11dynamicProbability}
\end{table}
The mean MC probability is $37.78\%$ and the mean KDE probability is $37.63\%$. For a confidence level of $95\%$ the confidence interval for the MC probability is $[37.65\%, 37.90\%]$ and the confidence interval for the KDE probability is $[37.51\%, 37.76\%]$. Just like in the stationary case, the deterministic upper bound is unsatisfactory, so we consider the probabilistic constrained optimization problem (\ref{eq:optimizationBoundsDynamic}) with $\alpha := 0.75$. The problems described in \hyperref[sec:GasLib11]{\textit{Section \ref*{sec:GasLib11}}} also occur here but \textit{fmincon} returns a solution if we choose the optimal deterministic solution as starting point. The results of $8$ Tests with $1 \cdot 10^5$ scenarios are shown in \hyperref[table:dynamicGasLib-11Results]{\textit{Table \ref*{table:dynamicGasLib-11Results}}}. This time we show $3$ decimal places because otherwise the solutions would be equal. In $8$ more Tests we solve (\ref{eq:optimizationBoundsDynamic}) by using \hyperref[corollary:dynamicKKT]{\textit{Corollary \ref*{corollary:dynamicKKT}}}. The results vary from the optimal solutions computed by \textit{fmincon} in a range of $1 \cdot 10^{-7}$. So here one could also expect that the necessary optimality conditions are sufficient but we do not analyze this here. \\
\begin{table} [h]
	\centering
	\begin{tabular}{| c | c | c | c | c | c | c | c |}
		\hline
		Test 1	& Test 2	& Test 3	& Test 4	& Test 5	& Test 6	& Test 7	& Test 8	\\
		\hline
		$\begin{bmatrix} 5.936 \\ 6.560 \\ 2.577 \end{bmatrix}$	& $\begin{bmatrix} 5.938 \\ 6.562 \\ 2.574 \end{bmatrix}$	& $\begin{bmatrix} 5.937 \\ 6.561 \\ 2.576 \end{bmatrix}$	& $\begin{bmatrix} 5.935 \\ 6.559 \\ 2.578 \end{bmatrix}$	& $\begin{bmatrix} 5.936 \\ 6.560 \\ 2.575 \end{bmatrix}$	& $\begin{bmatrix} 5.935 \\ 6.559 \\ 2.576 \end{bmatrix}$	& $\begin{bmatrix} 5.937 \\ 6.561 \\ 2.577 \end{bmatrix}$	& $\begin{bmatrix} 5.936 \\ 6.560 \\ 2.575 \end{bmatrix}$	\\
		\hline		
	\end{tabular}
	\caption{Stochastic optimal upper pollution bounds $r^{\max}_{\text{stoch}}$}
	\label{table:dynamicGasLib-11Results}
\end{table}
One can see, that all results are almost equal. The optimal upper pollution bounds of the probabilistic constrained optimization problem (\ref{eq:optimizationBoundsDynamic}) are larger then the optimal upper pollution bounds of the deterministic optimization problem (\ref{eq:optGasLib11DynDeter}). The computation time is quite similar to the stationary case with the difference that solving the necessary optimality conditions needed much more time here (about $2$ minutes per test).

\section{Conclusion}

In this paper, we have shown two different ways to evaluate probabilistic constraints in the context of hyperbolic balance laws on graphs for both, a stationary and a dynamic setting with box constraints for the solution. \\

The spheric radial decomposition provides a good method for computing the probabilities for random boundary data to be feasible in the stationary case and in the dynamic case for box constraints for a certain point in time. Because the spheric radial decomposition is explicitly based on the analytical solution, it leads to good results. But as soon as the analytical solution is not given, the inequalities cannot be derived and so the spheric radial decomposition becomes an almost purely numerical method. \\

A kernel density estimator does not need the analytical solution, a numerically computed solution is sufficient, and it provides an estimated, but explicit representation of the probability density function, which can be used for computing the probabilistic constraint and for deriving necessary optimality conditions for probabilistic constrained optimization problems. In addition the kernel density estimator is smooth even if the exact probability density function is nonsmooth. The examples showed, that the kernel density estimator provides results almost as good as the results from the spheric radial decomposition and a classical Monte Carlo approach. Of course, the Monte Carlo approach is faster and easier to use, but it cannot be used to get any kind of analytical result like e.g. the necessary optimality conditions. \\

Another advantage of both methods is, that they do not depend on the topology of the graph. The topology of the graph influences the analytical solution, but the methods themselves work independently of the complexity of the graph. Further, the idea of the kernel density estimator can easily be used for any kind of partial differential equation with random boundary data, because it only requires a sufficiently accurate solution of the PDE, e.g., by using numerical methods, and this has been a goal of many papers in the last decades.


\subsubsection*{Conflict of Interests}

The authors declare that there is no conflict of interest regarding the publication of this paper.

\subsubsection*{Acknowledgements}

The authors are supported by the Deutsche Forschungsgemeinschaft (DFG, German Research Foundation) within the collaborative research center TRR154 \grqq Mathematical modeling, simulation and optimisation using the example of gas networks\grqq\ (Project-ID 239904186, TRR154/2-2018, TP B01 (Lang, Strauch), TP C03 (Gugat, Schuster) and TP C05 (Giesselmann, Gugat)).


\small

\bibliographystyle{natdin}

\bibliography{bibliography}
\addcontentsline{toc}{section}{References}

\end{document}